\documentclass[12pt]{article}
\usepackage{amsmath, amssymb, amsthm}
\usepackage[margin=1in]{geometry}
\usepackage{mmacells}

\newcommand{\R}{\mathbb{R}}

\theoremstyle{plain}
\newtheorem{thm}{Theorem}[section]
\newtheorem*{thm*}{Theorem}
\newtheorem{lemma}{Lemma}[section]
\newtheorem*{lemma*}{Lemma}
\newtheorem{prop}{Proposition}[section]
\newtheorem*{prop*}{Proposition}
\newtheorem{cor}{Corollary}[section]
\newtheorem*{cor*}{Corollary}

\theoremstyle{definition}
\newtheorem{defn}{Definition}[section]
\newtheorem*{defn*}{Definition}
\newtheorem{conj}{Conjecture}[section]
\newtheorem{exmp}{Example}[section]

\theoremstyle{remark}
\newtheorem*{rem}{Remark}

\author{Dami Lee}
\title{Geometric realizations of cyclically branched coverings over punctured spheres}
\date{}

\begin{document}

\maketitle
\begin{abstract}
In classical differential geometry, a central question has been whether abstract surfaces with given geometric features can be realized as surfaces in Euclidean space. Inspired by the rich theory of embedded triply periodic minimal surfaces, we seek examples of triply periodic polyhedral surfaces that have an identifiable conformal structure. In particular, we are interested in explicit cone metrics on compact Riemann surfaces that have a realization as the quotient of a triply periodic polyhedral surface. This is important as Riemann surfaces where one has equivalent descriptions are rare. We construct periodic surfaces using graph theory as an attempt to make Schoen’s heuristic concept of a dual graph rigorous. We then apply the theory of cyclically branched coverings to identify the conformal type of such surfaces.
\end{abstract}

\tableofcontents


\section{Introduction}
In classical differential geometry, a central question has been whether abstract surfaces with given geometric features can be realized as surfaces in Euclidean space. Early results include Hilbert's proof that no complete surface with constant negative curvature can be immersed isometrically in Euclidean space \cite{hilbert1901uber} and the theorem of Hartmann-Nirenberg that a complete hypersurface of constant zero curvature in $(n + 1)$-Euclidean space is an $(n - 1)$-cylinder \cite{hartmann1959spherical}. In this paper, we are particularly interested in the constraints that the conformal structure of a surface (along with other geometric features) puts on possible realizations of the surface in Euclidean space. This question has been well studied in cases where the connection between surface geometry and conformal type is particularly strong for example, in minimal and constant mean curvature surfaces. Here we address this problem for another class of surfaces, surfaces with cone metrics. Two fundamental results are the following from \cite{troyanov1991prescribing} and \cite{alexandrov2005convex} respectively. 

\begin{thm*} [Schwarz-Christoffel; Troyanov] Let $X$ be a compact Riemann surface and $p_1,$ $\ldots ,$ $p_n$ be finitely many points on $X.$ Let $\theta_1, \ldots , \theta_n$ be positive numbers so that $-2 \pi \chi(X) = \sum\limits_{i = 1}^n (\theta_i - 2 \pi),$ then there exists a conformal flat metric on $X$ with cone angles $\theta_i$ at $p_i$ for each $i.$ The metric is unique up to homothety. 
\end{thm*}

\begin{thm*} [Aleksandrov] Any cone metric of positive curvature on the 2-sphere can be realized by the boundary of a convex body in Euclidean space.
\end{thm*}

Little is known in other cases. Inspired by the rich theory of embedded triply periodic minimal surfaces, we seek examples of triply periodic polyhedral surfaces in Euclidean space that have an identifiable conformal structure. In particular, we are interested in explicit cone metrics on compact Riemann surfaces that have a realization as the quotient of a triply periodic polyhedral surface in Euclidean space. Via the cone metrics we can derive holomorphic 1-forms, algebraic equations, and a hyperbolic structure. This is important because Riemann surfaces with equivalent descriptions are rare. They can be used to construct examples in areas including billiards, hyperbolic length spectrum, and minimal surfaces. The non-periodic case of genus zero is answered by Troyanov and Aleksandrov to some extent. \\

As a way to obtain cases for higher genera, we devise a construction method for triply periodic polyhedral surfaces that are potentially highly symmetric. In \cite{lee2017triply}, Lee investigates an example of a triply periodic polyhedral surface whose vertices are Weierstrass points. In Section~\ref{ch2 octa-4}, we summarize the paper and discuss the surface's geometric construction in $\R^3$ and its abstract quotient surface. Due to its many symmetries on the underlying surface, we can identify its hyperbolic structure. Specifically, it is identified as an eightfold cyclically branched cover over a thrice punctured sphere. As a result, we carry out explicit computations regarding cone metrics, basis of holomorphic 1-forms, automorphisms, algebraic descriptions, etc. Furthermore, we reach the following theorem.

\begin{thm*} [\cite{lee2017triply}] The conformal structure on the underlying surface of the Octa-4 is conformally equivalent to the Fermat's quartic.
\end{thm*}

This motivates us to investigate other cyclically branched covers over punctured spheres. In Section~\ref{ch3 cyc br cov}, we study cyclically branched covers over punctured spheres in detail. Interesting results include the construction of cyclically branched covers over punctured spheres and a finiteness theorem (Theorem~\ref{thm: genera}).\\

Influenced by the theory of triply periodic minimal surfaces, we are also interested in geometric realizations of such coverings. In Section~\ref{ch4 reg tpps}, we use graph theory to develop construction methods on building triply periodic polyhedral surfaces. We define a decoration (Definition~\ref{defn: deco}) of a graph as a polyhedron that is homotopy equivalent to the graph. This is an attempt to make Schoen's heuristic concept of a dual graph rigorous (\cite{schoen2012reflections}). This construction method enables us to expand Coxeter-Petrie's classification of infinite regular polyhedral surfaces in \cite{coxeter1937regular}. 

\begin{defn} \label{defn: reg} A polyhedral surface is \textit{regular} if it is tiled by regular $p$-gons with $q$ of them at each vertex.
\end{defn}

A feature that was included in Coxeter-Petrie's classification was to ``let the polygonal faces go up and down in a zig-zag formation''. As a result, they introduced three triply periodic regular polyhedra named the Mucube, Muoctahedron, and Mutetrahedron. On these polyhedra, there exist two types of symmetry: one that permutes the vertices of a face, and another that permutes the faces that meet at a vertex. The surfaces divide Euclidean space into an ``inside'' and an ``outside'' compartment and the two symmetries mentioned above interchange the inside and outside. It is proved in \cite{coxeter1937regular} that these three are the only possible cases. Fairly similar is the Octa-4 surface as it is a regular polyhedral surface tiled by regular triangles, eight at each vertex. However, the Octa-4 surface is not included in Coxeter-Petrie's classification as it does not carry the zig-zag formation of the triangles. Because of this, it lacks the Euclidean symmetries that interchange the inside and outside compartments. Nevertheless, the compact quotient of all four surfaces are all genus three surfaces whose corresponding conformally equivalent hyperbolic surfaces are highly symmetric. We achieve a hyperbolic tessellation for each surface by mapping the Euclidean polygons to hyperbolic polygons, and the two symmetries mentioned above generate a group that acts transitively on the hyperbolic tessellations. With the many symmetries on the hyperbolic surfaces, we take the quotient of the surfaces by their symmetries. With symmetries, specifically rotations, the quotients become spheres. Hence, we say that the surfaces are branched coverings over spheres. We combine the theory of cyclically branched coverings over punctured spheres to identify the conformal type of each surface. The tools that we use include flat structures on Riemann surfaces, hyperbolic geometry, and algebraic geometry.\\

The contents of this paper is as follows: In Section~\ref{ch2 octa-4}, we look into the Octa-4 surface as our leading example. We build a triply periodic polyhedron by regular octahedra, whose boundary yields the Octa-4 surface. We identify the hyperbolic structure of the underlying surface as an eightfold cyclically branched cover over a thrice punctured sphere. Then, we identify the automorphism group of the underlying surface. In Section~\ref{ch3 cyc br cov}, we develop the theory of cyclically branched covers. We study the topological construction of such abstract surfaces and prove that there can be finitely many such coverings with a given genus. Then, we study maps between punctured spheres and their lifts on their coverings. In addition, we look at lifts of cone metrics that induce holomorphic 1-forms on their coverings and lastly we study Wronski metrics to find Weierstrass points on the underlying surfaces. We will also be able to locate the Weierstrass points on the polyhedral surfaces in Euclidean space. In Section~\ref{ch4 reg tpps}, we broaden the classification of infinite regular polyhedral surfaces to find more examples alike the Octa-4. We discuss triply periodic symmetric graphs of lower genera and formulate a decoration of a graph as a method to construct a triply periodic polyhedral surface from a triply periodic graph. We end the section with classification theorems, Theorem~\ref{thm: finite decorations3} and Theorem~\ref{thm: finite decorations4}, that find all genus three and four triply periodic regular polyhedral surfaces that arise as decorations of graphs. Finally, in Section~\ref{ch5 main result}, we will study examples from Theorem~\ref{thm: finite decorations3} and Theorem~\ref{thm: finite decorations4} whose cone metrics on their underlying surfaces have realizations as cyclically branched covers over punctured spheres. Results include examples that shed new light on existing minimal or algebraic surfaces, such as the Schwarz minimal P-, D-surfaces, Fermat's quartic, Schoen's minimal I-WP surface (Theorem~\ref{thm: iwp}), and Bring's curve (\cite{weber2005kepler} and Theorem~\ref{thm: truncated octa8}). \\

The author would like to thank Matthias Weber, for his insight and support on this project. In addition, the author would like to thank Bruce Solomon, Kevin Pilgrim, Matt Bainbridge, and Dylan Thurston, for their advice and feedback while completing this paper. 

\section{A Regular Triply Periodic Polyhedral Surface}
\label{ch2 octa-4} 
Here we summarize \cite{lee2017triply} for the readers' convenience. This surface arises as the boundary of a triply periodic polyhedron achieved by gluing regular octahedra in $\R^3$ periodically. We denote the polyhedral surface by $\Pi$ and show that it has no self-intersection. We also show that the surface is triply periodic and that its compact quotient is a genus three Riemann surface. We denote the quotient surface by $X.$ Due to the many symmetries of $X,$ we identify its conformal structure as an eightfold cyclically branched covering over a thrice-punctured sphere and find cone metrics on $X$ that are induced from $\Pi.$ Later, we will study cone metrics that induce holomorphic 1-forms on $X,$ then find an algebraic description of $X.$ We show that the surface is non-hyperelliptic and particularly that there is no triply periodic minimal surface whose underlying structure is conformally equivalent to Octa-4 (\cite{meeks1990theory}). Furthermore, we will show that the algebraic equation of $X$ represents Fermat's quartic. Lastly, we will find the automorphism group that acts transitively on $X.$\\

First, we construct a triply periodic polyhedron by gluing regular octahedra in a periodic manner. We begin with regular octahedra, all of the same size. We label them as either a Type A or a Type B octahedron. We choose a pair of opposite faces on each Type A octahedron, and four faces that are pairwise non-adjacent on each Type B octahedron. We glue the octahedra to each other while alternating the types, allowing gluing only along the chosen faces. That is, we glue Type A octahedra to four non-adjacent faces of a Type B octahedron and Type B octahedra to a pair of opposite faces of a Type A octahedron (Figure \ref{octa4-5}). We name the boundary of this polyhedron the Octa-4 surface and denote it by $\Pi.$ Its name is due to the number of Type A octahedra we attach to each Type B octahedron (Definition \ref{defn: genus graph}). We also point out that $\Pi$ is a regular polyhedral surface tiled by regular triangles, eight at each vertex.

\begin{figure}[htbp] 
\centering
\begin{minipage}{.5\textwidth}
	\centering
	\includegraphics[width=2.5in]{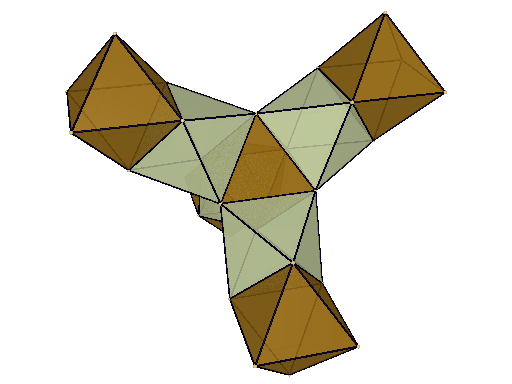}
\end{minipage}%
\begin{minipage}{.5\textwidth}
	\centering
	\includegraphics[width=2.5in]{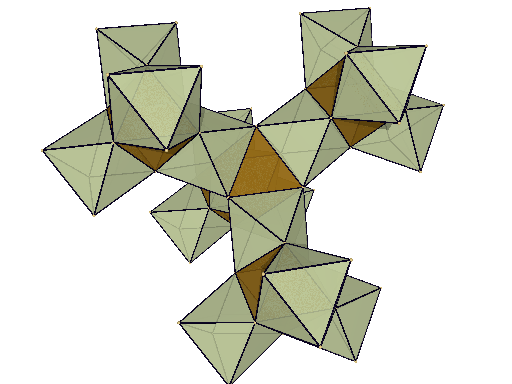}
\end{minipage}
	\caption{Geometric construction of $\Pi$}
	\label{octa4-5}
\end{figure}

\begin{thm*} [\cite{lee2017triply}]  The Octa-4 surface has no self-intersection. Furthermore, it is invariant under three independent translations in $\R^3.$
\end{thm*}

Next we look into the abstract genus three Riemann surface $X := \Pi / \Gamma,$ where $\Gamma$ is a rank-three lattice in $\R^3.$ We will put a conformal structure on $X$ induced by the polyhedral structure on $\Pi.$ This allows us to look at not only Euclidean symmetries on $\Pi$ but also hyperbolic symmetries on $X.$ This allows us to find translational structures on $X,$ which are geometric representations of holomorphic 1-forms induced from cone metrics.\\

As all triangles on $\Pi$ are regular and all vertices on $\Pi$ are octavalent, we map a Euclidean triangle in $\Pi$ to a hyperbolic $(\frac{\pi}{4}, \frac{\pi}{4}, \frac{\pi}{4})$-triangle in the hyperbolic disk. Then, by the Schwarz reflection principle, we get a hyperbolic tessellation of the hyperbolic disk. As the fundamental piece consists of 32 triangles that come from four Type A octahedra and two Type B octahedra, the hyperbolic 16-gon bounded by bold lines in Figure~\ref{hyperbolic_tiling} indicates the hyperbolic fundamental piece of $X.$ To pin down the identification of edges of the 16-gon, we use the definition of Petrie polygons from \cite{alexandrov_intrinsic}.

\begin{defn}\label{defn: petrie} A Petrie polygon of a regular tiling is an infinite regular skew polygon that turns alternately to the left and right. 
\end{defn}

Figure~\ref{fig: petrie} shows that a Petrie polygon turns alternately to the left and right when it meets a vertex on the square, hexagon, triangular tiling respectively. 

\begin{figure}[htbp] 
\centering
\begin{minipage}{.5\textwidth}
	\centering
	\includegraphics[width=2.5in]{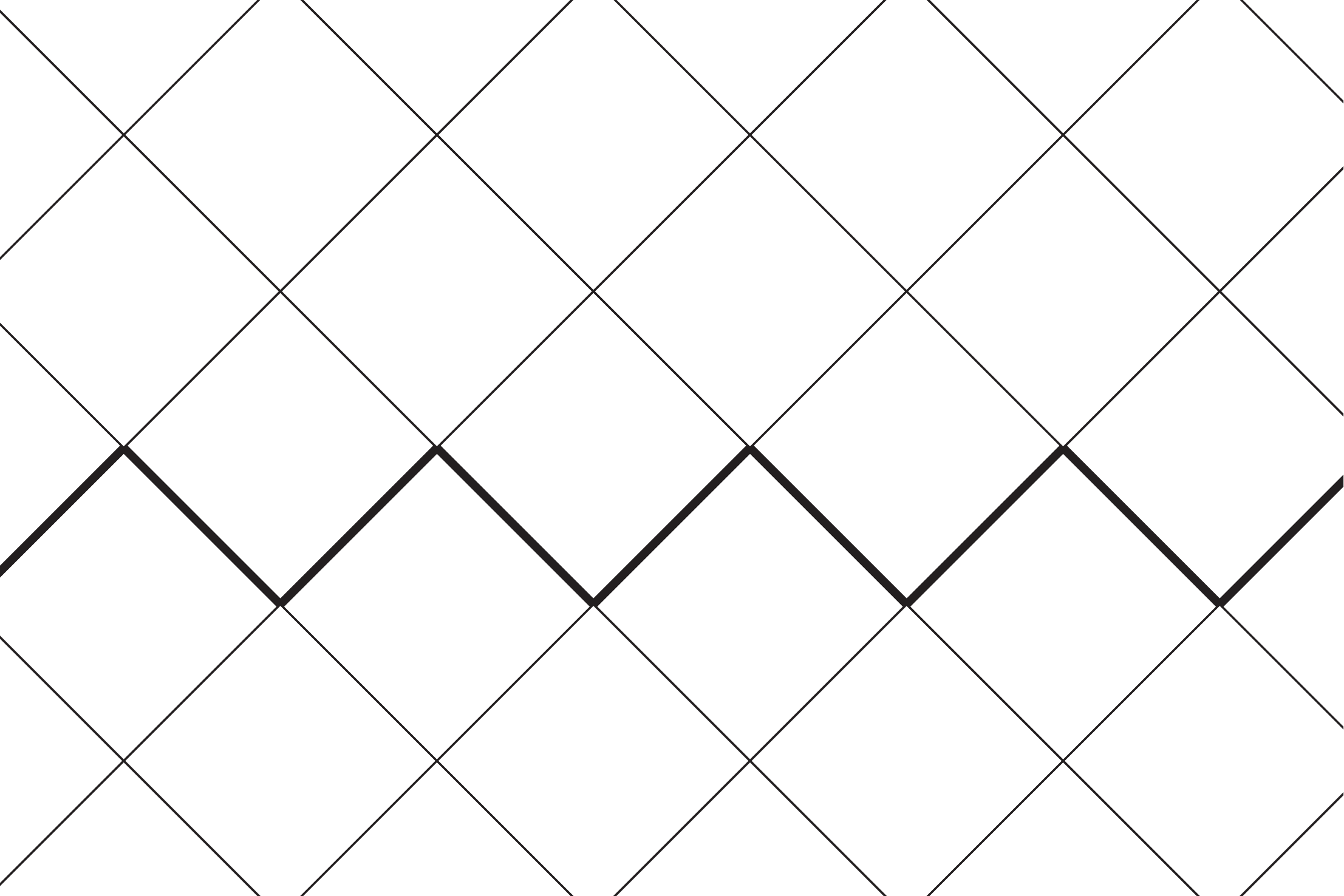}
\end{minipage}%
\begin{minipage}{.5\textwidth}
	\centering
	\includegraphics[width=2.5in]{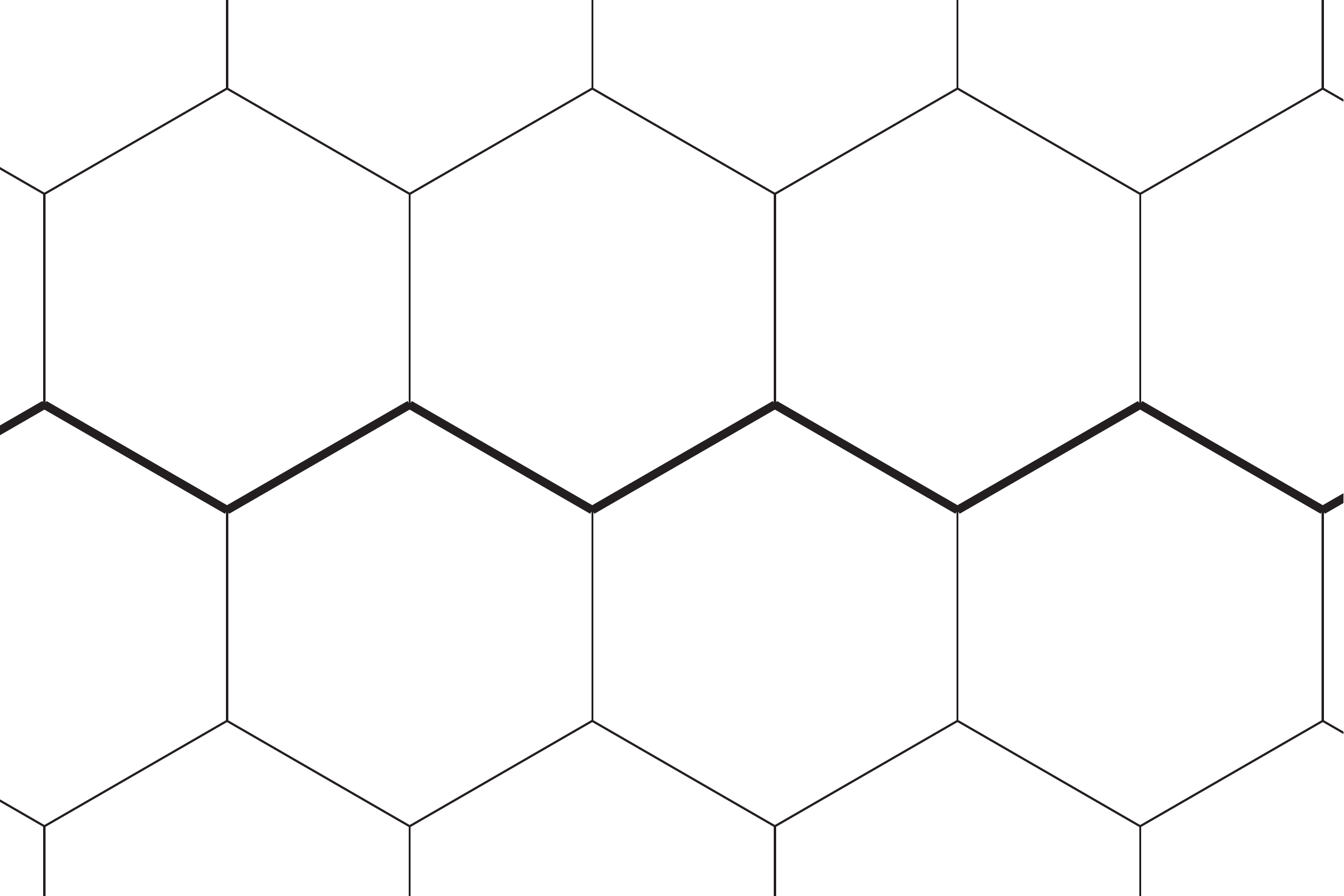}
\end{minipage}
	\caption{Examples of Petrie polygons}
	\label{fig: petrie}
\end{figure}

On $\Pi,$ we find Petrie polygons that go through the midpoints of edges instead of the vertices. Their images on the hyperbolic disk via the previous mapping also correspond to hyperbolic geodesics on $X.$ In either case, all Petrie polygons become closed after passing through six triangles. This shows that translations along the Petrie polygons on $X$ can be used to find the identification of the hyperbolic 16-gon, which are marked as dotted lines in Figure~\ref{hyperbolic_tiling}. It also shows that there is an order-eight rotation about any vertex on $X$ that preserves the identification of edges. This rotation is not induced from any Euclidean symmetry on $\Pi.$ 

\begin{figure}[htbp] 
   \centering
   \includegraphics[width=2.5in]{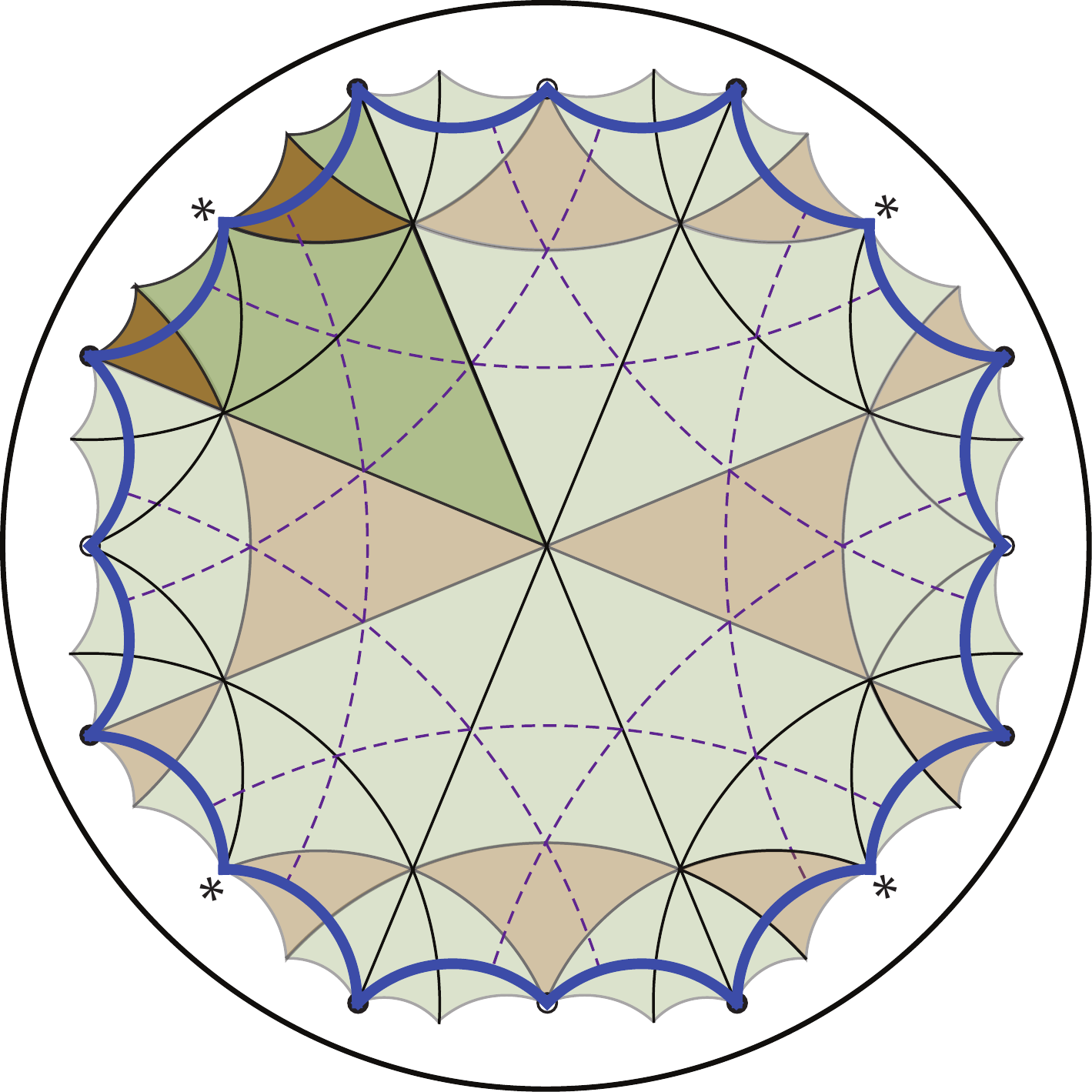} 
   \caption{Hyperbolic description of $X$}
   \label{hyperbolic_tiling}
\end{figure}

\begin{defn} \label{defn: cyc br cov} A branched covering $X \rightarrow Y$ is called a cyclically branched covering if $Y = X / (\mathbb{Z} / n \mathbb{Z})$ for some $n \in \mathbb{Z}.$
\end{defn}

The following theorem was shown as a remark in \cite{lee2017triply}.

\begin{thm*} [\cite{lee2017triply}]  
\label{thm: octa4_cycbr} $X$ is an eightfold cyclically branched cover over a thrice punctured sphere. 
\end{thm*}

\begin{proof} 
We use the Riemann-Hurwitz formula. We label the rotation about the center of the tessellation by $a.$ The center of the tessellation is fixed, hence the branching order is 7. Another fixed point is marked as $\bullet$ on Figure \ref{hyperbolic_tiling}. There are also two points that have orbits of length four, marked as $\ast$ and $\circ.$ The branching order is three at the two points. By the Riemann-Hurwitz formula, we have $$2 - 2 \cdot 3 = 8 (2 - 2 g_q) - \sum (7 + 3 + 3 + 7)$$ where $g_q$ is the genus of the quotient surface $X / \langle a \rangle.$ Then $g_q = 0.$ Since $\langle a \rangle$ is cyclic, we say that $X$ is a cyclically branched cover over a sphere. 
\end{proof}

The previous theorem shows an order-eight symmetry of $X.$ The following theorem finds all hyperbolic isometries of $X.$

\begin{thm*} [\cite{lee2017triply}] Let $a$ be the counterclockwise order-eight rotation that fixes a vertex of the hyperbolic tessellation on $X$ and $b$ be the counterclockwise order-three rotation that fixes a hyperbolic triangle on $X.$ Then, the group of orientation preserving isometries is generated by $a$ and $b.$ Its presentation is $\textrm{Aut}(X) = \langle a, b \mid a^8 = b^3 = (a b)^2 = (a^2 b^2)^3 = (a^4 b^2)^3 = 1 \rangle.$ Specifically, $|\textrm{Aut}(X)| = 96.$
\end{thm*}

\section{Cyclically Branched Covers Over Spheres}
\label{ch3 cyc br cov} 
In the previous section, we looked at an example of a regular triply periodic polyhedral surface. Due to its many symmetries on the underlying surface, we were able to identify its hyperbolic structure as a cyclically branched cover over a thrice-punctured sphere. This motivates us to find other cyclically branched covers over punctured spheres which have the same structure as compact quotients of triply periodic polyhedral surfaces. First we begin with the abstract construction and define a cyclically branched covering over a punctured sphere. We show that such a covering is uniquely defined up to homeomorphism. We also prove a finiteness theorem given the genus of the covering and the number of punctures on the 2-spheres. Then we study automorphisms on cyclically branched coverings by studying maps on punctured spheres and looking at their lifts. Next we seek a basis of holomorphic 1-forms on the coverings that arise from cone metrics on the quotient surfaces. We find holomorphic 1-forms on the coverings, as pullbacks of cone metrics on the quotients. In particular, this produces a basis of holomorphic 1-forms on the covering. Lastly, we search for Weierstrass points on the surfaces. They arise as points where the dimension count of the Riemann-Roch theorem is not generic. Moreover, they are permuted by automorphisms hence allow one to distinguish Riemann surfaces. Without the automorphisms, one can rarely locate all Weierstrass points on a given surface. However, we look into the Wronski metric that reveals their location. 

\subsection{Construction of cyclically branched covers over spheres}
\label{sec: construction}
In this section, we will demonstrate how to construct a $d$-fold cyclically branched covering over an $n$-punctured sphere. We will show that these covering surfaces are uniquely defined up to homeomorphism. We will also prove a finiteness theorem that says given $g,$ there only finitely many surfaces $X$ so that $X$ is a cyclically branched covering over a punctured sphere and $\textrm{genus}(X) = g.$\\

Let $p_1, \ldots , p_n$ be $n$ distinct points on a 2-sphere and let $q \in \mathbb{S}^2 \setminus \{p_1, \ldots , p_n\}$ be a base point. We denote $\mathbb{S}^2 \setminus \{p_1, \ldots , p_n\}$ by $Y.$ We take branch cuts $\gamma_i$ to be simple curves from $q$ to $p_i$ so that $\gamma_i$ are mutually disjoint. For each $i,$ we label the left side of $\gamma_i$ by $\gamma_i^l$ and the right side of $\gamma_i$ by $\gamma_i^r.$ To each $p_i,$ we assign a positive integer $d_i$ which we call the \textit{branching index.} The branching index indicates how the branching occurs around $p_i.$ Denote the degree of the covering map by $d$ so that we have sheets $Y_1,$ $\ldots ,$ $Y_d.$ Without loss of generality, we consider $Y_1$ and $p_1.$ We identify $\gamma_1^l$ of $Y_1$ with $\gamma_1^r$ of $Y_{1 + d_1 \pmod d}.$ Similarly, we identify $\gamma_1^l$ of $Y_{1 + d_1 \pmod d}$ with $\gamma_1^r$ of $Y_{1 + 2 d_1 \pmod d},$ and so on. Likewise, for any $i$ and $j$ we identify $\gamma_i^l$ of $Y_j$ with $\gamma_i^r$ of $Y_{j + d_i \pmod d}.$ It follows from the construction that we need to only consider $d_i \in \{1, \ldots , d - 1\}.$ In the following theorem, we show a necessary and sufficient condition on the branching indices for the covering to be well-defined. We also show that the covering surface is branched only at $p_i.$

\begin{thm} \label{thm: wdefd cyc br cov} Given an $n$-punctured sphere and branching indices $(d_1, \ldots , d_n),$ a $d$-fold cyclically branched cover over the $n$-punctured sphere is a closed surface if and only if $\sum\limits_{i=1}^n d_i \equiv 0 \pmod d.$ 
\end{thm}

\begin{proof} Let $Y$ be an $n$-punctured sphere and let $X \rightarrow Y$ be a $d$-fold covering. For the covering to be cyclic (Definition \ref{defn: cyc br cov}), we must show that there exists an order-$d$ cyclic map that preserves $X.$ As $Y_j$ was picked arbitrarily in the construction of the covering, a map that sends $Y_j$ to  $Y_{j + 1}$ is such a map. \\

Then we check the lifting properties. A positively oriented (counterclockwise) simple closed curve around $p_1 \in Y_j$ (and no other $p_i$ on $Y_j$) lifts to a closed curve that passes through $Y_j \rightarrow Y_{j + d_1 \pmod d} \rightarrow Y_{j + 2 d_1 \pmod d} \rightarrow \cdots \rightarrow Y_j.$ Regardless of the value of $d_1,$ we are back to $Y_j$ after $d$ steps. Lastly, we show that the covering map is not branched at any other point. Let $\gamma$ be a positively oriented simple closed curve around $q$ on $Y_j$ so that the intersection numbers $\iota(\gamma_i, \gamma) = 1$ for all $i.$ Since $\gamma$ crosses each branch cut exactly once, its lift is closed on the covering if and only if $\sum\limits_{i=1}^n d_i \equiv 0 \pmod d.$ 
\end{proof}

We denote the covering by $X.$ Now, we show that $X$ is uniquely defined up to homeomorphism and independent of $p_i,$ $q,$ or the branch cuts.

\begin{lemma} Given a $d$-fold cyclically branched covering over an $n$-punctured sphere with branching indices $(d_1, \ldots , d_n),$ a covering is uniquely defined up to homeomorphism. That is, the construction of $X$ is independent of $p_i,$ $q,$ or the branch cuts.
\end{lemma}

\begin{proof}
Let $Y = \mathbb{S}^2 \setminus \{p_1, \ldots , p_n\}$ and $Y' = \mathbb{S}^2 \setminus \{p_1', \ldots , p_n'\}$ both be $n$-punctured spheres. Let $X$ be a $d$-fold cyclically branched cover over $Y$ defined by $(d_1, \ldots , d_n),$ and $X'$ be a $d$-fold cyclically branched cover over $Y'$ with the same branching indices. Since $p_i$ and $p_i'$ have the same branching indices and $Y$ is homeomorphic to $Y',$ their coverings are homeomorphic to each other.
\end{proof}

Now we give a criterion on the branching indices so that the coverings constructed by the given branching indices are connected.

\begin{lemma} Let $X$ be a $d$-fold cyclically branched covering over an $n$-punctured sphere, defined by branching indices $(d_1, \ldots , d_n).$ Then $X$ is connected if and only if $\gcd(d_1, \ldots , d_n) = 1.$
\end{lemma}

\begin{proof} We refer to the lifting properties in the proof of Theorem \ref{thm: wdefd cyc br cov}. Without loss of generality, let $i = 1$ and assume that $\gcd(d, d_1) = 1.$ Then, a positively oriented simple closed curve around $p_1 \in Y_j$ (and no other $p_i$ on $Y_j$ for some $j$) lifts to a closed curve that passes through all $Y_j.$ That is, there exists one preimage of $p_i$ on $X.$ However, if $\gcd(d, d_1) \neq 1,$ say $\gcd(d, d_1) = 2,$ then a positively oriented simple closed curve lifts to a closed curve that passes through only $Y_1, Y_3, \ldots , Y_{d -1}$ or only $Y_2 ,Y_4, \ldots , Y_d.$ In other words, there exist $\gcd(d, d_i)$ preimages of $p_i$ on $X.$ Hence, if $\gcd(d_1, \ldots , d_n) \neq 1,$ then $X$ is a disconnected surface that has $\gcd(d_1, \ldots , d_n)$ components.
\end{proof}

Now that the covering is topologically well-defined, we compute the genus of the covering. 

\begin{prop}
\label{prop: g(X)}
Let $X$ be a $d$-fold cyclically branched cover over an $n$-punctured sphere defined by branching indices $(d_1, \ldots , d_n).$ Then $$\textrm{genus}(X) = \frac{d (n - 2)}{2} + 1 - \frac{1}{2} \sum \gcd(d, d_i).$$ Specifically, if $\gcd(d, d_i) = 1$ for all $i,$ then $\textrm{genus}(X) = (\frac{n}{2} - 1) (d - 1).$
\end{prop}

\begin{proof}
We will find all branched points and their branching orders, then apply the Riemann-Hurwitz formula to find the genus of the covering. If $\gcd(d, d_i) = 1,$ then the degree of the covering map at $\widetilde{p_i}$ is $d,$ hence the branching order is $d - 1.$ For any $d_i \geq 1,$ there exist $\gcd(d, d_i)$ preimages $\widetilde{p_i}$ on $X.$ At each $\widetilde{p_i},$ the degree of the covering is $\frac{d}{\gcd(d, d_i)},$ so the branching order is $\frac{d}{\gcd(d, d_i)} - 1.$ We apply the Riemann-Hurwitz formula, then $$\begin{array}{rcl}
2 - 2 \textrm{genus}(X) & = & d (2 - 2 \cdot 0) - \sum\limits_i \gcd(d, d_i) \left( \frac{d}{\gcd(d, d_i)} - 1 \right)\\
& = & d (2 - 2 \cdot 0) - \sum\limits_i \left( d - \gcd(d, d_i) \right) \\
& = & d (2 - n) + \sum\limits_i \gcd(d, d_i)
\end{array}$$ and therefore $\textrm{genus}(X) = \frac{d (n - 2)}{2} + 1 - \frac{1}{2} \sum \gcd(d, d_i).$
\end{proof}

In the following theorem we prove that, given $n \geq 3$ and $g \geq 2,$ there are only finitely many cyclically branched covers $X$ over an $n$-punctured sphere where $\textrm{genus}(X) = g.$

\begin{thm}
\label{thm: genera}
Let $n \geq 3$ and $g \geq 2.$ Then there are finitely many surfaces $X$ such that $X$ is a cyclically branched cover over an $n$-punctured sphere and $\textrm{genus}(X) = g.$ Moreover, the degree $d$ of such covering is bounded by $d \leq d(g, n)$ where
$$\begin{array}{ll}d \leq 84 (g - 1), \qquad & \textit{if} \quad n = 3\\
d \leq 12 (g - 1), \qquad & \textit{if} \quad n = 4\end{array}$$ and $$\frac{2 g}{n - 2}  + 1 \leq d \leq \frac{4 (g - 1)}{n - 4}, \qquad \textrm{if} \quad n \geq 5.$$
\end{thm}

\begin{proof}
In Proposition \ref{prop: g(X)}, we found the genus of a $d$-fold cyclically branched covering $X$ over an $n$-punctured sphere defined by branching indices $(d_1, \ldots , d_n).$ For a fixed $n \geq 3$ and a fixed $g \geq 2,$ we will use the relation between $n, d, d_i,$ and $\textrm{genus}(X)$ to find a bound on $d.$ Then, for a given $d$ there are only finitely many integer paritions of $d$ into $n$ parts. This proves that there are only finitely many $d$-fold cyclically branched covers over an $n$-punctured sphere, where $\textrm{genus}(X) = g.$ \\

Recall that we assume $\gcd(d_1, \ldots , d_n) = 1.$ Then $n \leq \sum \gcd(d, d_i) < \frac{d}{2} n,$ and due to the previous proposition, $$n \leq d (n - 2) + 2 - 2 g < \frac{d}{2} n$$ which yields the desired bound for $d$ given $n \geq 5.$ \\

If $n = 3,$ then due to Proposition \ref{prop: g(X)} we have $\sum \gcd(d, d_i) = d + 2 - 2 g.$ Since $\gcd(d, d_i)$ is a factor of $d,$ $\sum \gcd(d, d_i)$ can be written as $\sum \gcd(d, d_i) = \frac{d}{p} + \frac{d}{q} + \frac{d}{r}$ for some $p, q, r > 1,$ not necessarily distinct factors of $d.$ As our goal is to find an upper bound for $d,$ we look for an upper bound for $\sum \gcd(d, d_i).$ We use the fact that $\frac{1}{p} + \frac{1}{q} + \frac{1}{r} \leq \frac{1}{2} + \frac{1}{3} + \frac{1}{7}$ for any $p, q, r$ that satisfy $\frac{1}{p} + \frac{1}{q} + \frac{1}{r} < 1.$ Therefore, $$d + 2 - 2 g \leq \frac{d}{2} + \frac{d}{3} + \frac{d}{7} = \frac{41}{42} d$$ with which we achieve an upper bound $d \leq 84 (g - 1)$ for $g \geq 2.$\\

If $n = 4,$ then $\sum \gcd(d, d_i) = 2 (d + 1 - g).$ Similarly $\sum \gcd(d, d_i) = \frac{1}{p} + \frac{1}{q} + \frac{1}{r} + \frac{1}{s}$ for some $p, q, r, s > 1,$ are factors of $d.$ Since $\frac{1}{p} + \frac{1}{q} + \frac{1}{r} + \frac{1}{s} \leq \frac{1}{2} + \frac{1}{2} + \frac{1}{2} + \frac{1}{3}$ for any $p, q, r, s$ that satisfy $\frac{1}{p} + \frac{1}{q} + \frac{1}{r} + \frac{1}{s} < 2,$ $$2 (d + 1 - g) \leq \frac{d}{2} + \frac{d}{2} + \frac{d}{2} + \frac{d}{3} = \frac{11}{6} d.$$ Therefore, $d \leq 12 (g - 1)$ for $g \geq 2.$ 
\end{proof}

We note that a partition of an integer does not include permutation of the parts. For notation, we let $d_1 \leq \ldots \leq d_n.$ All such coverings up to genus five are listed in Appendix~\ref{genera}.

\subsection*{Octa-4 as a cyclically branched cover over a sphere}
\label{subsec: construction}
Earlier, we proved that the underlying surface of the Octa-4 surface has the conformal structure as an eightfold cyclically branched cover over a thrice-punctured sphere. The following theorem identifies the branching indices that define the Octa-4 surface.

\begin{thm*} [\cite{lee2017triply}] \label{thm: 125} The Octa-4 surface has the same conformal structure as the eightfold cyclically branched cover over a thrice-punctured sphere defined by branching indices $(d_1, d_2, d_3) = (1, 2, 5).$
\end{thm*}

\begin{proof} Let $a$ be an order-eight rotation that yields a sphere $X / \langle a \rangle.$ The quotient sphere $Y_1$ is shaded in Figure~\ref{hyperbolic_tiling}. The branched points are marked as $\bullet,$ $\ast,$ $\circ,$ including the center of the tessellation. We label the spheres so that $Y_1, Y_2, \ldots , Y_8$ are aligned counterclockwise around the center of the tessellation. A positively oriented simple closed curve around the center of the tessellation goes from $Y_j$ to $Y_{j + 1},$ then to $Y_{j + 2},$ and so on. We label the center of the tessellation as $\widetilde{p_1},$ then $d_1 = 1.$ A positively oriented simple closed curve around $\bullet$ via the identification of the edges goes from $Y_j$ to $Y_{j+5}$ for any $j.$ We label this point as $\widetilde{p_3},$ then the branching index at the corresponding point is $d_3 = 5.$ Notice that $\ast$ appears on $Y_1,$ $Y_3,$ $Y_5,$ and $Y_7,$ where $\circ$ appears on $Y_2,$ $Y_4,$ $Y_6,$ and $Y_8.$ That is, $p_2$ has two preimages on $X,$ denoted by $\widetilde{p_2}_1$ and $\widetilde{p_2}_2,$ and $\gcd(d, d_2) = 2.$
\end{proof}

\subsection{Maps between cyclically branched covers}
\label{sec: aut} 
To find the automorphism group on cyclically branched coverings over punctured spheres, we first study maps on punctured spheres and then look at their lifts. \\

Recall the construction of a cyclically branched cover $X$ from Subsection \ref{sec: construction}. Analogously, define $Y' = \mathbb{S}^2 \setminus \{p_1', \ldots , p_{n'}'\}$ to be an $n'$-punctured sphere, $q'$ a base point, and $\gamma_i'$ the branch cuts. Given $d'$ and $(d_1', \ldots , d_{n'}'),$ we likewise construct $X'.$ \\

Let $\phi: Y \rightarrow Y'$ be a holomorphic map so that $\phi(p_i) = p_{\phi(i)}'$ and $\phi(\gamma_i) = \gamma_{\phi(i)}'.$ In other words, the punctures are mapped to punctures and the branch cuts are mapped to branch cuts. This implies that $\phi(q) = q'$ and that there is an induced map $\phi: \{1, \ldots , n\} \rightarrow \{1, \ldots , n'\}$ on the indices. In the following theorem, we find a necessary and sufficient condition for $\phi$ so that its lift $\widehat\phi: X \rightarrow X'$ exists. Specifically, we seek maps that preserve the identification of branch cuts and are therefore compatible with the construction of the coverings.

\begin{thm} \label{thm:aut} Let $X$ be a $d$-fold cyclically branched cover over an $n$-punctured sphere $Y$ with branching indices $(d_i, \ldots , d_n),$ and $X'$ be a $d'$-fold cyclically branched cover over an $n'$-punctured sphere $Y'$ with branching indices $(d_i', \ldots , d_{n'}').$ Let $\phi: Y \rightarrow Y'$ be a map such that $\phi(p_i) = p_{\phi(i)}'$ and $\phi(\gamma_i) = \gamma_{\phi(i)}'$ that induces a map $\phi: \{1, \ldots , n\} \rightarrow \{1, \ldots , n'\}$ on the indices. Then $\widehat\phi: (X, \{1, \ldots , n\}) \rightarrow (X', \{1, \ldots , n'\})$ exists if and only if \begin{equation}\label{eqn: mu}\exists \, \mu \in \mathbb{Z} / d' \mathbb{Z} \quad \textrm{such that} \quad d_{\phi(i)}' \equiv \mu \cdot d_i \pmod{d'}.\end{equation} Then $\widehat\phi$ is given by \begin{equation}\label{eqn: lift}\widehat\phi: (\widetilde{p_i}, j) \mapsto (\widetilde{\phi(p_i)}, \mu \cdot j + \nu \pmod{d'})\end{equation} for some $\nu \in \{0, 1, \ldots , d'\}.$
\end{thm}

Before we prove the theorem, we look at examples of maps on thrice-punctured spheres and their possible lifts on cyclically branched coverings. 

\begin{exmp} \label{exmp: 124} Let $X$ be a sevenfold cyclically branched cover over a thrice-punctured sphere defined by branching indices $(d_1, d_2, d_3) = (1, 2, 4).$ By Proposition \ref{prop: g(X)}, the genus of the covering is $\frac{7}{2} + 1 - \frac{1}{2} \cdot 3 = 3.$ In fact, this is the description of Klein's quartic as a cyclically branched cover (\cite{weber1998klein}). Let $\phi$ be a self-map on the thrice-punctured sphere, that cyclically permutes $p_i.$ In other words, either $\phi$ is the identity map, $\phi: p_1 \mapsto p_2 \mapsto p_3 \mapsto p_1,$ or $\phi: p_1 \mapsto p_3 \mapsto p_2 \mapsto p_1.$ In other words, $\mu$ must be one of 1, 2, or 4 (Equation~\ref{eqn: mu}). If $\mu = 1,$ then $\phi$ is the identity map, so we have $$\widehat\phi: (\widetilde{p_i}, j) \mapsto (\widetilde{p_i}, j + \nu) \mapsto (\widetilde{p_i}, j + 2 \nu) \mapsto \cdots \mapsto (\widetilde{p_i}, j + 6 \nu) \mapsto (\widetilde{p_i}, j)$$ as an order-seven lift unless $\nu = 0,$ in which case it is a trivial map. If $\mu = 2$ (or 4), then $\phi: p_1 \mapsto p_2 \mapsto p_3 \mapsto p_1$ (or $\phi: p_1 \mapsto p_3 \mapsto p_2 \mapsto p_1,$ respectively). Take $\mu = 2,$ then we have $$\widehat\phi: (\widetilde{p_1}, j) \mapsto (\widetilde{p_2}, 2 j + \nu) \mapsto (\widetilde{p_3}, 4 j + 3 \nu) \mapsto (\widetilde{p_1}, j), \qquad \forall \nu = 0, 1, \ldots , 6$$ and all lifts are of order three. The case is similar for $\mu = 4.$
\end{exmp}

\begin{exmp} \label{exmp: 125} Let $X$ be the underlying surface of the Octa-4 surface, conformally equivalent to an eightfold cover over a thrice-punctured sphere defined by $(d_1, d_2, d_3) = (1, 2, 5)$ (Theorem \ref{thm: 125}). Let $\phi$ be a map on the thrice-punctured sphere that permutes $p_i.$ Then, $\mu$ must be either 1 or 5 by (Equation~\ref{eqn: mu}). If $\mu = 1,$ then $\phi$ is the identity map. While the lift fixes $\widetilde{p_1}$ and $\widetilde{p_3},$ it may interchange $\widetilde{p_2}_1$ and $\widetilde{p_2}_2,$ depending on the value of $\nu.$ If $\mu = 5,$ then $\phi$ interchanges $p_1$ and $p_3$ and fixes $p_2.$ However, this involution does not permute the branch cuts as we desire. Consider the thrice-punctured sphere as a doubled triangle with $p_i$ at the vertices. If $\phi$ interchanges $\gamma_1$ and $\gamma_3,$ then $\phi(\gamma_2)$ is no longer a branch cut. We modify the thrice-punctured sphere by making a second branch cut from $q$ to $p_2$ so that $\phi$ permutes the two branch cuts. The original branching index $d_2 = 2$ is then cut down by half at each branch cut. Then $$\widehat\phi: (\widetilde{p_1}, j) \mapsto (\widetilde{p_3}, 5 j + \nu) \mapsto (\widetilde{p_1}, j + 6 \nu) \mapsto \cdots \mapsto (\widetilde{p_3}, 5 j + 3 \nu) \mapsto (\widetilde{p_1}, j).$$ If $\nu$ is odd, then $\widehat{\phi}$ is a lift of order eight; if $\nu$ is either 2 or 6, then $\widehat{\phi}$ is a lift of order four; and if $\nu = 4,$ then $\widehat{\phi}$ is a lift of order two. Also, $\widehat{\phi}$ interchanges $\widetilde{p_2}_1$ and $\widetilde{p_2}_2$ if $\nu$ is odd and fixes $\widetilde{p_2}_i$ if $\nu$ is even.
\end{exmp}

\begin{proof}[Proof~\ref{thm:aut}] Our goal is to show that if a map $\phi: Y \rightarrow Y'$ satisfies the assumptions made above, its lift $\widehat\phi: X \rightarrow X'$ exists if and only if the identification of $\gamma_i'$ is preserved under $\widehat\phi.$ We show that $\widehat\phi$ must be of the form (Equation~\ref{eqn: lift}) and satisfy (Equation~\ref{eqn: mu}). \\

Recall that $X$ is determined by $(d_1, \ldots , d_n)$ where $\gamma_i^l$ on $Y_j$ is identified with $\gamma_i^r$ on $Y_{(j + d_i)}.$ Similarly, for $X',$ $\gamma_{\phi(i)}'^l$ of $Y_{j'}'$ is identified with $\gamma_{\phi(i)}'^r$ of $Y_{j' + d_{\phi(i)}'}'.$ Let $\sigma: \mathbb{Z} / d \mathbb{Z} \rightarrow \mathbb{Z} / d' \mathbb{Z}$ be a map defined on the indices of sheets $Y_j$ such that $\sigma(j) = j'.$ Suppose $\widehat\phi$ is given by $\widehat\phi: (\widetilde{p}, j) \mapsto (\widetilde{\phi(p)}, j' ( = \sigma(j))).$ Then, via $\widehat\phi,$ $\gamma_{\phi(i)}'^l$ of $Y_{\sigma(j)}' ( = Y_{j'}')$ is identified with $\gamma_{\phi(i)}'^r$ of the $Y_{\sigma(j + d_i)}' ( = Y_{j' + d_{\phi(i)}'}').$ Hence, the construction of $X'$ is compatible with the identification of branch cuts if and only if \begin{equation}\label{eqn: compatibility} \sigma(j + d_i) \pmod{d'} \equiv \sigma(j) + d_{\phi(i)}' \qquad \forall i, j.\end{equation} Mathematical induction implies $$\sigma(j + k \cdot d_i) \equiv (\sigma(j) + k \cdot d_{\phi(i)}') \pmod{d'}$$ for all $k.$ As we assume that the covering connected, we have $\gcd(d_1, \ldots, d_n) = 1$, hence there exist coefficients $c_1, \ldots, c_n \in \mathbb{Z}$ such that $\sum\limits_{i = 1}^n c_i d_i = 1.$ Then, $$\sigma(j + 1) = \sigma\left(j + \sum c_i d_i\right) \equiv \sigma(j) + \underbrace{\sum c_i d_{\phi(i)}'}_{=: \mu} \pmod{d'}.$$ By induction, this implies $\sigma(j + k) \equiv \sigma(j) + k \cdot \mu \pmod{d'}$ and by setting $j = 0,$ $$\sigma(k) \equiv \underbrace{\sigma(0)}_{=: \nu} + \mu \cdot k \pmod{d'}$$ as claimed. Lastly, (Equation~\ref{eqn: compatibility}) yields $$\sigma(j + d_i) = \mu \cdot (j + d_i) + \nu = \mu \cdot j + \nu + \mu \cdot d_i = \sigma(j) + \mu \cdot d_i,$$ in other words, $d_{\phi(i)}' \equiv \mu \cdot d_i \pmod{d'}.$
\end{proof}

\subsection{Cone metrics on punctured spheres}
\label{sec: cone}

Let $X$ be a cyclically branched cover over $Y,$ an $n$-punctured sphere. To find translational structures on $X,$ we will find cone metrics on $X$ that arise as pullbacks of cone metrics on $Y.$ Specifically, we will look at cone metrics on $Y$ that arise from the branching indices and define admissible cone metrics which induce holomorphic 1-forms on $X.$ In particular, this gives a translational structure and a basis of holomorphic 1-forms on $X.$ At the end of this section, we will revisit the Octa-4 surface and compute an explicit basis of holomorphic 1-forms on the underlying Riemann surface. 

\begin{defn} A cone metric is given by an atlas where open sets are mapped to either the Euclidean plane or a cone. In either case, the change of coordinates is via Euclidean motions, in other words, orientation preserving isometries. 
\end{defn}

The following proposition is a special case of the Gauss-Bonnet theorem that connects the topology (the Euler characteristic) and the geometry (total curvature) of a surface. It is also a generalization of Descartes' theorem on total angular defect of a polyhedron. We will show that given a cone metric on a compact Riemann surface of genus $g,$ the cone angles satisfy a sum condition. 

\begin{prop}
\label{prop: cone met}
Given a compact Riemann surface of genus $g$ and a cone metric, let $p_1, \ldots , p_n$ be distinguished points with respective cone angles $\theta_i.$ Then, $$\sum \theta_i = 2 \pi (2 g - 2 + n).$$
\end{prop}

\begin{proof} The angle defect, that is, the Gaussian curvature, at generic points is zero. In general, the angle defect is computed as the integral of the Gaussian curvature by taking a circle centered at the cone point. The integral equals the angular defect at the vertex and we say that the curvature is concentrated at the cone points. Then for the total curvature, we have $2 \pi \chi = \sum (2 \pi - \theta_i),$ hence $2 \pi (2 - 2 g) = 2 n \pi - \sum \theta_i.$
\end{proof}

The following theorem considers not only the topology of a surface but also its conformal type. Given distinguished points and cone angles on a compact Riemann surface that satisfies Proposition~\ref{prop: cone met}, there exists a unique cone metric that is compatible with the conformal structure of the given surface.

\begin{thm*} [Schwarz-Christoffel; Troyanov]
Let $Y$ be a compact Riemann surface of genus $g.$ Let $p_1, \ldots , p_n$ be preassigned points on $Y,$ and let $\theta_1, \ldots , \theta_n \in \R$ satisfy $\sum \theta_i = 2 \pi (2 g - 2 + n).$ Then there exists a cone metric on $Y$ compatible with the conformal structure of $Y$ where $p_i$ are the cone points with respective cone angles $\theta_i.$ Moreover the cone metric is uniquely defined up to dilation.
\end{thm*}

As we have cone metrics on $Y,$ now we construct cone metrics on $X$ by pulling back cone metrics on $Y,$ then determine which pullbacks give us translation structures on $X.$

\begin{defn} Let $X \rightarrow Y$ be a covering. We say a cone metric on $Y$ is \textit{admissible} if its pullback yields a translational structure on $X.$
\end{defn}

In the following lemma and theorem, $Y$ is an $n$-punctured sphere and $X$ is a $d$-fold cyclically branched cover over $Y$ with branching indices $(d_1, \ldots , d_n).$ To find cone metrics on $Y$ whose lifts yield translational structures on $X,$ we need to determine the holonomy of the metrics for all closed curves on $X.$ The following lemma gives us a necessary and sufficient condition for admissible cone metrics.

\begin{lemma}
Let $\gamma_i$ be a branch cut on $Y$ from a base point $q$ to $p_i.$ Let $\Gamma$ be an oriented closed curve on $Y$ and let $c_i = \iota (\Gamma, \gamma_i)$ be the intersection number counting multiplicities. Then the lift $\widetilde{\Gamma}$ is closed in $X$ if and only if $\sum c_i d_i \equiv 0 \pmod d.$
\end{lemma}

\begin{proof} Let $\Gamma: [0, 1] \rightarrow Y$ be a closed curve and $\widetilde{\Gamma}(0) \in Y_j$ for some $j \in \{1, \ldots , d\}.$ If $c_1 = \iota (\Gamma, \gamma_1) = 1$ and $c_i = 0$ for $i = 2, \ldots, n,$ then we have $\widetilde{\Gamma}(1) \in Y_{j + d_1 \pmod d}.$ Still assuming $c_i = 0$ for $i = 2, \ldots, n,$ now take $c_1$ to be arbitrary. Then we have $\widetilde{\Gamma}(1) \in Y_{j + c_1 d_1 \pmod d}.$ In other words, $\widetilde{\Gamma}$ is closed in $X$ if and only if $c_1 d_1 \equiv 0 \pmod d.$ Now, in general, for arbitrary $c_i = \iota (\Gamma, \gamma_i),$ $\widetilde{\Gamma}(1)$ lies in $Y_{j + \sum c_i d_i \pmod d}.$ In other words, $\widetilde{\Gamma}$ is closed in $X$ if and only if $\sum c_i d_i \equiv 0 \pmod d.$ 
\end{proof}

\begin{thm} 
Let $Y$ be an $n$-punctured sphere and $X$ be a $d$-fold cyclically branched cover over $Y$ determined by branching indices $(d_1, \ldots , d_n).$ Let $\gamma_i$ be the branch cut on $Y$ from $q$ to $p_i.$ Let $\theta_i = a_i \cdot \frac{2 \pi}{d},$ for some $a_i,$ be respective cone angles at $p_i$ that satisfy $\sum\theta_i = 2 \pi (2 g - 2 + n).$ Then the cone metric on $Y$ defined by $(\theta_1, \ldots , \theta_n)$ is admissible if and only if $a_i \in \mathbb{Z},$ and also $\sum C_i \cdot d_i \equiv 0 \pmod d$ implies $\sum C_i \cdot a_i \equiv 0 \pmod d,$ where $C_i \in \mathbb{Z}$ implies that $\Gamma = C_1 \Gamma_1 + \cdots + C_n \Gamma_n$ is a positively oriented curve in $Y,$ where $\Gamma_i$ a closed curve on $Y$ with winding number $\textrm{Ind}_{\Gamma_i}(p_j) = \delta_{i j}.$
\end{thm}

\begin{proof}
Let $(\theta_1, \ldots , \theta_n)$ define a cone metric on $Y$ that is compatible with its conformal type. Given a closed curve $\Gamma_i$ in $Y,$ whose winding number is $\textrm{Ind}_{\Gamma_i}(p_j) = \delta_{i j},$ its lift $\widetilde{\Gamma_i}$ on $X$ has an argument of $\theta_i.$ By the previous lemma, the lift of the $d$-fold concatenation of $\Gamma_i$ has trivial holonomy on $X$ if and only if $d \cdot \theta_i$ is a multiple of $2 \pi,$ therefore $a_i$ must be an integer. As we can write any closed curve $\Gamma$ on $Y$ as a concatenation of $\Gamma_i,$ we show that its lift $\widetilde{\Gamma}$ is closed in $X$ if and only if the arguments derived from each $\widetilde{\Gamma_i}$ add up to a multiple of $2 \pi.$ In other words, $\widetilde{\Gamma}$ is closed in $X$ if and only if $\sum C_i \cdot d_i \equiv 0 \pmod d$ implies that $\Gamma = C_1 \Gamma_1 + \cdots + C_n \Gamma_n$ closed in $Y.$ In this case, we say a cone metric given by $(\theta_1, \ldots , \theta_n)$ is admissible. Given an admissible cone metric, analytic continuation along $\widetilde{\Gamma}$ yields trivial holonomy on $X,$ hence $$0 \equiv \sum C_i \cdot \theta_i \equiv \sum C_i \cdot a_i \cdot \frac{2 \pi}{d} = \frac{2 \pi}{d} \sum C_i \cdot a_i \pmod {2 \pi}$$ which implies $\sum C_i \cdot a_i \equiv 0 \pmod d.$\\

To show that a cone metric defined on $Y$ is admissible, we need to show that its pullback is a well-defined 1-form on $X.$ Let a cone metric on $Y$ be defined by $(\theta_1, \ldots , \theta_n) = \frac{2 \pi}{d} (a_1, \ldots , a_n)$ and $\Gamma$ be a closed curve in $Y.$ If $a_i \in \mathbb{Z}$ and $0 \equiv \sum C_i \cdot a_i \pmod d$ for every $\widetilde{\Gamma}$ on $X,$ then $\textrm{arg} \int_{\widetilde{\gamma}} \omega$ is a multiple of $2 \pi,$ where $\omega$ is the pullback of the cone metric given by the cone angles $\theta_i.$ The trivial holonomy implies well-definedness of $\omega$ on $X,$ hence the cone metric is admissible.
\end{proof}

Specifically if $a_i \equiv d_i \pmod d$ for all $i,$ then the cone metric defined by $(\theta_1, \ldots , \theta_n) = \frac{2 \pi}{d} (d_1, \ldots , d_n)$ is admissible. Given an admissible cone metric, we propose a natural way of finding other admissible cone metrics using the following definition.

\begin{defn} Given branching indices $(d_1, \ldots , d_n),$ we say $a \in \{1, \ldots, d - 1\}$ is a \textit{multiplier} if $a (d_1, \ldots , d_n) := (a \cdot d_1 \pmod d, \ldots , a \cdot d_n \pmod d)$ is admissible. For simplicity, we denote $a (d_1,\ldots , d_3)$ as $(a_1, \ldots , a_3)$ where $a_i \in \{0, \ldots , d - 1\}.$ 
\end{defn}

The following corollary shows how the multipliers give us a natural way of finding other admissible cone metrics.

\begin{cor}
\label{cor: admissible metric} Given branching indices $(d_1, \ldots , d_n),$ let $a_i \equiv a \cdot d_i \pmod d$ for some $a \in \mathbb{Z}.$ If $a_i > 0$ for all $i$ and $\sum a_i = d (n - 2)$ then the cone metric given by cone angles $\frac{2 \pi}{d} (a_i, \ldots , a_n)$ is admissible.
\end{cor}

\begin{proof} Apply Proposition \ref{prop: cone met} and let $g = 0.$
\end{proof}

We revisit Example \ref{exmp: 124} and Example \ref{exmp: 125} and find an explicit basis of holomorphic 1-forms on each covering that arise from multipliers of branching indices. 

\begin{exmp} \label{exmp: 124_} (Continued from Example \ref{exmp: 124}.) Let $X$ be the sevenfold cyclically branched cover over a thrice-punctured sphere defined by branching indices $(1, 2, 4).$ We have multipliers 1, 2, and 4, which yield admissible cone metrics defined by cone angles $\frac{2 \pi}{7} (1, 2, 4),$ $\frac{2 \pi}{7} (2, 4, 1),$ and $\frac{2 \pi}{7} (4, 1, 2),$ respectively. Then, we have holomorphic 1-forms that have zeros of order $\frac{\theta_i}{2 \pi} - 1$ at each $p_i.$ They form a basis of holomorphic 1-forms $\{\omega_i\}$ on $X$ with divisors $$(\omega_1) = \widetilde{p_2} + 3 \widetilde{p_3}, \qquad (\omega_2) = \widetilde{p_1} + 3 \widetilde{p_2}, \qquad (\omega_3) = 3 \widetilde{p_1} + \widetilde{p_3}.$$
\end{exmp}

\begin{exmp} \label{exmp: 125_} (Continued from Example \ref{exmp: 125}.) Let $X$ be an eightfold cyclically branched cover over a thrice-punctured sphere defined by $(1, 2, 5).$ Then we get admissible cone metrics with cone angles $\frac{2 \pi}{8} (1, 2, 5),$ $\frac{2 \pi}{8} (2, 4, 2),$ and $\frac{2 \pi}{8} (5, 2, 1).$ The three admissible cone metrics induce a basis of holomorphic 1-forms $\{\omega_i\}$ on $X$ with divisors $(\omega_1) = 4 \widetilde{p_3}, (\omega_2) = \widetilde{p_1} + \widetilde{p_2}_1 + \widetilde{p_2}_2 + \widetilde{p_3},$ and $(\omega_3) = 4 \widetilde{p_1}.$ Note that 4 is not a multiplier as a cone metric derived from $\frac{2 \pi}{8} (4, 0, 4)$ is not admissible. The cone metric induces a meromorphic 1-form whose divisor is $3 \widetilde{p_1} - \widetilde{p_2}_1 - \widetilde{p_2}_2 + 3 \widetilde{p_3}.$ This shows why $a_i > 0$ is necessary.
\end{exmp}

From the two examples, one might ask whether all admissible cone metrics arise from multiples of $\frac{2 \pi}{d} (d_i)$. We prove that the answer is yes, if $n = 3.$

\begin{thm}
\label{thm: multipliers}
Let $X \rightarrow Y$ be a $d$-fold cyclically branched cover over a thrice-punctured sphere with branching indices $(d_1, d_2, d_3).$ Then there are exactly $g$ admissible cone metrics that arise from multipliers. 
\end{thm}

\begin{proof}
Given $d$ and $(d_1, d_2, d_3),$ we have $d - 1$ multiples $(a_1, a_2, a_3)$ of the branching indices. Out of the $d - 1$ multiples, we discard those if $a_i = 0,$ for any $i.$ Since $\gcd(d_1, d_2, d_3) = 1,$ we can have at most one $a_i$ be zero among $a_1, a_2,$ and $a_3,$ for a given multiple of $(d_1, d_2, d_3).$ This results in removing $(\gcd(d, d_1) - 1) + (\gcd(d, d_2) - 1) + (\gcd(d, d_3) - 1)$ many multiples out of $d - 1.$ Note that if $\gcd(d, d_i) = 1$ for all $i,$ then none of the multiples are removed at this step. Now we are left with $2 g$ multiples of $(d_1, d_2, d_3)$ (Proposition \ref{prop: g(X)}). We show that there is an involution on the remaining $2 g$ multiples of $(d_1, d_2, d_3),$ that maps the admissible cone metrics to non-admissible cone metrics. The involution is a multiplication by $d - 1.$ That is, if $\frac{2 \pi}{d} (d_1, d_2, d_3)$ yields an admissible cone metric then $$\begin{array}{rl}(d_1, d_2, d_3) \xrightarrow{\times (d - 1)} & ((d - 1) d_1, (d - 1) d_2, (d - 1) d_3) \\
\equiv & (d - d_1, d - d_2, d - d_3) \pmod d\end{array}$$ does not yield an admissible cone metric because $\sum (d - d_i) = 2 d \neq d (n - 2)$ due to the corollary.
\end{proof}

However, the theorem does not apply in general cases if $n \geq 4.$ For example, let $X$ be a twofold cover over a 4-punctured sphere with branching indices $(1, 1, 1, 1).$ Then, the covering is a genus one surface (Proposition \ref{prop: g(X)}). We observe 1 is the only multiplier and $\frac{2 \pi}{2} (1, 1, 1, 1)$ yields the only admissible cone metric. On the other hand, if $X$ is a fourfold cover with branching indices $(1, 1, 1, 1),$ then the genus of the covering is three. However, the only multiplier is 2 where the cone angles are given by $\frac{2 \pi}{4} (2, 2, 2, 2).$ Other admissible cone metrics arise from $(5, 1, 1, 1),$ $(1, 5, 1, 1),$ $(1, 1, 5, 1),$ and $(1, 1, 1, 5).$ This does not change the construction of $X,$ since $(5, 1, 1, 1)$ $\equiv$ $(1, 5, 1, 1)$ $\equiv$ $(1, 1, 5, 1)$ $\equiv$ $(1, 1, 1, 5)$ $\equiv$ $(1, 1, 1, 1) \pmod 4.$ To form a basis, one needs to make a choice of admissible cone metrics. Without loss of generality, pick $p_1$ then we have a basis of 1-forms that have different orders of zero at $p_1,$ whose divisors are $$\begin{array}{ccccc} (\omega_1) = & & 4 \widetilde{p_2} & & \\
(\omega_2) = & \widetilde{p_1} & + \widetilde{p_2} & + \widetilde{p_3} & + \widetilde{p_4}\\
(\omega_3) = & 4 \widetilde{p_1} & & & \end{array}.$$

\begin{conj} Given a $d$-fold cyclically branched covering over an $n$-punctured sphere defined by branching indices $(d_1, \ldots , d_n),$ there are at least $g$ admissible cone metrics that are of the form $a (d_1, \ldots , d_n) := (a \cdot d_1 \pmod d, \ldots , a \cdot d_n \pmod d).$ \end{conj}

Recall Corollary \ref{cor: admissible metric}. In other words, we allow $a \cdot d_i \pmod d$ to be greater than $d - 1.$ Furthermore, the cone metrics that arise this way yield a basis of holomorphic 1-forms. In Appendix~\ref{genera}, we check that the conjecture is true for all coverings of lower genera $(g \leq 5).$

\subsection*{Octa-4, flat structures and algebraic equations}
\label{subsec: cone}
We will mainly discuss two topics in this section. First, we will find the flat structures on the underlying surface of the Octa-4 that are compatible with its conformal type. In Example \ref{exmp: 125_}, we found cone metrics that arose from multipliers and found an explicit basis of holomorphic 1-forms. We will show that the geometric representations of the 1-forms are translational structures. Then, with the given basis of holomorphic 1-forms we will find an algebraic equation that describes the Octa-4 surface and show that it is projectively equivalent to Fermat's quartic. \\

First we look at $\omega_1$ using the notation from Example~\ref{exmp: 125_}. We map a hyperbolic $(\frac{\pi}{4}, \frac{\pi}{4}, \frac{\pi}{4})$-triangle to a Euclidean $(\frac{\pi}{8}, \frac{2 \pi}{8}, \frac{5 \pi}{8})$-triangle. Then, by Schwarz reflection principle we obtain a conformally equivalent flat metric (Figure \ref{flat}). The hyperbolic geodesics are mapped to straight lines in the Euclidean plane. In other words, the identification of the edges are preserved in the flat structure by translations. We put a translational structure $d z$ at the interior and along the edges of the flat polygon. The cone angle is $2 \pi$ at $\widetilde{p_1}$ and $\widetilde{p_2}_i,$ hence they are generic points. However, the cone angle at $\widetilde{p_3}$ is $\frac{5 \pi}{4} \times 8 = 2 \pi (4 + 1).$ This corresponds to a holomorphic 1-form that has a zero of order four at $\widetilde{p_3}.$

\begin{figure}[htbp] 
   \centering
   \includegraphics[width=4in]{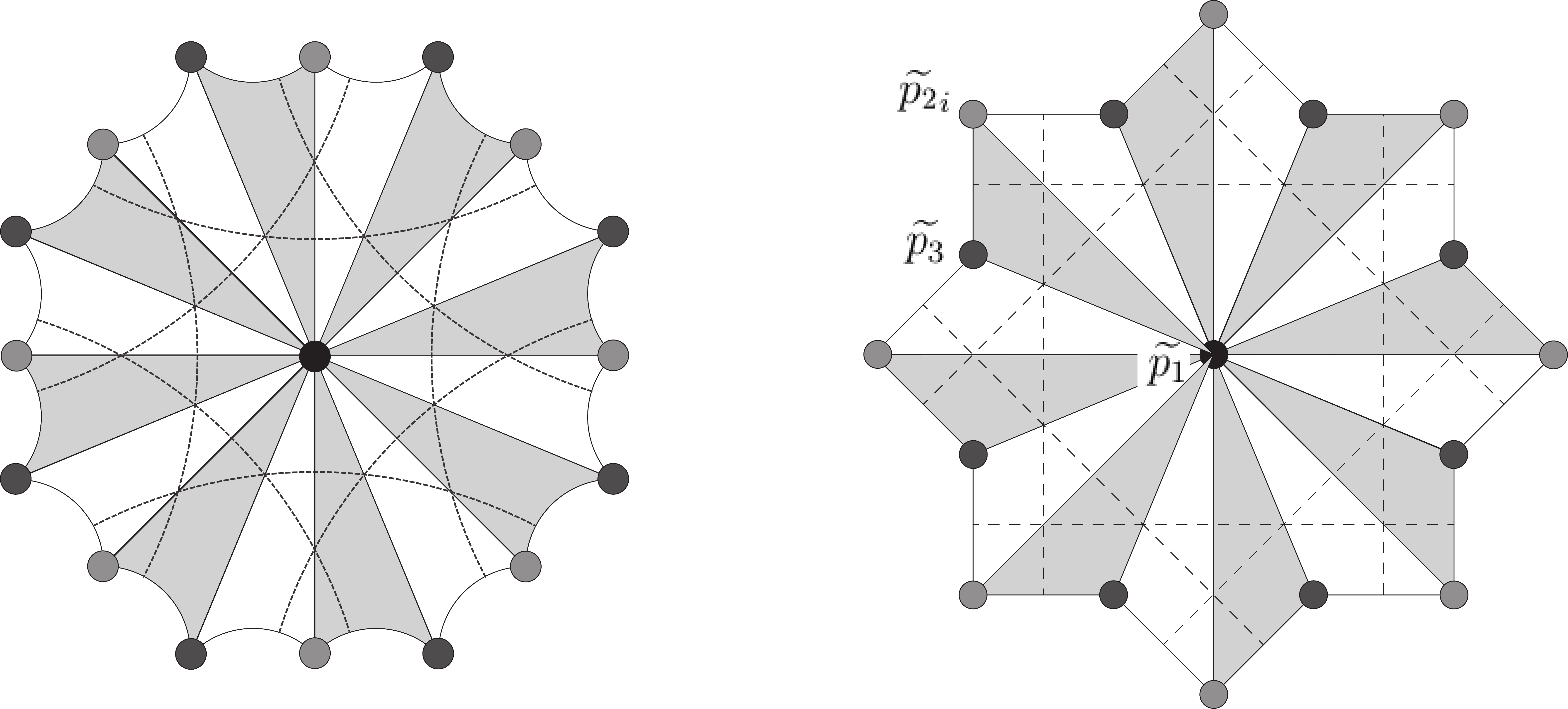} 
   \caption{A translational structure on $(X, \omega_1)$}
   \label{flat}
\end{figure}

Next we look at the translational structure induced from $\omega_2$ and map the hyperbolic triangles to Euclidean $(\frac{2 \pi}{8}, \frac{4 \pi}{8}, \frac{2 \pi}{8})$-triangles, where the edges are again identified by translations (Figure \ref{flat2}). Then around each $\widetilde{p_1}, \widetilde{p_2}_1, \widetilde{p_2}_2,$ and $\widetilde{p_3},$ we get a cone angle of $2 \pi (1 + 1).$ This corresponds to $\omega_2$ that has simple zeros at $\widetilde{p_1}, \widetilde{p_2}_1, \widetilde{p_2}_2,$ and $\widetilde{p_3}.$ 

\begin{figure}[htbp] 
   \centering
   \includegraphics[width=4in]{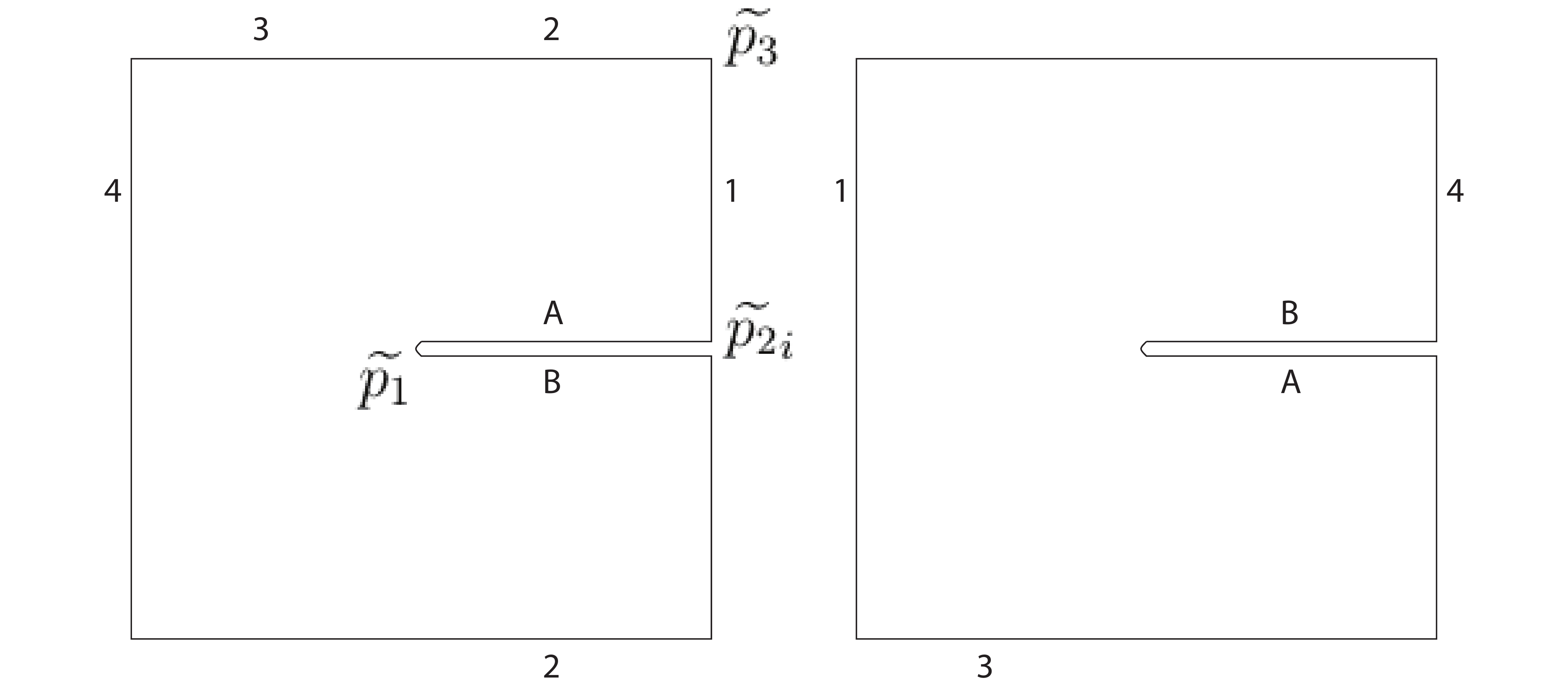} 
   \caption{A translational structure on $(X, \omega_2)$}
   \label{flat2}
\end{figure}

Similarly, we can map hyperbolic triangles to $\left(\frac{5 \pi}{8}, \frac{2 \pi}{8}, \frac{\pi}{8}\right)$-triangles for $\omega_3$ and show that the cone angle is $2 \pi (4 + 1)$ at $\widetilde{p_1};$ $2 \pi$ at $\widetilde{p_2}_1,$ $\widetilde{p_2}_2,$ and $\widetilde{p_3}.$ This yields $\omega_3,$ therefore we get $\{\omega_i\}$ as a basis of holomorphic 1-forms.\\

As every compact Riemann surface is a complex algebraic curve, Proposition 2.1 of Chapter 7 from \cite{miranda1995algebraic} states that if $X$ is a non-hyperelliptic algebraic curve of genus $\geq 3,$ then it can be embedded into $\mathbb{P}^{g - 1}$ as a smooth projective curve of degree $2 g - 2.$ As the Octa-4 is a genus three curve, we will find a quartic curve to show that is not hyperelliptic. \\
 
We will find an algebraic description of the Octa-4 surface that is equivalent to its conformal type. With $\omega_i$ as defined, we define holomorphic functions $f, g: X \rightarrow \widetilde{\mathbb{C}}$ so that $$(f) := \left(\frac{\omega_1}{\omega_2}\right) = -\widetilde{p_1} - \widetilde{p_2}_1 - \widetilde{p_2}_2 + 3 \widetilde{p_3}, \qquad (g) := \left(\frac{\omega_3}{\omega_2}\right) = 3 \widetilde{p_1} - \widetilde{p_2}_1 - \widetilde{p_2}_2 - \widetilde{p_3}.$$ Then, with $(f^3 g) = -4 \widetilde{p_2}_1 - 4 \widetilde{p_2}_2+ 8 \widetilde{p_3}$ and $(f g^3) = 8 \widetilde{p_1} - 4 \widetilde{p_2}_1 - 4 \widetilde{p_2}_2$ and after proper scaling of the functions we get \begin{equation} \label{eqn: Fermat} f^3 g -1 = f g^3 \Rightarrow \omega_1^3 \omega_3 - \omega_1 \omega_3^3 = \omega_2^4. \end{equation}

In \cite{lee2017triply}, it is shown that the quartic equation in (\ref{eqn: Fermat}) is projectively equivalent to Fermat's quartic.

\begin{thm*} [\cite{lee2017triply}]
The conformal structure on the underlying surface of the Octa-4 is conformally equivalent to Fermat's quartic.
\end{thm*}

\subsection{Wronski metric and Weierstrass points}
\label{sec: wronski}

In this section, we look for Weierstrass points on compact Riemann surfaces. As they carry information about the automorphism group of the surface, it allows one to distinguish one Riemann surface from another. In this section, we will discuss how one can locate all Weierstrass points on a given surface by defining the Wronski metric. Lastly, we will locate all Weierstrass points on the underlying surface of the Octa-4 surface and show that they correspond to the vertices on its triply periodic polyhedral realization in $\R^3.$\\

The following definitions, theorems, and propositions can be found in \cite{farkas1992riemann}.

\begin{thm*}[Weierstrass ``gap'' theorem] \label{thm: gap} Let $X$ be a Riemann surface of positive genus and let $p \in X$ be an arbitrary point. Then there are precisely $g$ integers $$1 = n_1 < n_2 < \cdots < n_g < 2 g$$ such that there does not exist a meromorphic function $f$ on $X$ that is holomorphic on $X \setminus \{p\}$ with a pole of order $n_i$ at $p.$
\end{thm*}

\begin{defn*} Given the gap sequence $n_1, \ldots , n_g$ at $p \in X,$ the \textit{weight} of a point $p$ is defined by $\textrm{wt}_p := \sum\limits_{i = 1}^g (n_i - i).$ If $\textrm{wt}_p > 0,$ then we say that $p$ is a \textit{Weierstrass point.} 
\end{defn*}

\begin{prop*} Let $X$ be a compact Riemann surface of genus $\geq 2.$ Let $\textrm{wt}_p$ be the weight of $p \in X.$ Then $\sum\limits_{p \in X} \textrm{wt}_p = (g - 1) g (g + 1).$
\end{prop*}

\begin{prop*} Let $W(X)$ be the finite set of Weierstrass points on $X.$ If $\phi \in \textrm{Aut}(X),$ then $\phi (W(X)) = W(X).$ In fact, the gap sequences at $p \in X$ and $\phi(p)$ are the same.
\end{prop*}

In general, there is no straightforward way of locating all Weierstrass points even on a highly symmetric surface. For example, on Klein's quartic (Example \ref{exmp: 124_}) the gap sequence at each $\widetilde{p_i}$ is $1, 2, 4,$ hence $\textrm{wt}_{\widetilde{p_i}} = 1,$ therefore the surface is not hyperelliptic. Since the genus of Klein's quartic is three, all weights should sum up to 24. We use $\widehat\phi \in \textrm{Aut}(X)$ to locate the other Weierstrass points. For instance, there are seven copies of $\widetilde{q} \in X$ for $q \notin \{p_i\},$ and if $\widetilde{q}$ is a Weierstrass point, then so is $\widehat\phi(q).$ However, it is not apparent where the Weierstrass points are located nor can we conclude the distribution of weights. 

The following proposition from Chapter 3, Section 5 in \cite{farkas1992riemann} tells us that the Wronskian of a basis of holomorphic 1-forms induces a metric on the surface. With this metric, we can find all Weierstrass points on a surface without direct information of its automorphism group. 

\begin{defn*} Given a basis of holomorphic functions $\{f_1, \ldots , f_g\}$ on a Riemann surface $X$ of genus $g$, the Wronskian defined by $$\mathcal{W}(z) := \textrm{det} \left( \frac{d^j f_k(z)}{d z^j} \right)_{j = 0, \ldots , g - 1, \, k = 1, \ldots , g}$$ is a non-trivial holomorphic function on $X$ that induces a metric which we call the \textit{Wronski metric.} 
\end{defn*}

\begin{rem} A change of basis results in a non-zero constant multiple of the Wronskian, hence the zeros are independent of the change of basis (from Chapter 3, Section 5 in \cite{farkas1992riemann}). By induction on $g,$ one can show that a zero of the Wronskian is a Weierstrass point on $X.$ The order of a zero at a point equals its weight which encodes the information about the cone angle. Hence, the Wronski metric is a cone metric. \end{rem}

Let $X$ be a $d$-fold cyclically branched cover over an $n$-punctured sphere $Y$ defined by branching indices $(d_1, \ldots , d_n).$ To compute the Wronski metric on $X,$ we need to find an appropriate coordinate chart $w$ on $X$ which is induced from the pullback of a coordinate chart on $Y.$ Let $z$ be a local coordinate chart on $Y$ such that $z(p_i) = 0.$ Then its pullback on $X$ has a zero of order $\frac{d}{\gcd(d, d_i)}$ at $\widetilde{p_i},$ so we write $w$ as $w(z) = z^{\frac{d}{\gcd(d, d_i)}}$ on $X$ near $\widetilde{p_i}.$ On the other hand, let $\omega_k = f_k(z) d z$ be a holomorphic 1-form on $X$ that is induced from an admissible cone metric defined by cone angles $\frac{2 \pi}{d} (a_1^k, \ldots , a_n^k) \footnote{$k$ is a superscript to denote the cone angles that yield a holomorphic 1-form $\omega_k.$},$ then $\omega_k$ has a zero of order $\frac{a_i^k}{\gcd(d, d_i)} - 1$ at $\widetilde{p_i}.$ Since $d z$ has a zero of order $\frac{d}{\gcd(d, d_i)} - 1,$ $f$ has a zero of order $\frac{a_i^k - d}{\gcd(d, d_i)}$ at $\widetilde{p_i}.$ Furthermore, we can write $f_k(z) = g_k(z) h(z)$ for $k = 1, \ldots , g$ where $g_k(z) = z^{\frac{a_i^k}{\gcd(d, d_i)}}$ and $h(z) = z^{-\frac{d}{\gcd(d, d_i)}}.$ Then, $$\begin{array}{rcl}\mathcal{W}(z) & := & \textrm{det} \left( \frac{d^j f_k(z)}{d z^j} \right)_{j = 0, \ldots , g - 1 \, k = 1, \ldots , g}\\
& = & \textrm{det} \left( f_k^{(j)} \right)\\
& = & h^g \cdot \underbrace{\textrm{det} \left( g_k^{(j)}\right)}_\text{$\mathcal{W}_1(z)$}\end{array}.$$ In other words, the Weierstrass points that are not preimages of any $p_i$ arise as zeros of $\mathcal{W}_1(z).$ However, the weight at $\widetilde{p_i}$ is computed as follows.

\begin{prop}[Chain rule for Wronskian] \label{prop: wronski} Given a basis of holomorphic 1-forms on a compact Riemann surface of genus $g,$ the Wronskian of the basis is a non-trivial holomorphic $\frac{g (g + 1)}{2}$-differential on $X$ and $$\mathcal{W}(\phi(w)) = [ h(\phi(w)) ]^g \cdot [ \phi'(w) ]^{\frac{(g - 1) g}{2}} \cdot \mathcal{W}_1(\phi(w))$$ where the derivatives are taken with respect to $w.$
\end{prop}

\begin{proof}
Let $z = \phi(w)$ then by definition $\mathcal{W}(z) = \mathcal{W}(\phi(w)) = [ h(\phi(w)) ]^g \cdot \textrm{det} \left( g_k^{(j)}(\phi(w))\right)$ where $$\begin{array}{rcl}\textrm{det} \left( g_k^{(j)}(\phi(w))\right) & = & \textrm{det} \begin{pmatrix}
g_k(\phi(w))\\
g_k'(\phi(w)) \cdot \phi'(w)\\
g_k''(\phi(w)) \cdot \left( \phi'(w) \right)^2 + g_k'(\phi(w)) \cdot \phi''(w)\\
\vdots \\
g_k^{(g - 1)}(\phi(w)) \cdot (\phi'(w))^{g - 1} + \cdots + g_k'(\phi(w)) \cdot \phi^{(g - 1)}(w)\end{pmatrix}\\
&&\\
& = & [ \phi'(w) ]^{\frac{(g - 1) g}{2}} \textrm{det} \begin{pmatrix}
g_k(\phi(w))\\
g_k'(\phi(w))\\
g_k''(\phi(w))\\
\vdots\\
g_k^{(g - 1)}(\phi(w))\end{pmatrix},\end{array}$$ hence $$\mathcal{W}(\phi(w)) = [ h(\phi(w)) ]^g \cdot [ \phi'(w) ]^{\frac{(g - 1) g}{2}} \cdot \mathcal{W}_1(\phi(w)).$$ 
\end{proof}

\begin{cor} \label{cor: wronski} $\mathcal{W}(w)$ has a zero of order $$\textrm{wt}_i := \frac{d}{\gcd(d, d_i)} \cdot \left( \frac{(g - 1) g}{2} + b_i \right) - \frac{g (g + 1)}{2}$$ at $\widetilde{p_i}$ where $b_i$ is defined as the order of $\mathcal{W}_1(z)$ at $\widetilde{p_i}.$
\end{cor} 

\begin{proof}
By Proposition \ref{prop: wronski}, we have $$\frac{\gcd(d, d_i)}{d} \times \textrm{wt}_i = - \frac{\gcd(d, d_i)}{d} \times g - \left(\frac{\gcd(d, d_i)}{d} - 1\right) \times \frac{(g - 1) g}{2} + b_i$$ and the order of $\mathcal{W}(w)$ at $\widetilde{p_i}$ is as desired.
\end{proof}

\subsection*{Weierstrass points on the Octa-4 surface}
\label{subsec: wronski}

In this section, we will locate all Weierstrass points on the Octa-4 surface in two different ways. In \cite{lee2017triply}, the author uses the fact that all vertices on $\Pi$ are similar and shows that all Weierstrass points on the compact quotient of the Octa-4 surface correspond to the vertices on the polyhedral surface. 

\begin{thm*} [\cite{lee2017triply}] The Weierstrass points on the underlying surface of the Octa-4 surface correspond to the vertices on the polyhedral surface.
\end{thm*}

Next, we compute the Wronksi metric on the compact quotient of the Octa-4 surface using Corollary~\ref{cor: wronski} to locate all Weierstrass points. Given the branching indices $(1, 2, 5)$ and a chart $(z(p_i)) = (0, 1, -1),$ we have $(b_i) = (-2, -1, -2),$ $\textrm{wt}_i= 2$ for all $i,$ and $\mathcal{W}_1(z) = (1 + 3 z)^2$ (Appendix~\ref{wronski}). In other words, there exists a $\widetilde{q} \in X$ where $z(q) = -\frac{1}{3}$ and $\textrm{wt}(\widetilde{q}) = 2.$ This point is fixed by an involution defined by $\mu = 5$ and $\nu = 4$ (Example \ref{exmp: 125}), that interchanges $\widetilde{p_1}$ and $\widetilde{p_3}$ and interchanges $\widetilde{p_2}_1$ and $\widetilde{p_2}_2.$ In particular, this involution is not hyperelliptic. Hence there is no triply periodic minimal surface whose underlying structure is conformally equivalent to the Octa-4 surface. 

\begin{thm*} [\cite{meeks1990theory}] If $M$ is a minimal surface of genus three, then $M$ is hyperelliptic. \end{thm*}

\begin{cor*} [\cite{lee2017triply}] There is no triply periodic minimal surface that has the same conformal type as the Octa-4 surface.\end{cor*}

\section{Regular Triply Periodic Polyhedral Surfaces}
\label{ch4 reg tpps}
In Section~\ref{ch2 octa-4}, we discussed a regular triply periodic polyhedral surface (Definition~\ref{defn: reg}) that was not included in Coxeter-Petrie's classification of infinite regular polyhedral surfaces due to its lack of Euclidean isometries. To loosen the criteria and broaden the classification, we observe a common factor between the construction of the Mucube, Muoctahedra, Mutetrahedra, and the Octa-4 surface. The Mucube (Muoctahedra and Mutetrahedra, respectively) is the boundary of a triply periodic polyhedron built by cubes (truncated octahedra and truncated tetrahedra, respectively) in a periodic manner. Similarly, the Octa-4 surface arises as the boundary surface of a polyhedron built by regular octahedra. In this section, we seek more examples of triply periodic polyhedral surfaces that arise from such a construction. We formulate a gluing pattern of regular solids using tools from graph theory. We define a \textit{decoration} of a skeletal graph in $\R^3$ to construct triply periodic polyhedra. Then, we put restrictions on the decoration to find regular triply periodic polyhedra and classify all triply periodic polyhedral surfaces that arise as the boundary of such decorations.

\subsection{Triply periodic symmetric skeletal graphs}
\label{sec: defs}

As we are interested in triply periodic polyhedra, we look at their deformation retracts as triply periodic skeletal graphs in $\R^3.$ A \textit{skeletal graph} or a \textit{1-skeleton} of a topological space $X$ is a simplicial complex that is a union of the 0-simplices and 1-simplices of $X.$ 

\begin{exmp} [Octa-4]
\label{exmp: rad octa-4} 
The Octa-4 surface is the boundary of the triply periodic polyhedron we built with regular octahedra. The skeletal graph of the polyhedron is its deformation retract. Each Type B octahedron is adjacent to four (Type A) octahedra, hence we retract each Type B octahedron to its center, a 0-simplex. Similarly, each Type A octahedron is adjacent to two (Type B) octahedra, so we retract each Type A octahedron to a 1-simplex that is incident to two 0-simplices (Type B octahedra). Then the skeletal graph of the solid polyhedron is connected and triply periodic where all of its 0-simplices are incident to four 1-simplices.
\end{exmp}

We are interested in connected, simple, and symmetric skeletal graphs that are embedded in $\R^3.$ All definitions follow from \cite{iyanaga1977encyclopedic}. 

\begin{defn}
A \textit{skeletal graph} $\Gamma = \{V, E\}$ consists of a set $V$ of vertices (0-simplices) and a set $E$ of edges (1-simplices). An edge $e \in E$ is a 2-element subset of $V$ which we denote as an unordered pair $e = \{v_1, v_2\}$ for some $v_1, v_2 \in V.$ A pair of vertices $\{v, v'\}$ is \textit{connected} if sequence of edges leads from $v$ to $v'.$ A \textit{connected graph} is a graph in which every unordered pair of vertices in the graph is connected. We say that a graph is \textit{regular} if every vertex has the same number of incident edges. The number of edges incident to a vertex is called the \textit{degree} of the vertex. If all vertices have the same degree $d,$ we say that a graph is a $d$-regular graph or a regular graph of degree $d.$ We say that a graph is \textit{symmetric} if given any two edges $\{v_1, v_2\}, \{v_1', v_2'\},$ there is an automorphism $\varphi: V \rightarrow V$ such that $\varphi(v_1) = v_1'$ and $\varphi(v_2) = v_2'.$ Lastly, a \textit{simple graph} is a graph that does not allow neither multiple edges nor loops, \textit{multiple edges} are two or more edges that connect the same two vertices, and a \textit{loop} is an edge that connects a vertex to itself. 
\end{defn}

If a graph is symmetric, then its group of automorphisms acts transitively on the edges and hence on the vertices. Note that a connected symmetric graph is regular. In the Euclidean setting, we require the automorphisms to be Euclidean isometries. Therefore, we need to define lengths of edges and angles between two incident edges. We define the length of an edge $e = \{v_1, v_2\}$ by $\|v_2 - v_1\|$ using the Euclidean norm and the angles between two edges $e = \{v, v_1\}$ and $e' = \{v, v_1'\}$ by $\arccos\left(\frac{(v_1 - v) \cdot (v_1' - v)}{\|v_1 - v\| \|v_1' - v\|}\right).$ We say that a skeletal graph $\Gamma$ is \textit{triply periodic} (or \textit{doubly periodic,} respectively) if $\Gamma$ is invariant under $\Lambda,$ a rank-three (or two, respectively) lattice. Given a periodic skeletal graph $\Gamma,$ we define its \textit{compact quotient graph} by $\Gamma' = \{V', E'\} := \Gamma / \Lambda.$ Figure~\ref{fig: symm quot} shows examples of doubly periodic graphs, namely the square tiling and the hexagonal tiling on $\R^2.$ On each graph, the fundamental piece is bounded by the dotted lines. To the rights are their compact quotient graphs. 

\begin{figure}[htbp] 
\centering
\begin{minipage}{.5\textwidth}
	\centering
	\includegraphics[width=2in]{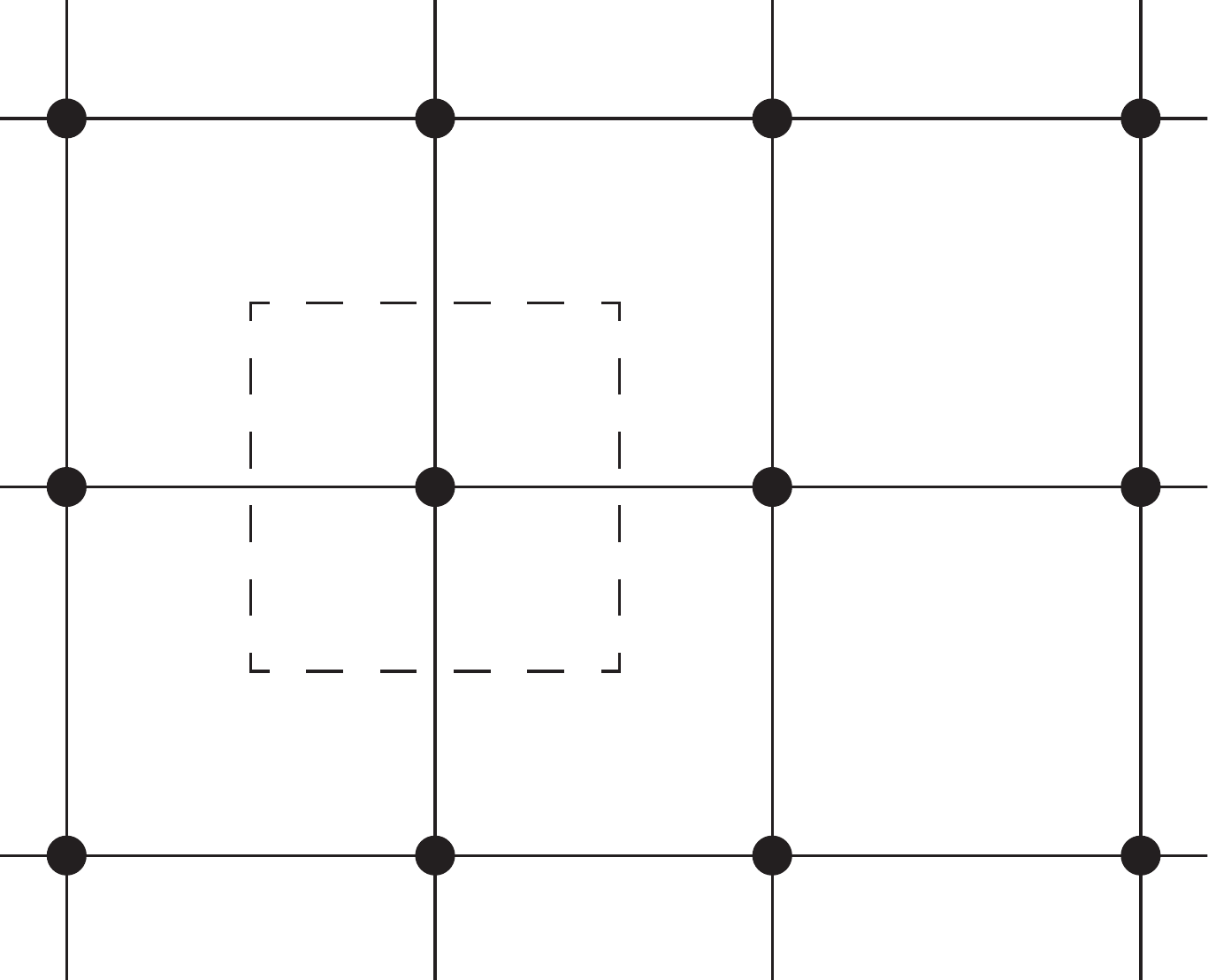}
	\end{minipage}
\begin{minipage}{.5\textwidth}
	\centering
	\includegraphics[width=1.5in]{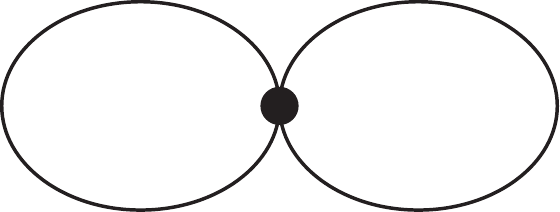}
	\end{minipage}
\vspace{0.35cm}

\begin{minipage}{.5\textwidth}
	\centering
	\includegraphics[width=2in]{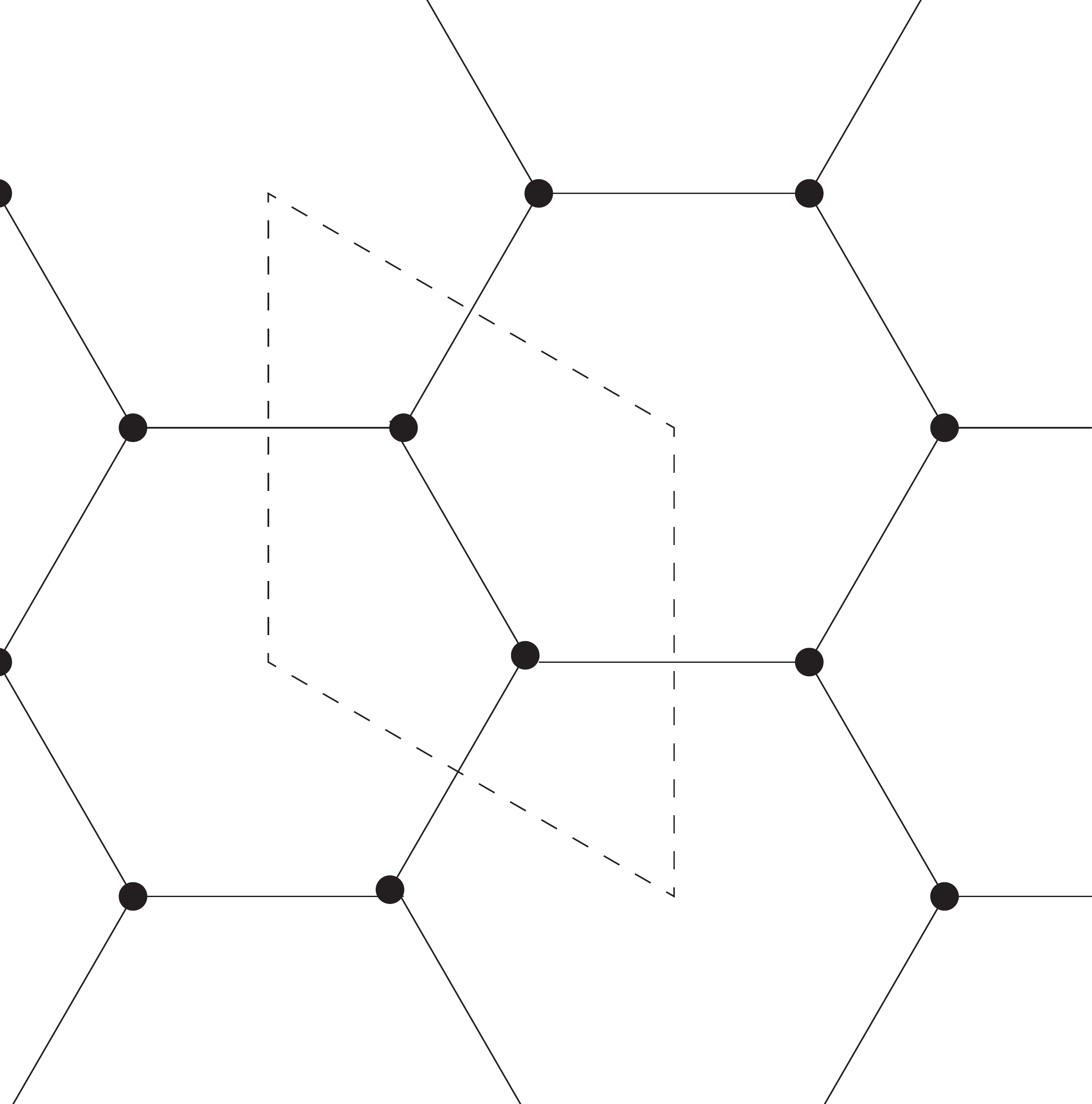}
	\end{minipage}
\begin{minipage}{.5\textwidth}
	\centering
	\includegraphics[width=1.2in]{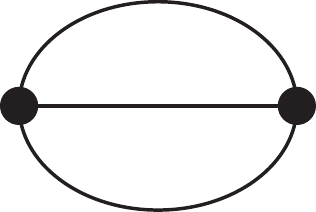}
	\end{minipage}
	\caption{Doubly periodic graphs and their quotients}
	\label{fig: symm quot}
\end{figure}

Next, we define the genus of periodic skeletal graphs. 

\begin{defn}\label{defn: genus graph} Given a graph $\Gamma$ that has $\textrm{v}$ 0-simplices and $\textrm{e}$ 1-simplices, we define the genus of $\Gamma$ as $\textrm{genus}(\Gamma) := \textrm{e} - \textrm{v} + 1.$ 
\end{defn}

Given a skeletal graph $\Gamma$ in $\R^3,$ we let $\mathcal{N}_{\delta}(\Gamma)$ be the $\delta$-tubular neighborhood of $\Gamma$ for small enough $\delta > 0$ so that $\mathcal{N}_{\delta}(\Gamma)$ embeds in $\R^3.$ Then $\partial \mathcal{N}_{\delta}(v)$ is isotopic to a sphere with $d$-holes and $\partial \mathcal{N}_{\delta}(e)$ is isotopic to a cylinder. The Euler characteristic of a $d$-punctured sphere is $2 - d$ and the Euler characteristic of a cylinder is zero. The genus of $\Gamma$ is defined as the genus of $\partial \mathcal{N}_{\delta}(\Gamma).$ For a $d$-regular periodic skeletal graph, we consider its quotient $\Gamma' = \{V', E'\}$ and denote by $\textrm{v} = |V'|.$ Then $\textrm{e} := |E'| = \frac{\textrm{v} d}{2}.$ Then, \begin{equation}\label{eqn: graph genus} 2 - 2 \cdot \textrm{genus}(\Gamma) = \textrm{v} (2 - d) + \frac{\textrm{v} d}{2} \cdot 0.\end{equation} We label each quotient graph by $\Gamma_{g, \textrm{v}, d}'.$ \\

The following proposition proves that given $g > 1,$ there are finitely many symmetric quotient graphs of genus $g.$ 

\begin{prop}\label{prop: fin quot graph} Let $\Gamma$ be a periodic skeletal graph of genus $> 1$ whose quotient is symmetric. Then, the degree of the graph is $\geq 3.$ Furthermore, for a given genus $> 1,$ there are finitely many symmetric quotient graphs of the type $\Gamma_{g, \textrm{v}, d}'.$
\end{prop}

\begin{proof} We refer to (Equation~\ref{eqn: graph genus}). For any given $g > 1,$ $2 g - 2$ has finitely many divisors. Hence, there are only finitely many quotient graphs that can be labeled as $\Gamma_{g, \textrm{v}, d}'$ for a fixed $g.$ If $\textrm{v} = 1,$ then $d = 2 g,$ so the quotient graph has one vertex and $g$ loops. If $\textrm{v} \geq 2,$ then due to connectivity and symmetry, we can only allow multiple edges on the quotient graph. Then, $\textrm{e} = g + \textrm{v} + 1$ and $d = \frac{2 g - 2}{\textrm{v}} + 2,$ given that this number is an integer. Therefore, $d \geq 3.$
\end{proof}

The following shows all symmetric quotient graphs of genus three and four. If $g = 3,$ then $4 = \textrm{v} (d - 2),$ hence $(\textrm{v}, d)$ is $(1, 6), (2, 4),$ or $(4, 3).$ If $g = 4,$ then $6 = \textrm{v} (d - 2),$ so $(\textrm{v}, d)$ is $(1, 8), (2, 5), (3, 4),$ or $(6, 3).$

\begin{table}[htbp]
\centering
\begin{tabular}[h]{|c|c|c|c|c|c|c|c|}
\hline
$g$ & $\textrm{v}$ & $d$ & & $g$ & $\textrm{v}$ & $d$ & \\ \hline
3 & 1 & 6 & \includegraphics[height=0.4in]{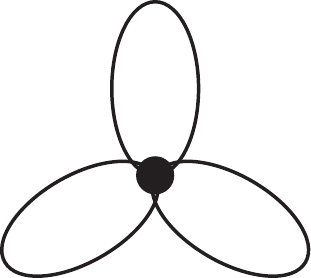} & 4 & 1 & 8 & \includegraphics[height=0.4in]{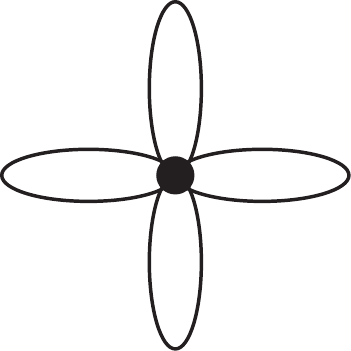} \\
3 & 2 & 4 & \includegraphics[height=0.4in]{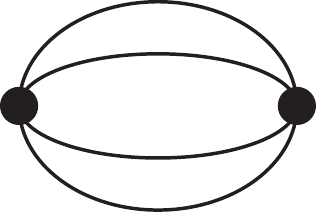} & 4 & 2 & 5 & \includegraphics[height=0.4in]{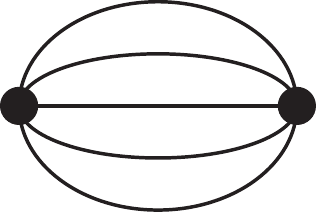} \\
3 & 4 & 3 & \includegraphics[height=0.4in]{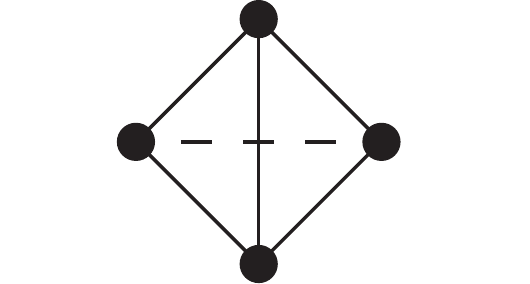} & 4 & 3 & 4 & \includegraphics[height=0.4in]{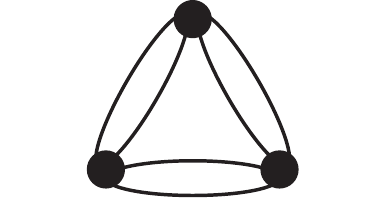} \\
& & & & 4 & 6 & 3 & \includegraphics[height=0.4in]{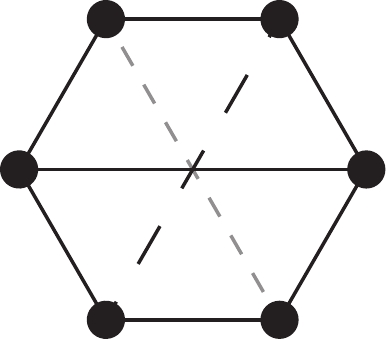} \\ \hline
\end{tabular}
	\label{tab: symmetric quotient graphs}
	\caption{Symmetric quotient graphs of $\textrm{genus}(\Gamma) = 3, 4$}
\end{table}

Although we require the skeletal graphs $\Gamma$ to be simple, we do not impose this on their quotient graphs. 

\subsection{Classification of regular triply periodic polyhedral surfaces}
\label{sec: RAD}

In Example~\ref{exmp: rad octa-4}, we showed that a triply periodic polyhedron can be retracted to a triply periodic skeletal graph in $\R^3.$ In this section, we will define a \textit{decoration} of a skeletal graph as a way to associate a polyhedral surface. Given a skeletal graph, we will regard its tubular neighborhood that consists of solids.

\begin{defn} \label{defn: deco} Given a skeletal graph $\Gamma$ embedded in $\R^3,$ a \textit{decoration} of $\Gamma$ is a polyhedron achieved by replacing the 0-simplices and 1-simplices with convex polyhedral solids (including the empty solid) so that i) there is a deformation retract of the polyhedron to the skeletal graph and ii) the polyhedra are identified only along faces. In essence, if a 0-simplex and a 1-simplex in $\Gamma$ are incident, then their corresponding replacement solids in the decoration are identified along a face. A \textit{regular decoration} of $\Gamma$ is a decoration whose boundary surface is regular. That is, the surface is tiled by regular $p$-gons with the same valency at every vertex. Moreover, the solids must be adjacent to only one type of polygon. A regular decoration is denoted by $\{p, q| r\}$ if the boundary surface is tiled by regular $p$-gons, $q$ at each vertex, and the intersection of adjacent solids is a regular $r$-gon. The notation $\{p, q|r\}$ is called the Schl{\"a}fli symbols.
\end{defn} 

In Proposition~\ref{prop: fin quot graph} we showed that given a periodic skeletal graph of $g > 1$ whose quotient is symmetric, the degree of the graph is at least three. Note that each 0-simplex is incident to $d$ 1-simplices and each 1-simplex is incident to two 0-simplices. Hence, if a polyhedron is adjacent to three or more polyhedra, then it retracts to its center as a 0-simplex. On the other hand, if a polyhedron is adjacent to only two polyhedra, then it retracts to a 1-simplex that is incident to the centers of the two adjacent polyhedra.\\

The next definition puts restrictions on the solids that replace the 0- and 1-simplices.

\begin{defn}
An \textit{Archimedean decoration} of $\Gamma$ is a decoration where only Platonic solids and Archimedean solids are allowed to replace the 0-simplices. For the 1-simplices, we only allow prisms and anti-prisms over regular polygons, and empty solids as replacement solids.
\end{defn}

If two solids that are adjacent to each other both retract to 0-simplices, we allow an empty solid to replace the 1-simplex that is incident to the 0-simplices.

\begin{exmp} For example, the Octa-4 surface (Example \ref{exmp: rad octa-4}) is a regular Archimedean decoration of $\Gamma_{3, 2, 4}$ where the 0-simplices are replaced with regular octahedra and the 1-simplices are replaced with triangular anti-prisms (which are essentially octahedra). The surface is denoted by $\{3, 8| 3\}$ in Schl{\"a}fli symbols. Also, the Mucube is a regular Archimedean decoration of $\Gamma_{3, 1, 6}'$ where the 0-simplex is replaced with a cube and the 1-simplices are replaced with square prisms (that is, cubes). The boundary surface is tiled by squares, six around each vertex, and the solids are only adjacent along squares, hence the Mucube is denoted as $\{4, 6| 4\}$ in Schl{\"a}fli symbols. The Muoctahedron is a regular Archimedean decoration of $\Gamma_{3, 1, 6}'$ where the 0-simplex is replaced with a truncated octahedron and the 1-simplices are replaced with empty solids. In other words, the truncated octahedra are adjacent to each other along their square faces. The surfaces is tiled by hexagons, four around each vertex, therefore denoted as $\{6, 4| 4\}.$ Lastly, the Mutetrahedron is a regular Archimedean decoration of $\Gamma_{3, 2, 4}'$ where one of the 0-simplices is replaced with a tetrahedron, the other is replaced with a truncated tetrahedron, and the 1-simplices are replaced with empty solids. In other words, the tetrahedra and truncated tetrahedra are adjacent to each other along the triangular faces. Its boundary surface is tiled by hexagons, three around each vertex, hence denoted by $\{6, 3| 3\}.$\\
\end{exmp}

In the remainder of the section, we will prove that there are only finitely many regular Archimedean decorations for a given skeletal graph. In Lemma \ref{lem: fin rad}, we will show all possible combinations of solids that we can use to decorate a skeletal graph. Then in Theorem \ref{thm: finite decorations3} and Theorem \ref{thm: finite decorations4}, we will show all regular polyhedral surfaces of genus three and four, respectively, that arise from a regular Archimedean decoration of a symmetric quotient graph.

\begin{lemma} \label{lem: fin rad} Given a skeletal graph $\Gamma \subset \R^3,$ there are at most finitely many regular Archimedean decorations of $\Gamma.$
\end{lemma}

\begin{proof}
Recall that a decoration of a skeletal graph is a replacement of the 0- and 1-simplices with convex polyhedra. We will show that there are at most finitely many possible combinations of solids that we can use to replace the 0- and 1-simplices. One way to decorate a graph is to replace the 0-simplices with one type of Platonic or Archimedean solid and the 1-simplices with one type of prism or anti-prism. The Octa-4 surface and the Mucube arise from such construction. Secondly, we can replace the 0-simplices with one type of Platonic or Archimedean solid and the 1-simplices with the empty solid. An example is the Muoctahedra. Lastly, we replace the 0-simplices by different types of Platonic or Archimedean solids. This forces us to replace the 1-simplices with the empty solid, otherwise the boundary surface cannot be regular. The Mutetrahedra is constructed this way. 
\end{proof}

Combining Proposition \ref{prop: fin quot graph} and Lemma \ref{lem: fin rad}, in the following theorems we show all regular Archimedean decorations of symmetric quotient graphs of genus three and four.

\begin{thm} \label{thm: finite decorations3} There exist only finitely many genus three regular triply periodic polyhedral surfaces that are regular Archimedean decorations of $\Gamma_{3, 1, 6}'$ and $\Gamma_{3, 2, 4}'.$
\end{thm}

\begin{proof} We will prove this theorem by case distinction. Given a symmetric quotient graph (Proposition \ref{prop: fin quot graph}), we sort out the Platonic and Archimedean solids that can replace the 0-simplices based on the degree of the graph. The solids must have $d$ faces placed in a symmetric way in $\R^3$ so that we can replace the 0-simplices of a degree $d$ graph. For this reason, we do not consider the quotient graph $\Gamma_{3, 4, 3}'$ as we cannot choose three faces from a Platonic or Achimedean solid in a symmetric way. Then, for each solid, depending on the remaining faces, we select the solids that can replace the 1-simplices of the graph. That is, we consider prisms if the remaining faces are squares and anti-prisms if the remaining faces are triangles.\\

We begin with the quotient graph $\Gamma_{3, 1, 6}'.$ Since $d = 6,$ the candidates that can replace the 0-simplex are a cube, octahedron, cuboctahedron (six squares and eight triangles), truncated cube (six octagons and eight triangles), truncated octahedron (six squares and eight hexagons), and a snub cube. Since $\textrm{v} = 1,$ we must choose the six $r$-gonal faces as three pairs of opposite faces. The three parallel translations will form a lattice of translations in $\R^3.$ \begin{itemize}\item If the six $r$-gonal faces are placed as the bases of a prism and the remaining faces are squares, then we consider $r$-gonal prisms to replace the 1-simplices. \textbf{A cube with six square prisms} (the Mucube) arises this way. \item If the $r$-gonal faces are placed as the bases of a prism but the rest of the faces are not squares, then we consider empty solids to replace the 1-simplices. \textbf{A truncated octahedron} (Muoctahedron) arises this way. Other cases are a cuboctahedron and a truncated cube. However, in both cases the boundaries of the decorations are disconnected. \item If the $r$-gonal faces are placed as the bases of an anti-prism involving a $\frac{\pi}{r}$-rotation and the rest of the faces are triangles, then we consider $r$-gonal anti-prisms to replace the 1-simplices. \textbf{An octahedron with six triangular anti-prisms} arises this way. \end{itemize}

Lastly, for the rest of the cases we consider empty solids. A snubcube is an Archimedean solid that has chirality, hence it cannot replace a 0-simplex of a quotient graph that has only one 0-simplex \footnote{However, it can replace a 0-simplex of a quotient graph if $\textrm{v} = 2$ and $d = 6.$ By (\ref{eqn: graph genus}), this is the quotient graph $\Gamma_{5, 2, 6}'.$}. \\

Next we look at $\Gamma_{3, 2, 4}'.$ As $d = 4,$ the candidates that can replace the 0-simplices are a tetrahedron, cube, octahedron, icosahedron, and a truncated tetrahedron. We claim that there are two distinct ways to choose four faces from the solids. \begin{itemize}\item One way is to observe the tetrahedral symmetry, where the corresponding solids are the tetrahedron, octahedron, icosahedron, and truncated tetrahedron. After we choose four faces from an octahedron or icosahedron by tetrahedral symmetry, the remaining faces are triangles, hence we replace the 1-simplices with triangular anti-prisms and achieve \textbf{two octahedra with four triangular anti-prisms} (Octa-4) and \textbf{two icosahedra with four triangular anti-prisms.} We can also inscribe a tetrahedron in a cube as we do an octahedron in \cite{lee2017triply}, hence we have a decoration by \textbf{two tetrahedra with triangular anti-prisms.} 

Once we choose four triangular faces from a truncated tetrahedron, the remaining faces are hexagons, hence the 1-simplices must be replaced with empty solids. Instead of replacing both 0-simplices with truncated tetrahedra, we replace one with a tetrahedron. \textbf{A truncated tetrahedron and tetrahedron with empty solids} yields the Mutetrahedron. On the other hand, if we choose four hexagonal faces from a truncated tetrahedron, the remaining faces are triangles, hence one can attempt to replace the 1-simplices with hexagonal anti-prisms. However, this does not yield an embedded surface in $\R^3.$ \item Another way of choosing four faces from a solid is to choose two pairs of opposite faces. Then we have two possibilities: \textbf{two cubes with four prisms} and \textbf{two octahedra with four triangular anti-prisms} (Figure~\ref{fig: 324}). 

\begin{figure}[htbp] 
\centering
\begin{minipage}{.5\textwidth}
	\centering
	\includegraphics[width=3in]{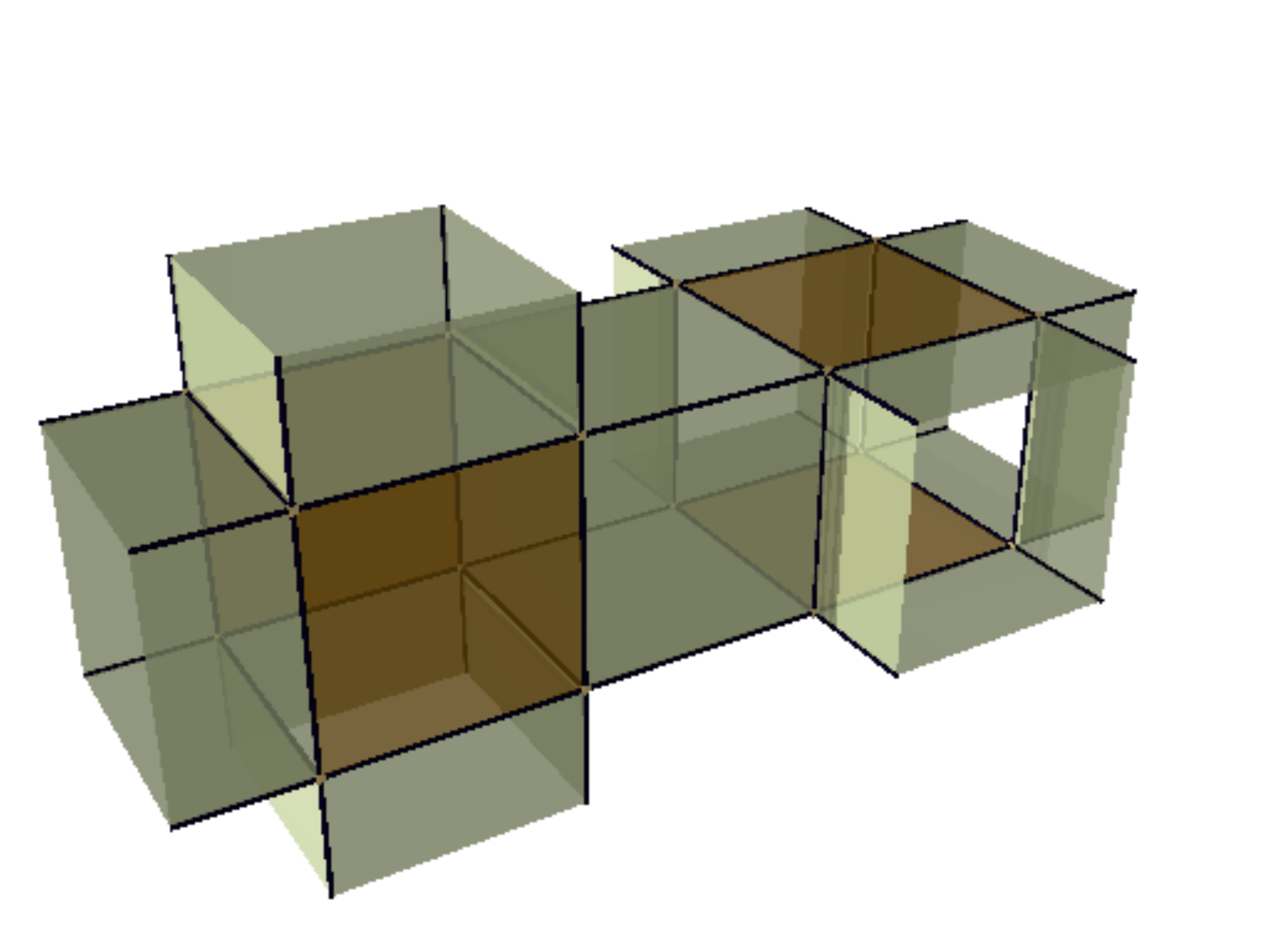} 
\end{minipage}%
\begin{minipage}{.5\textwidth}
	\centering
	\includegraphics[width=3in]{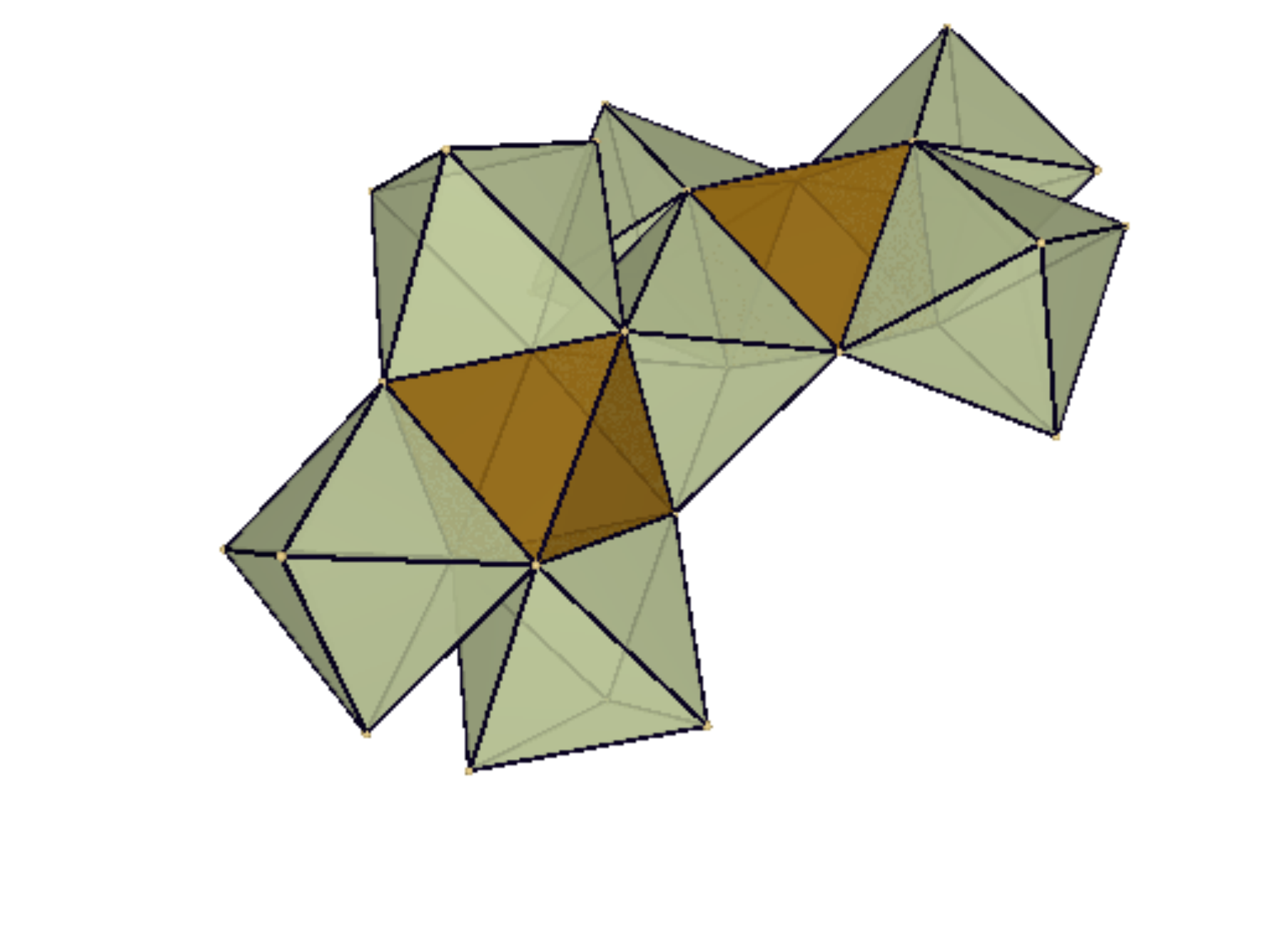} 
\end{minipage}
	\caption{Regular Archimedean decorations of $\Gamma_{3, 2, 4}'$ by cubes and octahedra, respectively}
	\label{fig: 324}
\end{figure}

\end{itemize}
\end{proof}

\begin{thm} \label{thm: finite decorations4} There exist only finitely many genus four regular triply periodic polyhedral surfaces that are regular Archimedean decorations of $\Gamma_{4, 1, 8}'$ and $\Gamma_{4, 3, 4}'.$
\end{thm}

\begin{proof} The proof of this theorem will follow the proof of Theorem \ref{thm: finite decorations3}. Proposition \ref{prop: fin quot graph} yields symmetric quotient graphs $\Gamma_{4, 1, 8}',$ $\Gamma_{4, 3, 4}',$ $\Gamma_{4, 2, 5},$and $\Gamma_{4, 6, 3}',$ but we will will not consider the latter two as we cannot choose five faces (or three, respectively) from a Platonic or Achimedean solid in a symmetric way in $\R^3.$\\

First we look at $\Gamma_{4, 1, 8}'.$ Since $d = 8,$ we seek solids that have eight faces placed in a symmetric way in $\R^3.$ These solids are an octahedron, icosahedron, cuboctahedron, truncated cube, truncated octahedron, and a rhombicuboctahedron. We seek four pairs of opposite $r$-gonal faces. \begin{itemize} \item We choose eight faces from an octahedron (or icosahedron) so that when an octahedron (or icosahedron) is inscribed in a regular cube, the eight faces ``face'' the vertices of the cube. This yields \textbf{an octahedron with eight triangular anti-prisms} and \textbf{an icosahedron with eight triangular anti-prisms.} \item For the cuboctahedron, rhombicuboctahedron (eight triangles and eighteen squares), and truncated cube, we choose eight faces as four pairs of opposite triangular faces. These faces are placed as bases of an anti-prism involving a $\frac{\pi}{3}$-rotation. As there is only one 0-simplex in this quotient graph, we cannot replace the 1-simplices with prisms nor empty solids. On the other hand, the remaining faces are not triangular faces, which implies that we cannot replace the 1-simplices with anti-prisms either. \item After choosing four pairs of opposite faces from a truncated octahedron (six squares and eight hexagons), we are left with parallel square faces, hence we can replace the 1-simplices with hexagonal prisms or empty solids. Replacing the 1-simplices with empty solids yields a disconnected boundary surface, however \textbf{a truncated octahedron with hexagonal prisms} (Figure \ref{fig: truncated octa8_poly}) yields a triply periodic polyhedral surface. 

\begin{figure}[htbp] 
   \centering
	\includegraphics[width=3in]{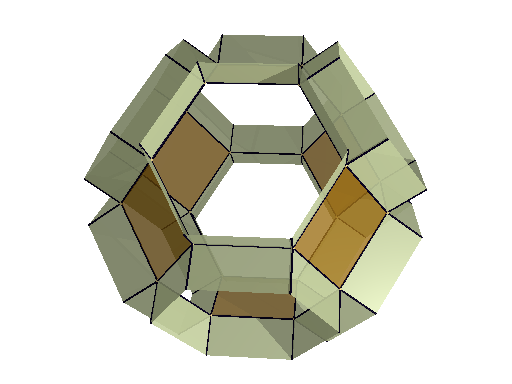}
	\caption{Fundamental piece of the Truncated Octa-8 surface}
	\label{fig: truncated octa8_poly}
\end{figure}
\end{itemize}

Lastly, we look at $\Gamma_{4, 3, 4}'.$ We mimic the argument from the decorations of $\Gamma_{3, 2, 4}'$ from the previous theorem. Then, \textbf{three cubes with four square prisms} and \textbf{three octahedra with four triangular anti-prisms} yield triply periodic polyhedral surfaces (Figure \ref{fig: 434}).

\begin{figure}[htbp] 
\centering
\begin{minipage}{.5\textwidth}
	\centering
	\includegraphics[width=3in]{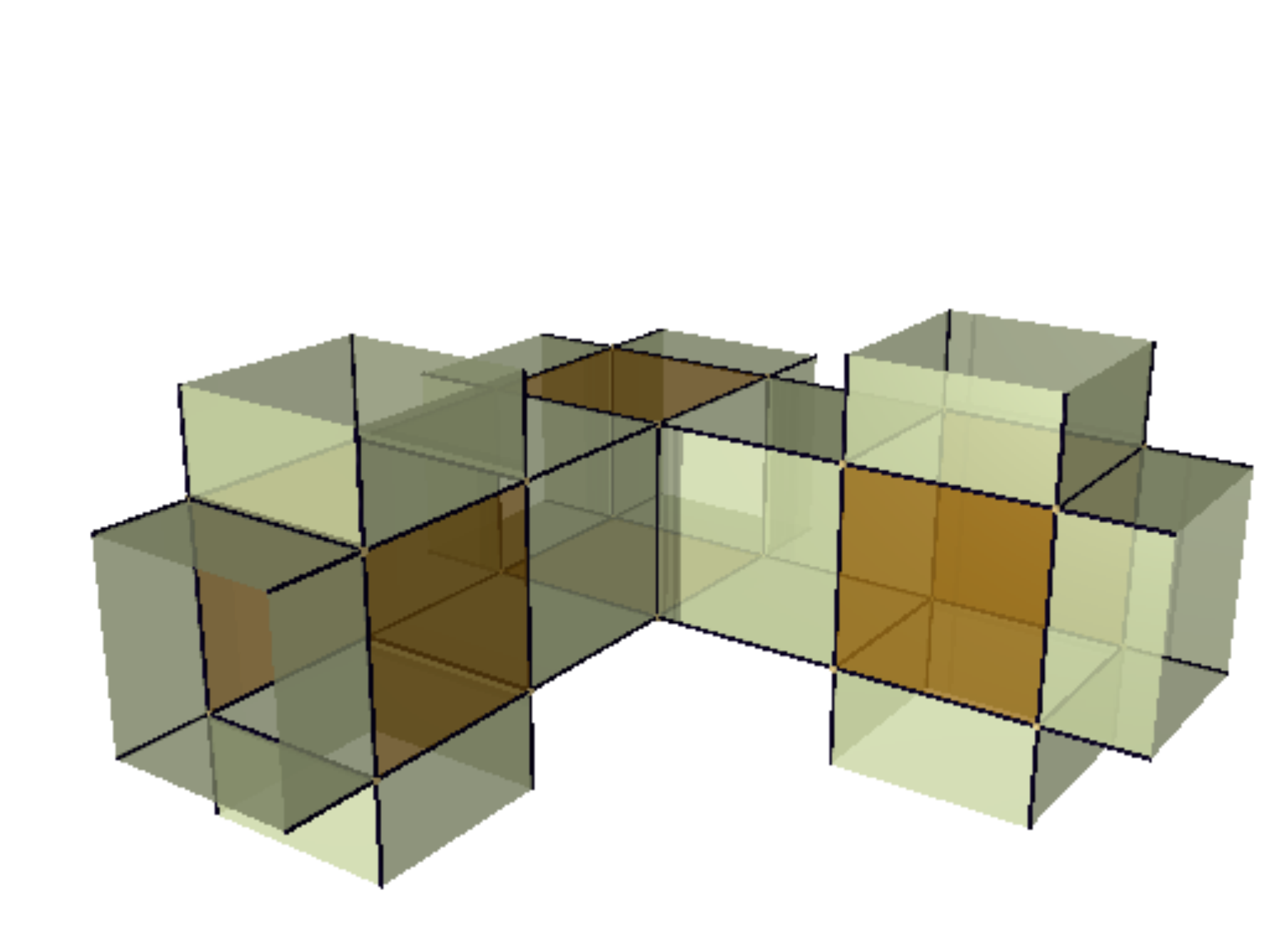} 
\end{minipage}%
\begin{minipage}{.5\textwidth}
	\centering
	\includegraphics[width=3in]{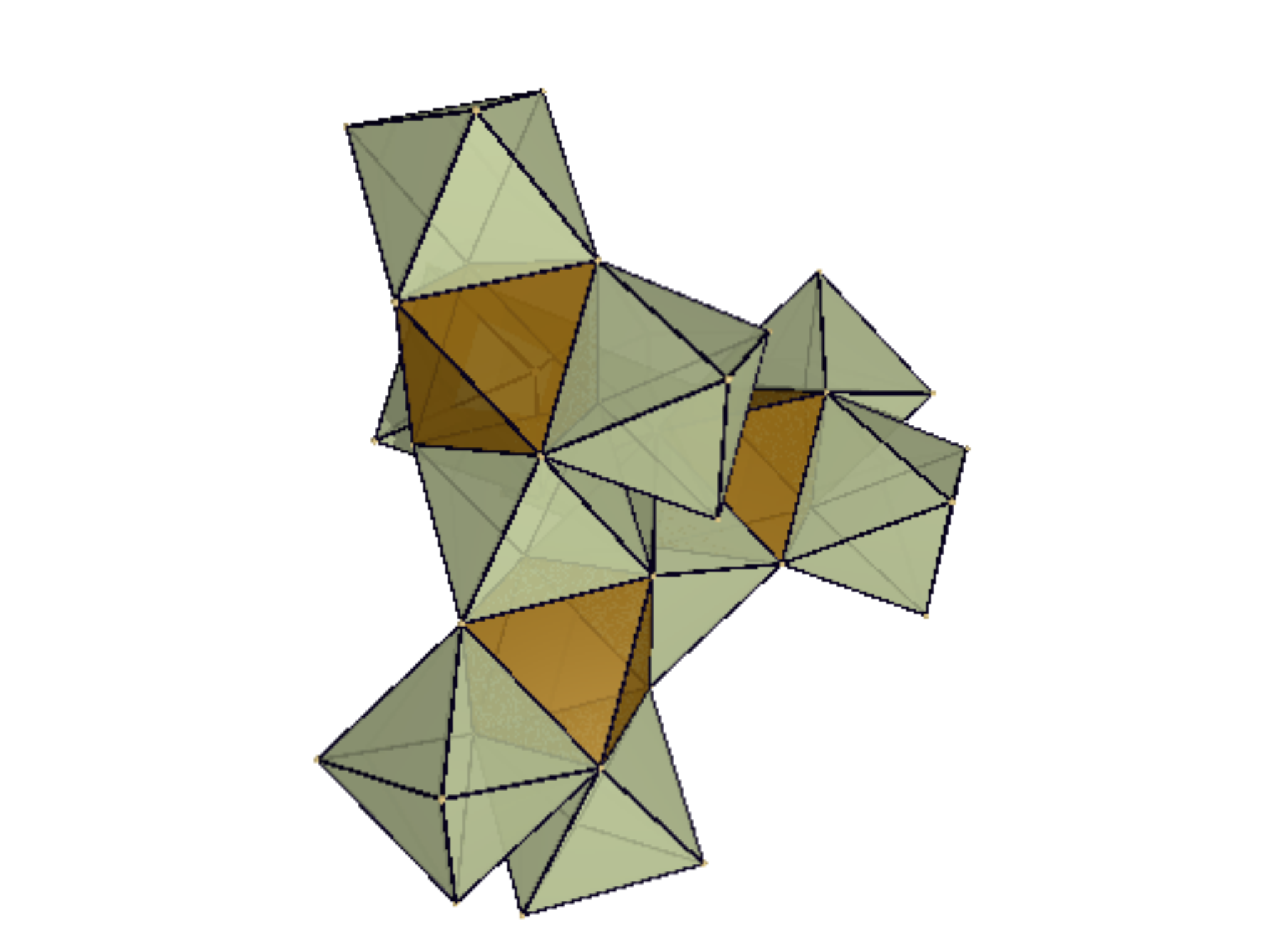} 
\end{minipage}
	\caption{Regular Archimedean decorations of $\Gamma_{4, 3, 4}'$ by cubes and octahedra, respectively}
	\label{fig: 434}
\end{figure}

\end{proof}

\section[Conformal Structures of Regular Triply Periodic Polyhedral Surfaces]{Conformal Structures of Regular Triply Periodic Polyhedral Surfaces as Cyclically Branched Covers Over Spheres}
\label{ch5 main result}
In this section, we will study the conformal structures of the triply periodic polyhedral surfaces we found in Theorem~\ref{thm: finite decorations3} and Theorem~\ref{thm: finite decorations4}. Since the surfaces are triply periodic, we look at their compact quotients which carry the same conformal type of the polyhedral surface. Since the polyhedral surfaces are regular (Definition~\ref{defn: reg}), they carry many symmetries on the underlying surfaces that are intrinsic to the surface and independent of the embedding in $\R^3.$ To identify the conformal structure of a surface, we need a large enough symmetry group so that the quotient is conformally rigid. The simplest case is where such group is a cyclic group. We are particularly interested in surfaces whose quotient via a cyclic group is a sphere. \\

With tools developed from Section~\ref{ch3 cyc br cov}, we will determine the conformal structure on the underlying surfaces that are induced from the cone metrics. In Subsection~\ref{sec: three mu}, we look at hyperbolic structures on the Mucube, Muoctahedron, and Mutetrahedron. We will show that all three surfaces are biholomorphic to each other by analysing their cone metrics. We will show that the surfaces have conformally equivalent descriptions of the genus three Schwarz minimal P- and D-surfaces. In Subsection~\ref{sec: (1,4,7)}, we will look at a regular Archimedean decoration of $\Gamma_{4, 1, 8}'$ by an octahedron and eight triangular anti-prisms. Its cone metric shows that the compact quotient is a twelvefold cyclically branched cover over a thrice punctured sphere. We will show that its conformal structure is equivalent to that of the genus four Schoen's minimal I-WP surface. In Subsection~\ref{sec: (1,2,4,3)}, we will look at another example of a regular Archimedean decoration of $\Gamma_{4, 1, 8}'$ by a truncated octahedron and hexagonal prisms. We show that the compact quotient is a fivefold cyclically branched cover over a 4-punctured sphere. We also show that the surface has the same conformal description as Kepler's small stellated dodecahedron and Bring's curve \cite{weber2005kepler}. In the last section, we will revisit the Octa-4 surface. As we already determined the conformal structure throughout Section~\ref{ch2 octa-4} and Section~\ref{ch3 cyc br cov}, we will take an excursion from the main discussion. As pointed out earlier, the surface cannot be conformally equivalent to any minimal surface. However, we find a harmonic map on the surface that has the same period matrix as the Octa-4 surface. In other words, we find a smooth surface that has the same conformal type as the Octa-4 surface. 

\subsection{Mucube, Muoctahedron, and Mutetrahedron}
\label{sec: three mu}
In this section, we will study the three regular triply periodic polyhedral surfaces that appear in \cite{coxeter1937regular}. First we will look at their hyperbolic structures and show that all three surfaces have the same hyperbolic description. Due to various rotational symmetries on the underlying surface, we will show that the surface has multiple descriptions as a cyclically branched covering over a punctured sphere. Then, we will find the conformal structure that is induced from the cone metrics. With admissible cone metrics, we find a basis of holomorphic 1-forms on the surface and find an algebraic description of the surface. We will locate all Weierstrass points on the underlying surface as well as on the polyhedral surfaces. Lastly, we will show that the surfaces can be immersed into $\R^3$ as triply periodic minimal surfaces.\\

Here we will find the hyperbolic tessellation for each of the Mucube, Muoctahedron, and Mutetrahedron. We will focus on the symmetries on the surfaces and show that all three have the same hyperbolic description. We begin with the Mucube which is denoted by $\{4, 6| 4\}$ in Schl{\"a}fli symbols. Its fundamental piece (quotient via the lattice of translations) is tiled by twelve squares. We map a Euclidean square to a hyperbolic $\frac{\pi}{3}$-square and by Schwarz reflection principle obtain the following hyperbolic tessellation (Figure~\ref{fig: 464}). As we used Petrie polygons (Definition~\ref{defn: petrie}) to find identification of edges in Figure~\ref{hyperbolic_tiling} for the Octa-4 surface, we look at polyhedral geodesics for the Mucube. Consider any surface tiled by regular $p$-gons where $p$ is even, either the polyhedral surface embedded in Euclidean space or the hyperbolic setting. We define a polyhedral geodesic as a geodesic that goes through opposite sides of a $p$-gon. Then, all polyhedral geodesics on the Mucube, in either Euclidean or hyperbolic tiling, are closed after going through four squares.

\begin{figure}[htbp] 
   \centering
   \includegraphics[width=2in]{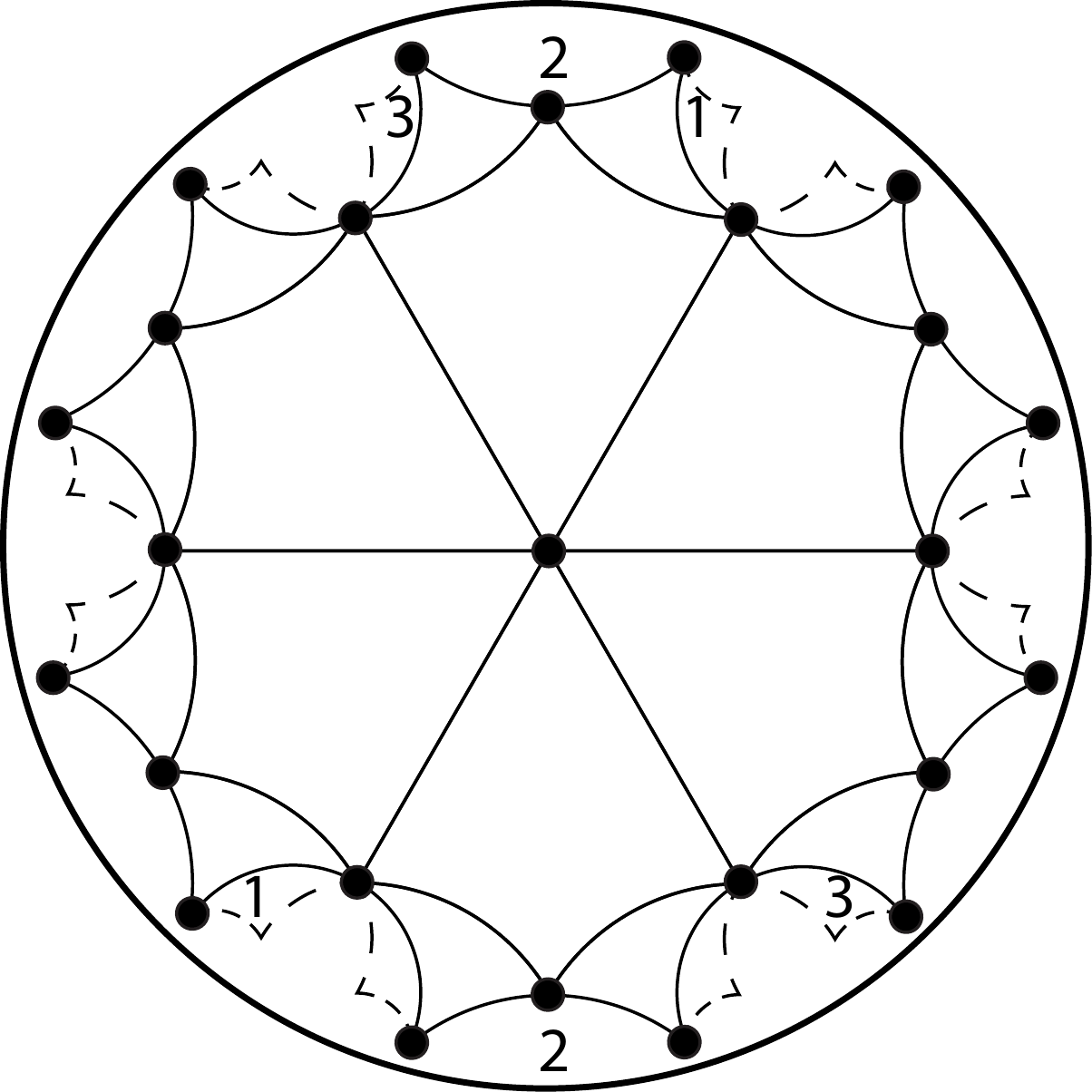} 
	\caption{Hyperbolic description of the Mucube}
	\label{fig: 464}
\end{figure}

Given a polygonal tiling, its dual tiling is formed by taking the center of each polygon as a vertex and joining the centers of polygons that share an edge. As the Muoctahedron is denoted by $\{6, 4| 4\},$ the hexagonal tiling of the Muoctahedron is dual to the square tiling of the Mucube. We map the Euclidean hexagons to hyperbolic $\frac{\pi}{2}$-hexagons. Figure~\ref{fig: 644} shows the hyperbolic description of the Muoctahedron, where the identification of edges is identical to that of the Mucube.

\begin{figure}[htbp] 
   \centering
   \includegraphics[width=2in]{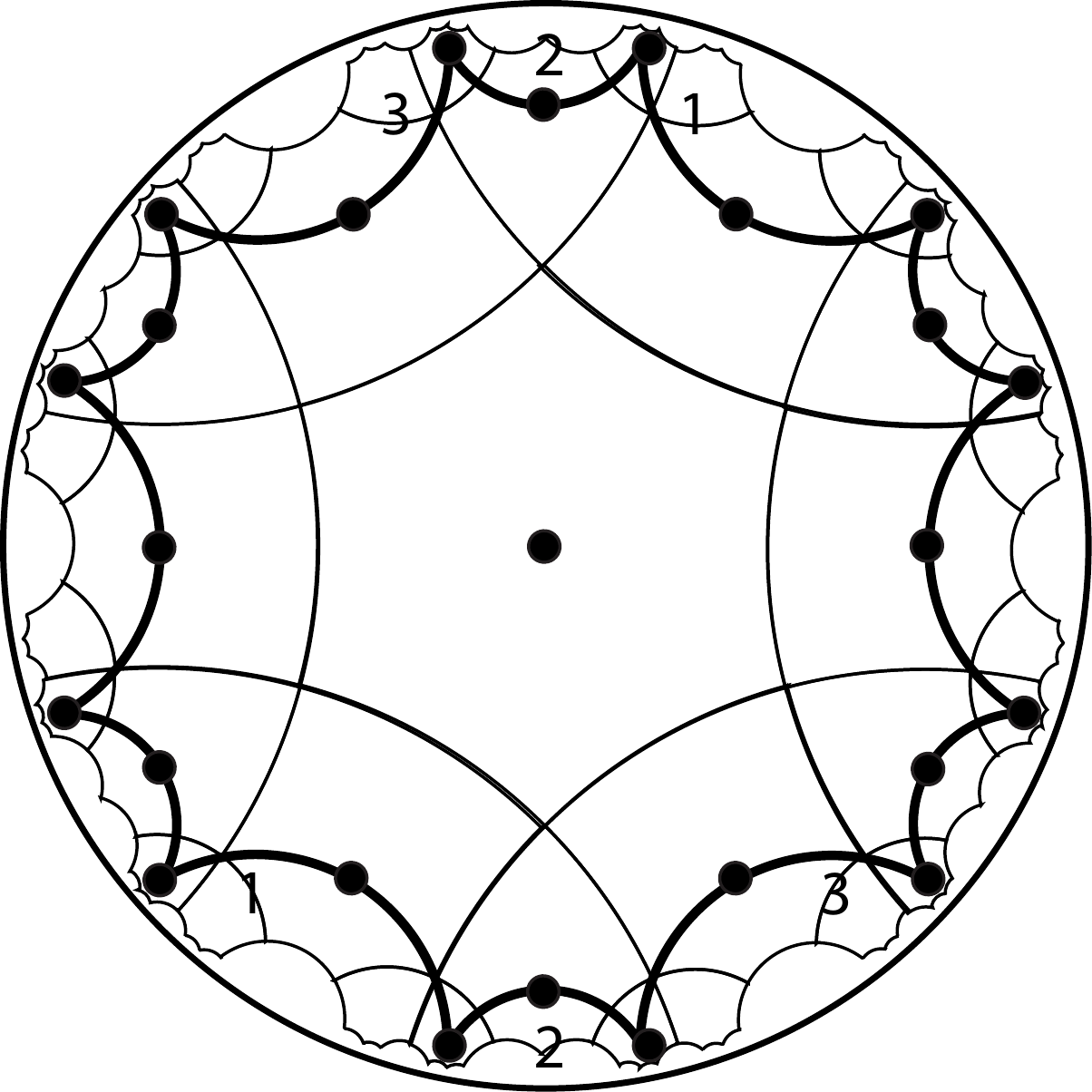} 
	\caption{Hyperbolic description of the Muoctahedron}
	\label{fig: 644}
\end{figure}

The Mutetrahedron, denoted by $\{6, 6| 3\},$ is also tiled by hexagons. The fundamental piece consists of four hexagons that are twice the size of the hexagons on the Muoctahedron. The identification of edges suggests that this has the same hyperbolic description as the two surfaces above. 

\begin{figure}[htbp] 
   \centering
   \includegraphics[width=2in]{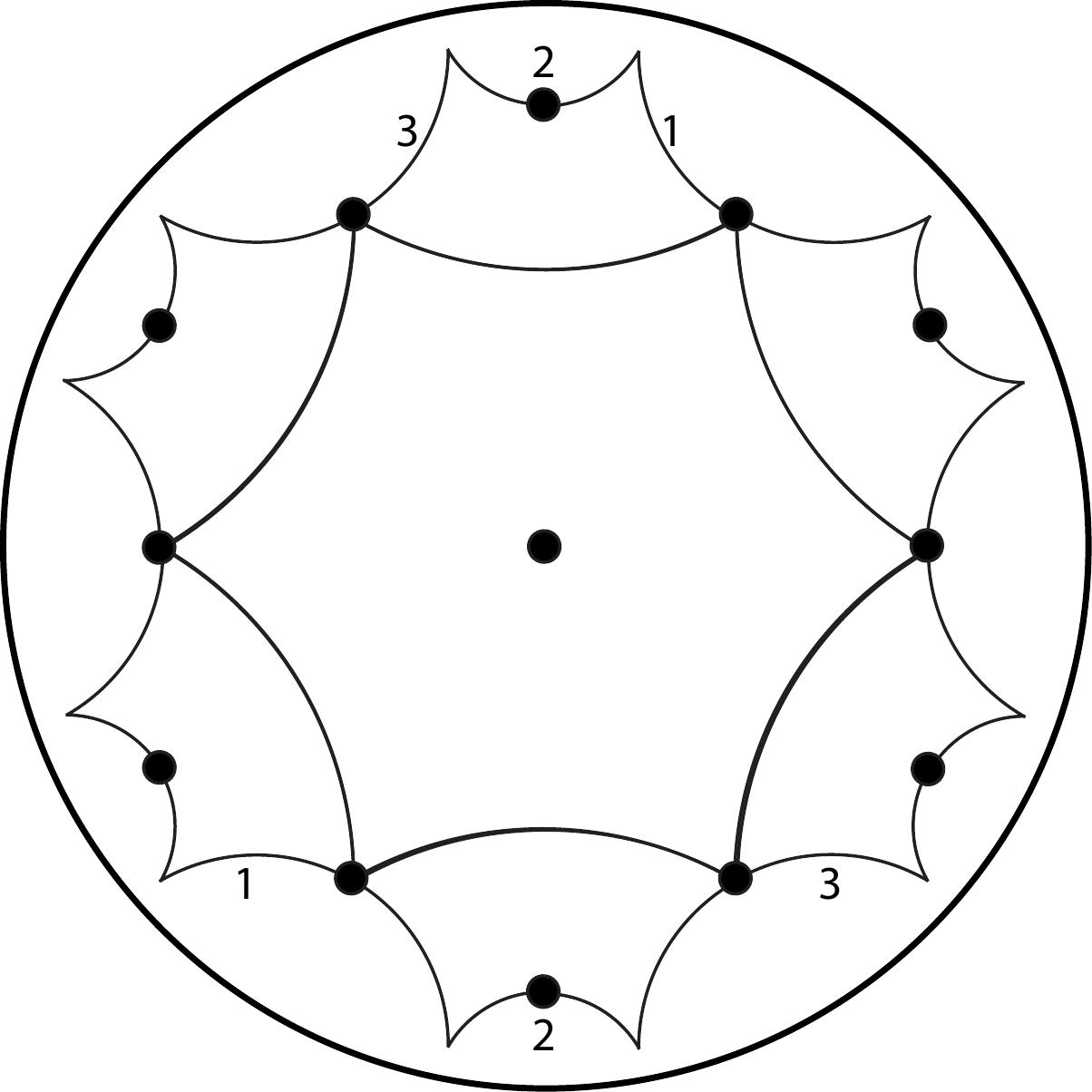} 
	\caption{Hyperbolic description of the Mutetrahedron}
	\label{fig: 663}
\end{figure}

The previous figures suggest that the hyperbolic geodesics are preserved under an order six rotational symmetry. We take the quotient of the hyperbolic fundamental domain by the cyclic group of order six. Then, the quotient is a doubled hyperbolic rhombus of angles $\left( \frac{\pi}{6}, \frac{\pi}{2}, \frac{\pi}{6}, \frac{\pi}{2} \right),$ which is topologically a sphere. We label the center of the tessellation as $p_1,$ and using the same argument as in Theorem~\ref{thm: 125}, we have $d_i = 1.$ We label the remaining vertices of the quotient rhombus as $p_2,$ $p_3,$ and $p_4$ counterclockwise. There exist one $\widetilde{p_3}$ and three of each $\widetilde{p_2}$ and $\widetilde{p_4},$ hence the branching indices $(d_i)$ satisfy $\gcd(6, d_3) = 1$ and $\gcd(6, d_2) = \gcd(6, d_4) = 3.$ Therefore, the underlying surface of the Mucube (equivalently the Muoctahedron or Mutetrahedron) is a sixfold cyclically branched cover over a 4-punctured sphere with branching indices $(d_i) = (1, 3, 5, 3).$ The admissible cone metrics $\frac{2 \pi }{6} (1, 3, 5, 3), \frac{2 \pi }{6} (3, 3, 3, 3),$ and $\frac{2 \pi }{6} (5, 3, 1, 3)$ arise from the multipliers from which we get a basis of holomorphic 1-forms $\{\omega_i\}$ where the divisors are $(\omega_1) = 4 \widetilde{p_3}, (\omega_2) = 2 \widetilde{p_1} + 2 \widetilde{p_3},$ and $(\omega_3) = 4 \widetilde{p_1}.$ Given an explicit basis of 1-forms, we have an algebraic equation $\omega_1 \omega_3 = \omega_2^2$ that describes the abstract surface. By Corollary 1.5 of Chapter 3 from \cite{miranda1995algebraic}, a nonsingular conic (curve of degree two) is isomorphic to $\mathbb{P}^1.$ This shows that the surface is hyperelliptic.\\

We will show that the surface can also be described as a fourfold cyclically branched covering over a sphere. The hyperbolic hexadecagon in Figure~\ref{fig: 464_} represents the fundamental piece of the Mucube where a square face is at the center of the tessellation. We find the identification of the edges using polyhedral geodesics and observe that an order four rotational symmetry preserves the geodesics. We take the quotient of the hyperbolic fundamental domain by the cyclic group of order four. This yields a doubled hyperbolic rhombus of angles $\left( \frac{\pi}{4}, \frac{3 \pi}{4}, \frac{3 \pi}{4}, \frac{\pi}{4} \right)$ which is topologically a sphere. We label the vertices of the rhombus as $p_i$ with $p_1$ at the center of the tessellation so that $d_1 = 1.$ The cone angles at the cone points imply that the abstract surface is a fourfold cyclically branched cover over a 4-punctured sphere with branching indices $(d_1, d_2, d_3, d_4)$ such that $\gcd(4, d_i) = 1$ for each $i$ and that $\sum d_i \equiv 0 \pmod 4.$ Then the branching indices must be either $(1, 1, 1, 1)$ or $(1, 1, 3, 3).$ Since we already know that the surface is hyperelliptic, we claim that $(1, 1, 1, 1)$ is not of our desire by showing that such covering is not hyperelliptic. If $(d_1, d_2, d_3, d_4) = (1, 1, 1, 1),$ then we get a basis of holomorphic 1-forms with $(\omega_1) = 4 \widetilde{p_2},$ $(\omega_2) = \widetilde{p_1} + \widetilde{p_2} + \widetilde{p_3} + \widetilde{p_4},$ and $(\omega_3) = 4 \widetilde{p_1}.$ We refer to the discussion in Subsection~\ref{sec: wronski}. Since $\textrm{wt}_1 = 2,$ the covering cannot be hyperelliptic. 

\begin{figure}[htbp] 
   \centering
   \includegraphics[width=2in]{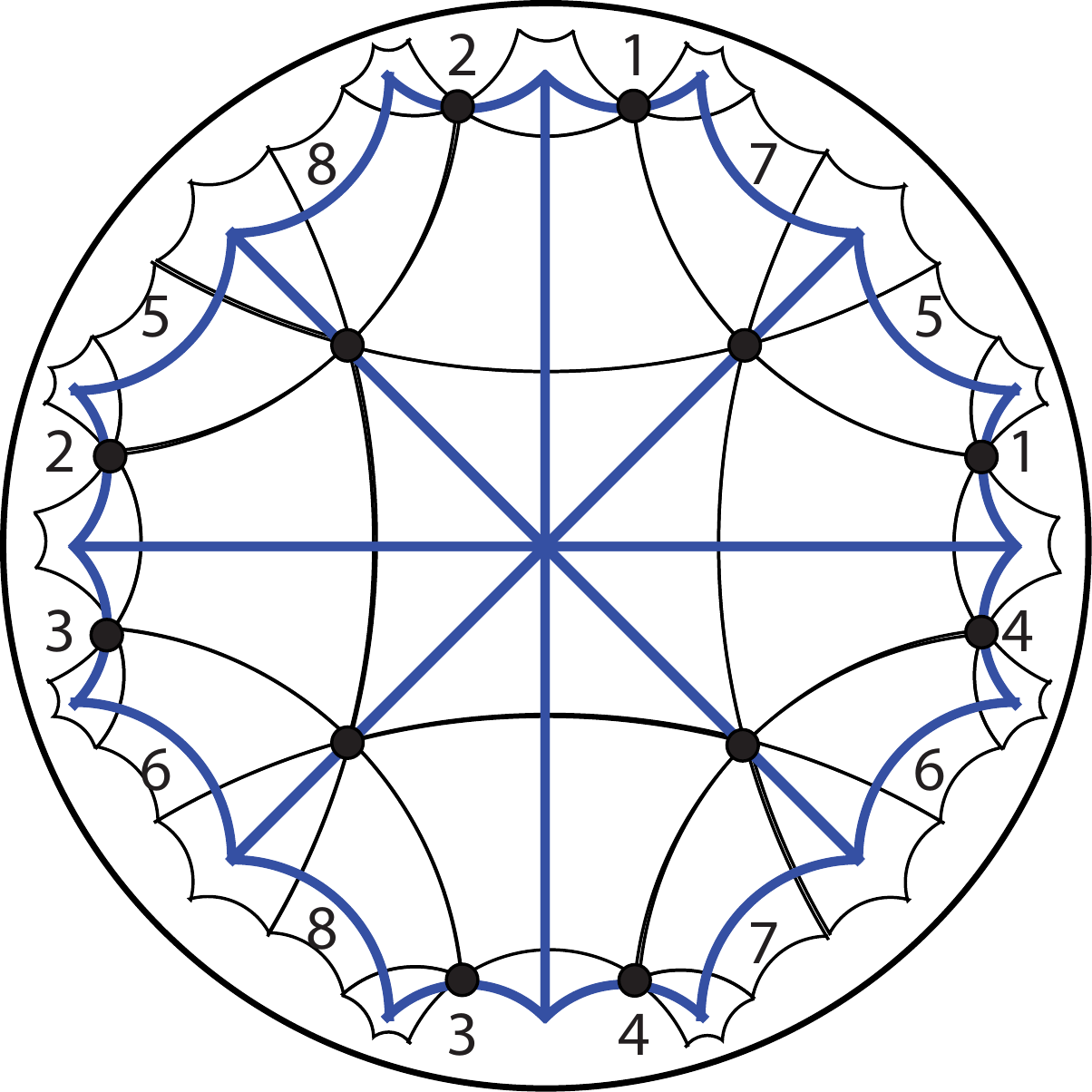} 
	\caption{Another hyperbolic description of the Mucube}
	\label{fig: 464_}
\end{figure}

Note that there are no larger cyclic groups that preserve the hyperbolic geodesics of the surfaces. We have shown the following proposition.

\begin{prop} The compact quotients of the Mucube, Muoctahedron, and Mutetrahedron are biholomorphic to each other. Moreover, the underlying surface has at least two distinct descriptions as cyclically branched coverings over a 4-punctured sphere. The coverings are defined as a sixfold cover defined by branching indices $(1, 3, 5, 3)$ and as a fourfold cover defined by branching indices $(1, 1, 3, 3).$
\end{prop}

Next we locate the Weierstrass points on the underyling surface of the Mucube, Muoctahedra, and Mutetrahedron. We view the surface as a sixfold covering defined by branching indices $(1, 3, 5, 3).$ Then $\textrm{wt}_1 = \textrm{wt}_3 = 3,$ hence $\widetilde{p_1}$ and $\widetilde{p_3}$ are Weierstrass points that are fixed by the hyperelliptic involution. In Figure~\ref{fig: 464}, the $180^{\circ}$-rotation about the center of the tessellation is an involution that fixes $\widetilde{p_1}$ and $\widetilde{p_3}.$ In fact, it fixes all $\widetilde{p_i},$ hence all $\widetilde{p_i}$ are Weierstrass points. As we found eight Weierstrass points, they are the only Weierstrass points. Furthermore, they are marked in Figure~\ref{fig: 464}, Figure~\ref{fig: 644}, and Figure~\ref{fig: 663}, corresponding to vertices on the Mucube, the centers of the hexagons on the Muoctahedron, and the vertices and centers of the hexagons on the Mutetrahedron.

Lastly, using the following theorem and definition, we show that we can immerse the underlying surface of the Mucube, Muoctahedron, and Mucube into $\R^3$ as a minimal surface. \\

\begin{thm*} [\cite{weber2001clay}] Given holomorphic 1-forms $\omega_1, \omega_2,$ and $\omega_3$ such that $\sum \omega_i^2 \equiv  0$ and $\sum |\omega_i|^2 \neq 0,$ we have $z \mapsto \textrm{Re} \int^z \left(\omega_1, \omega_2, \omega_3\right)$ that defines a conformally parametrized minimal surface in $\R^3,$ and every conformal minimal surface parametrization arises this way.
\end{thm*}

\begin{defn*} Given holomorphic 1-forms $\omega_i$ for $i = 1, 2, 3,$ the Weierstrass representation is defined by $$z \mapsto \textrm{Re} \int^z \left(\frac{1}{2} \left(\frac{1}{G} - G\right), \frac{i}{2} \left(\frac{1}{G} + G\right), 1\right) d h$$ where $G := \frac{-\omega_1 + i \omega_2}{\omega_3}$ is the Gauss map.
\end{defn*}

\begin{thm} The underlying conformal structure of Schwarz minimal P-surface is compatible with that of the Mucube and Muoctahedron. That of Schwarz minimal D-surface is compatible with the Mutetrahedron.
\end{thm}

\begin{proof} 
Given the basis of holomorphic 1-forms $\{\omega_i\}$ from the previous computation, define $\varphi_i$ by $$\varphi_1 = \frac{-\omega_1 + \omega_3}{2 \omega_2} d h, \qquad \varphi_2 = i \frac{\omega_1 + \omega_3}{2 \omega_2} d h, \qquad \varphi_3 = d h$$ so that $\varphi_i$ satisfy $\sum \varphi_i^2 \equiv 0$ and $\sum |\varphi_i|^2 \neq 0.$ In other words, the following Weierstrass representation $$\textrm{Re} \int^z \left(\frac{1}{2} \left(\frac{1}{G} - G\right), \frac{i}{2} \left(\frac{1}{G} + G\right), 1\right) d h$$ is a conformal parametrization of a minimal surface in $\R^3$ where $G$ is the Gauss map. Schwarz minimal P-surface has the Mucube and Muoctahedron as the underlying conformal structure and Schwarz minimal D-surface has the Mutetrahedron as the underlying conformal structure. \end{proof}

\subsection{Octa-8}
\label{sec: (1,4,7)}
In this section, we look at a genus four triply periodic polyhedral surface that arises as a regular Archimedean decoration of $\Gamma_{4, 1, 8}'.$ The graph is decorated by an octahedron and eight triangular anti-prisms so we call it the Octa-8 surface. We will show that its compact quotient is a twelvefold cyclically branched covering over a thrice punctured sphere. We will find an explicit basis of holomorphic 1-forms on the surface and show that the surface is conformally equivalent to the genus four Schoen's minimal I-WP surface. \\

The topological construction of the Octa-8 surface is similar to that of the Octa-4 surface. Instead of attaching four Type A octahedra on a Type B octahedron, we attach eight. The fundamental piece of the Octa-8 surface is tiled by 24 triangles where every vertex is twelve-valent. We map a flat triangle to a hyperbolic $\frac{\pi}{6}$-triangle, then tile the hyperbolic disk using Schwarz reflection principle. Figure~\ref{fig: octa8} shows the hyperbolic description along with the identification of edges. Its hyperbolic description displays an order-twelve rotational symmetry which preserves the hyperbolic geodesics. We take the quotient of the hyperbolic fundamental piece by the cyclic group of order twelve, which results in a doubled triangle. We label the center of the tessellation by $p_1,$ thus $d_1 = 1.$ The identification of edges suggests that $\gcd(12, d_3) = 1$ and $(12, d_2) = 4.$ Therefore, the covering is defined by branching indices $(d_1, d_2, d_3)  = (1, 4, 7).$ 

\begin{figure}[htbp] 
   \centering
   \includegraphics[width=2in]{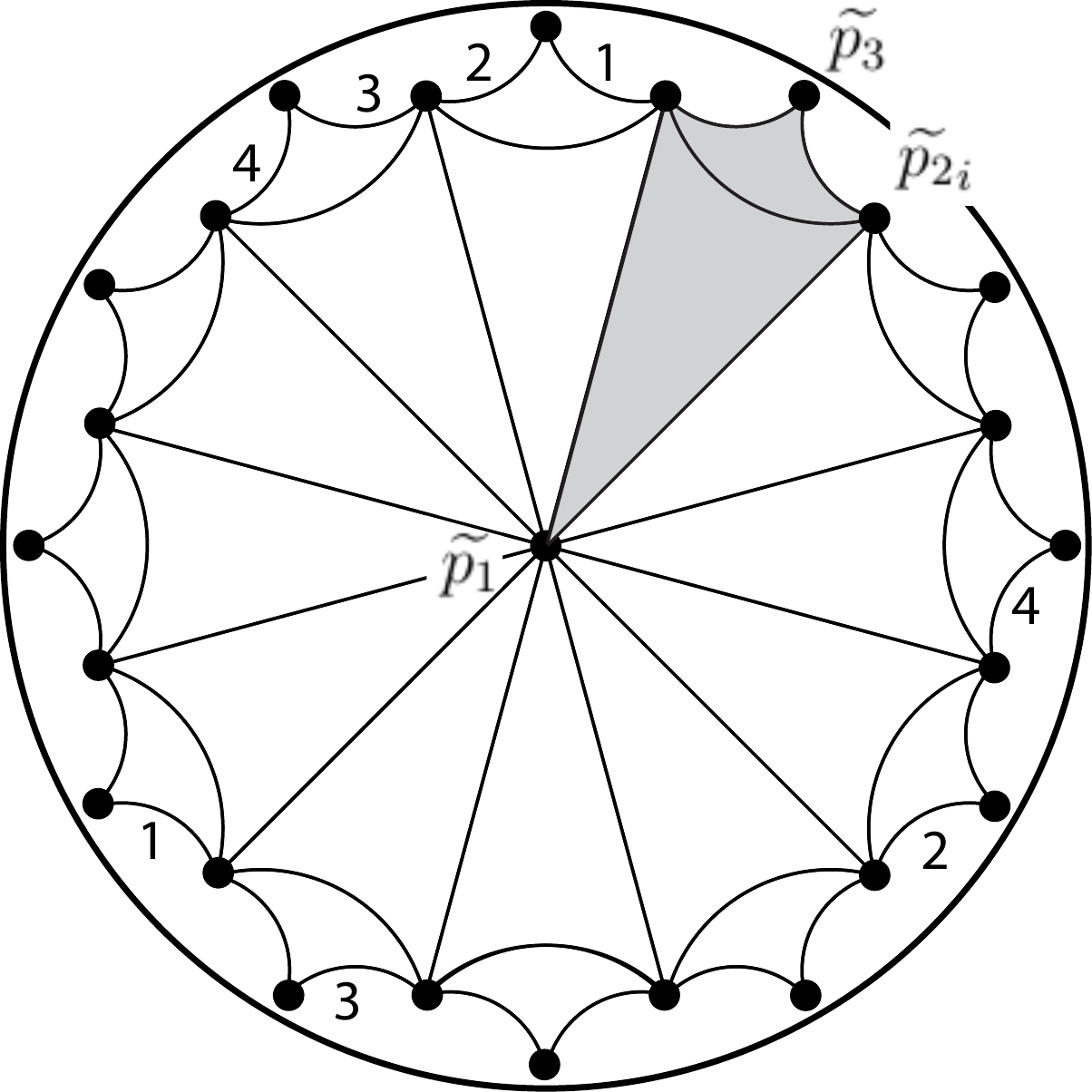} 
	\caption{Hyperbolic description of the Octa-8 surface}
	\label{fig: octa8}
\end{figure}

Given the branching indices, the admissible cone metrics arise from $(1, 4, 7),$ $(2, 8, 2),$ $(4, 4, 4),$ and $(7, 4, 1),$ which yield a basis of holomorphic 1-forms $\{\omega_i\}$ with the following divisors $$\begin{array}{ccccccc}(\omega_1) = & & & & & & 6 \widetilde{p_3}\\
(\omega_2) = & \widetilde{p_1} & + \widetilde{p_2}_1 & + \widetilde{p_2}_2 & + \widetilde{p_2}_3 & + \widetilde{p_2}_4 & + \widetilde{p_3}\\
(\omega_3) = & 3 \widetilde{p_1} & & & & & + 3 \widetilde{p_3}\\
(\omega_4) = & 6 \widetilde{p_1}. & & & & &
\end{array}$$ 

By Weierstrass gap theorem, we have $\textrm{wt}(\widetilde{p_1}) = \textrm{wt}(\widetilde{p_3}) = 4.$ In other words, the surface is not hyperelliptic. Since the Octa-8 surface is a non-hyperelliptic genus four surface, its canonical curve embeds in $\mathbb{P}^3.$ We refer to the following proposition (Proposition 2.6, Chapter 7 from \cite{miranda1995algebraic}) that shows that all genus four non-hyperelliptic curves are in the complete intersection of a cubic and a quadric polynomial. 

\begin{prop*} Let $X$ be a non-hyperelliptic algebraic curve of $g = 4.$ Then $X$ embeds in $\mathbb{P}^3$ as a smooth curve of degree six defined by the vanishing of a quadric and a cubic polynomial.
\end{prop*}

Given the basis of holomorphic 1-forms on the Octa-8 surface that arise from admissible cone metrics, we have $\omega_1 \omega_4 = \omega_3^2.$ Furthermore, we define meromorphic 1-forms $f, g,$ and $h$ by $$\begin{array}{rl}(f) & := \left(\frac{\omega_1}{\omega_2}\right) = - \widetilde{p_1} - \widetilde{p_2}_1 - \widetilde{p_2}_2 - \widetilde{p_2}_3 - \widetilde{p_2}_4 + 5 \widetilde{p_3},\\
(g) & := \left(\frac{\omega_3}{\omega_2}\right) = 2 \widetilde{p_1} - \widetilde{p_2}_1 - \widetilde{p_2}_2 - \widetilde{p_2}_3 - \widetilde{p_2}_4 + 2 \widetilde{p_3},\\
(h) & := \left(\frac{\omega_4}{\omega_2}\right) = 5 \widetilde{p_1} - \widetilde{p_2}_1 - \widetilde{p_2}_2 - \widetilde{p_2}_3 - \widetilde{p_2}_4 - \widetilde{p_3}.\end{array}$$ Then $$\left(f^2 g\right) = -3 \widetilde{p_2}_1 - 3 \widetilde{p_2}_2 - 3 \widetilde{p_2}_3 - 3 \widetilde{p_2}_4 + 12 \widetilde{p_3} \quad \textrm{and} \quad \left(h^2 g\right) = 12 \widetilde{p_1} - 3 \widetilde{p_2}_1 - 3 \widetilde{p_2}_2 - 3 \widetilde{p_2}_3 - 3 \widetilde{p_2}_4.$$ After appropriate scaling, we have $f^2 g - 1 = h^2 g$ which yields a cubic $\omega_1^2 \omega_3 - \omega_2^3 = \omega_3 \omega_4^2.$\\

On the other hand, we can define meromorphic functions $u$ and $v$ such that $(u) : = \left(\frac{\omega_2}{\omega_4}\right)$ and $(v) := \left(\frac{\omega_3}{\omega_4}\right).$ Then after scaling, we have $\frac{v^5}{u^3} - 1 = \frac{v}{u^3},$ or in other words, $v^5 - v = u^3$ or $\omega_3^5 - \omega_3 \omega_4^4 = \omega_2^3 \omega_3^2.$ The equation $u^3 = v^5 - v$ suggests that the abstract surface is a threefold branched cover over a sphere branched at $h = 0,$  $\infty,$ and the fourth roots of unity. According to the classification of cyclically branched coverings over spheres (Appendix~\ref{wronski}), the branching indices of the threefold covering must be either $(1, 1, 1, 1, 1, 1)$ or $(1, 1, 1, 2, 2, 2).$ We claim that it cannot be the latter by computing the weight distributions. On the Octa-8 surface, we have $\textrm{wt}(\widetilde{p_1}) = \textrm{wt}(\widetilde{p_3}) = 4.$ Computation of the Wronskian (Corollary \ref{cor: wronski}) tells us that $\textrm{wt}(\widetilde{p_2}_i) = 4$ for $i = 1, 2, 3, 4,$ and that there are 36 points of $\textrm{wt} = 1$ each. All six branched points on the $(1, 1, 1, 1, 1, 1)$ covering also have $\textrm{wt}_i = 4,$ and the surface has 36 points of $\textrm{wt} = 1$ each. On the other hand, the six branched points on the covering defined by $(1, 1, 1, 2, 2, 2)$ have $\textrm{wt}_i = 2,$ hence this surface cannot have the same conformal structure as the Octa-8 surface. \\

In the following theorem, we prove that the threefold cyclic cover defined by $(1, 1, 1, 1, 1, 1)$ is the order three Gauss map on a genus four minimal surface, namely, Schoen's minimal I-WP surface. The fundamental piece of Schoen's minimal I-WP surface can be viewed as a central chamber with handles toward the vertices of a cube. We will show that the underlying conformal structure of Schoen's minimal I-WP surface is compatible with that of the Octa-8 surface.

\begin{thm}\label{thm: iwp} Given the cone metrics on the compact quotient of the Octa-8 surface, the Weierstrass representation defines a conformally parametrized genus four minimal surface, namely Schoen's minimal I-WP surface.
\end{thm}

\begin{proof}
Given the basis of holomorphic 1-forms that arise from the admissible cone metrics, we define $\varphi_i$ as follows: $$\varphi_1 = \frac{- \omega_1 + \omega_4}{2 \omega_3} d h, \qquad \varphi_2 = i \frac{\omega_1 + \omega_4}{2 \omega_3} d h, \qquad \varphi_3 = d h.$$ Then we have $\sum \varphi_i^2 = 0$ and $\sum |\varphi_i|^2 \neq 0.$ Then the following Weierstrass representation $$\textrm{Re} \int^z \left(\frac{1}{2} \left(\frac{1}{G} - G\right), \frac{i}{2} \left(\frac{1}{G} + G\right), 1\right) d h$$ defines a conformally parametrized minimal surface in $\R^3$ where the Gauss map $G = \frac{\omega_1}{\omega_3}$ has a zero and a pole of order three. 
\end{proof}

This is quite a remarkable result since the order-twelve rotational symmetry is not visible on the polyhedral Octa-8 surface or the I-WP surface in $\R^3.$

\subsection{Truncated Octa-8}
\label{sec: (1,2,4,3)}
In this section, we look at another genus four triply periodic polyhedral surface that also arises as a regular Archimedean decoration of $\Gamma_{4, 1, 8}'.$ The 0-simplex is replaced by a truncated octahedron and the 1-simplices are replaced by hexagonal prisms. We call it the Truncated Octa-8 surface. We will show that its compact quotient is a fivefold cyclically branched covering over a 4-punctured sphere. We will find an explicit basis of holomorphic 1-forms and show that the quotient surface has a conformal realization as a well-known Euclidean uniform polyhedron, the dodecadodecahedron, also known as Kepler's small stellated dodecahedron. It is shown that this surface is also equivalent to Bring's curve (\cite{weber2005kepler}). We will also show that the underlying conformal structure is not compatible with any genus four minimal surface.\\

To prove the embeddedness of the polyhedral surface in Euclidean space, we again use the fact that cubes tile space. First we place a truncated octahedron in a cube so that the square faces of the truncated octahedron are tangent to the faces of the cube. Tiling space with such cubes results in the Muoctahedron. However, now we uniformly scale the trunctated octahedra in each cube so that the distance between two parallel hexagonal faces on neighboring truncated octahedra is equal to the edge length of the solids where we attach hexagonal prisms to connect the neighboring truncated octahedra. Figure~\ref{fig: truncated octa8_poly} shows the fundamental piece of the boundary surface. As there are five squares at each vertex of the boundary surface, we map the Euclidean squares to hyperbolic $\frac{2 \pi}{5}$-squares and obtain a hyperbolic tessellation (Figure~\ref{fig: truncated octa8_hyp}).

\begin{figure}[htbp] 
   \centering
   \includegraphics[width=2in]{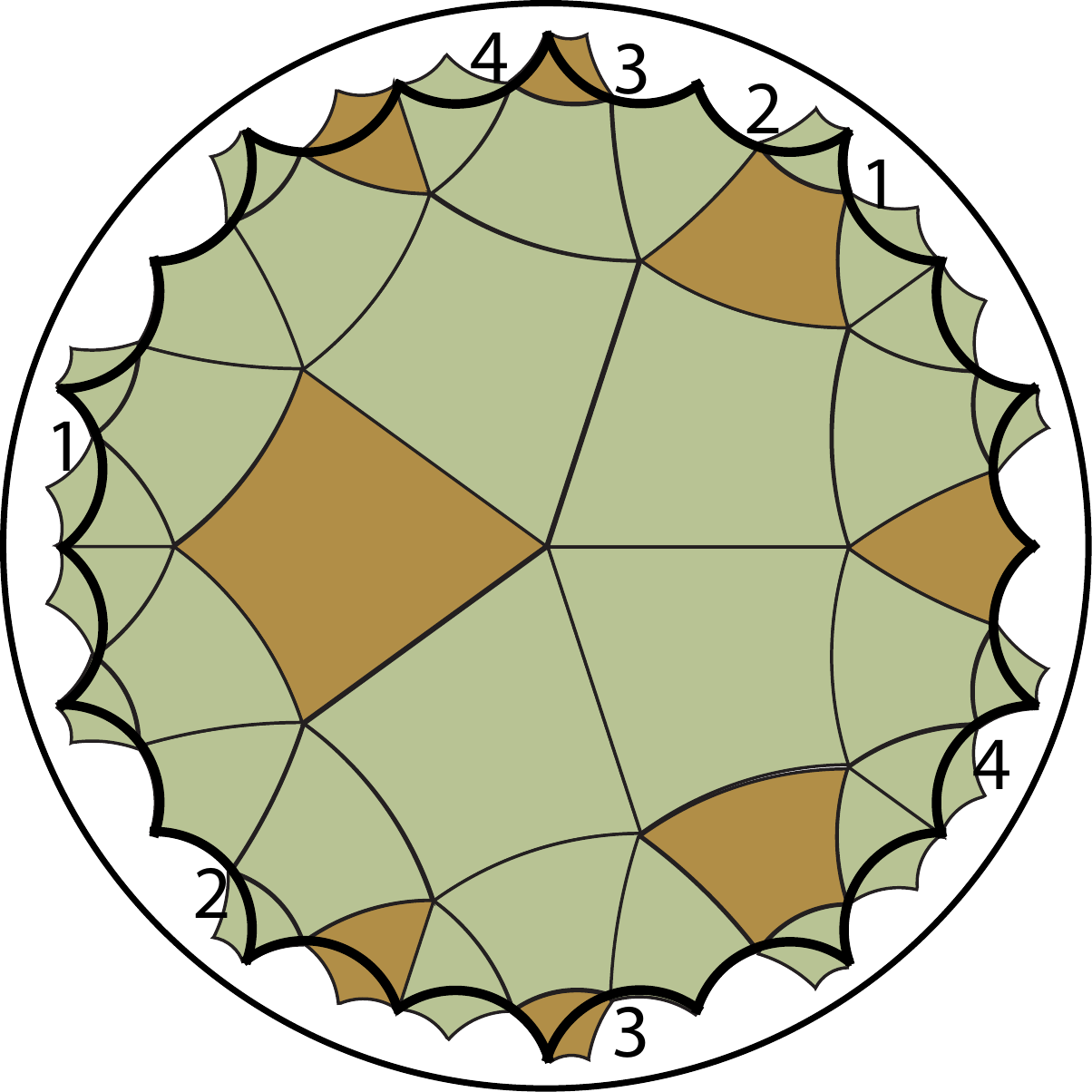} 
	\caption{Hyperbolic description of the Truncated Octa-8 surface}
	\label{fig: truncated octa8_hyp}
\end{figure}

We trace the polyhedral geodesics to find the identification of edges in the hyperbolic setting where all of them are closed after going through six squares. Figure~\ref{fig: truncated octa8_hyp} exhibits an order-five rotational symmetry that preserves the hyperbolic geodesics. We take the quotient of the hyperbolic fundamental domain by the cyclic group of order five which results in a sphere. In fact, the surface is a fivefold cyclically branched covering over a 4-punctured sphere with branching indices $(1, 2, 4, 3).$ This is shown in Lemma 2.1, \cite{weber2005kepler}. Weber also shows that this quotient surface has a conformal realization as a well-known compact polyhedron, the dodecadodecahedron, also known as Kepler's small stellated dodecahedron. This surface has an algebraic description as Bring's curve.\\

Lastly, we show that the conformal structure induced from the cone metric on the abstract quotient of the Truncated Octa-8 surface is not compatible with any minimal surface in $\R^3.$

\begin{thm} \label{thm: truncated octa8} The conformal structure induced from the cone metric on the abstract quotient of the Truncated Octa-8 surface is not compatible with any triply periodic minimal surface in $\R^3.$
\end{thm}

\begin{proof} We refer to earlier discussions in Section~\ref{sec: three mu} and Section~\ref{sec: (1,4,7)}. To parametrize a minimal surface in $\R^3,$ we need a rank-three quadric. However, we will show the existence of a quadric $Q(\omega_1, \omega_2, \omega_3, \omega_4) = 0$ of rank four, then use B{\'e}zout's theorem to show that we cannot have a rank-three quadric of our desire. B{\'e}zout's theorem says if $X$ lies on two distinct quadrics $Q$ and $Q',$ then $Q \cap Q'$ is a curve of degree 4. \\

Given the branching indices $(1, 2, 4, 3)$ that define the surface, the admissible cone metrics induce a basis of holomorphic 1-forms with the following divisors:

$$\begin{array}{ccccc}(\omega_1) = && \widetilde{p_2} & + 3 \widetilde{p_3} & + 2 \widetilde{p_4} \\
(\omega_2) =  & \widetilde{p_1} & + 3 \widetilde{p_2} & + 2 \widetilde{p_3} & \\
(\omega_3) = & 2 \widetilde{p_1} && + \widetilde{p_3} & + 3 \widetilde{p_4} \\
(\omega_4) = & 3 \widetilde{p_1} & + 2 \widetilde{p_2} && + \widetilde{p_4}. \end{array}$$

Then $Q(\omega_1, \omega_2, \omega_3, \omega_4) = \omega_1 \omega_4 - \omega_2 \omega_3 = 0$ is a rank-four quadric, hence by B{\'e}zout's theorem, there can be no quadric of rank three to parametrize a minimal surface in $\R^3.$ 
\end{proof}

\subsection{Octa-4}
\label{sec: (1,2,5)}
In \cite{lee2017triply} is shown that there is no minimal surface that has the same underlying structure as the Octa-4 surface. In this section, given a basis of 1-forms and a homology basis, we will mimic the Weierstrass representation and find a smooth surface of the same conformal type. Given $\{\omega_i\}$ that is induced from the admissible cone metrics, we pick a homology basis and calculate the period matrix $(\Pi) = (\int_{\gamma_k} \omega_j).$ Then we will find coefficients $a_{ij}$ that satisfy \begin{equation} \label{eqn: lattice} \sum_j \begin{pmatrix}a_{1 j}\\
a_{2 j}\\
a_{3 j}\end{pmatrix} \int_{\gamma_k} \omega_j = \textrm{Re} \int_{\gamma_k} \begin{pmatrix}\varphi_1\\
\varphi_2\\
\varphi_3\end{pmatrix} = (P) \end{equation} where $P$ is a $3 \times 6$ matrix whose column vectors form a lattice in Euclidean space. The lattice is a result of the periods of cycles on the polyhedral surface in $\R^3.$ 

\begin{thm} There exists an explicit parametrization of a smooth surface that has the same conformal type as the Octa-4 surface.
\end{thm}

\begin{proof} We will choose a homology basis that makes computation of $P$ easier. We let $\alpha_1, \alpha_2, \alpha_3$ lie in three distinct anti-prisms. Then we pick $\beta_k$ so that the intersection numbers are $\iota(\alpha_i, \beta_j) = \delta_{i j}.$ In other words, $\{\alpha_k, \beta_k\}_{k = 1}^3$ results in a (canonical) homology basis. In Figure~\ref{fig: octa4 basis}, $\alpha_k$ are marked as bold lines and $\beta_k$ are marked as dotted lines. 

\begin{figure}[htbp] 
   \centering
   \includegraphics[width=4in]{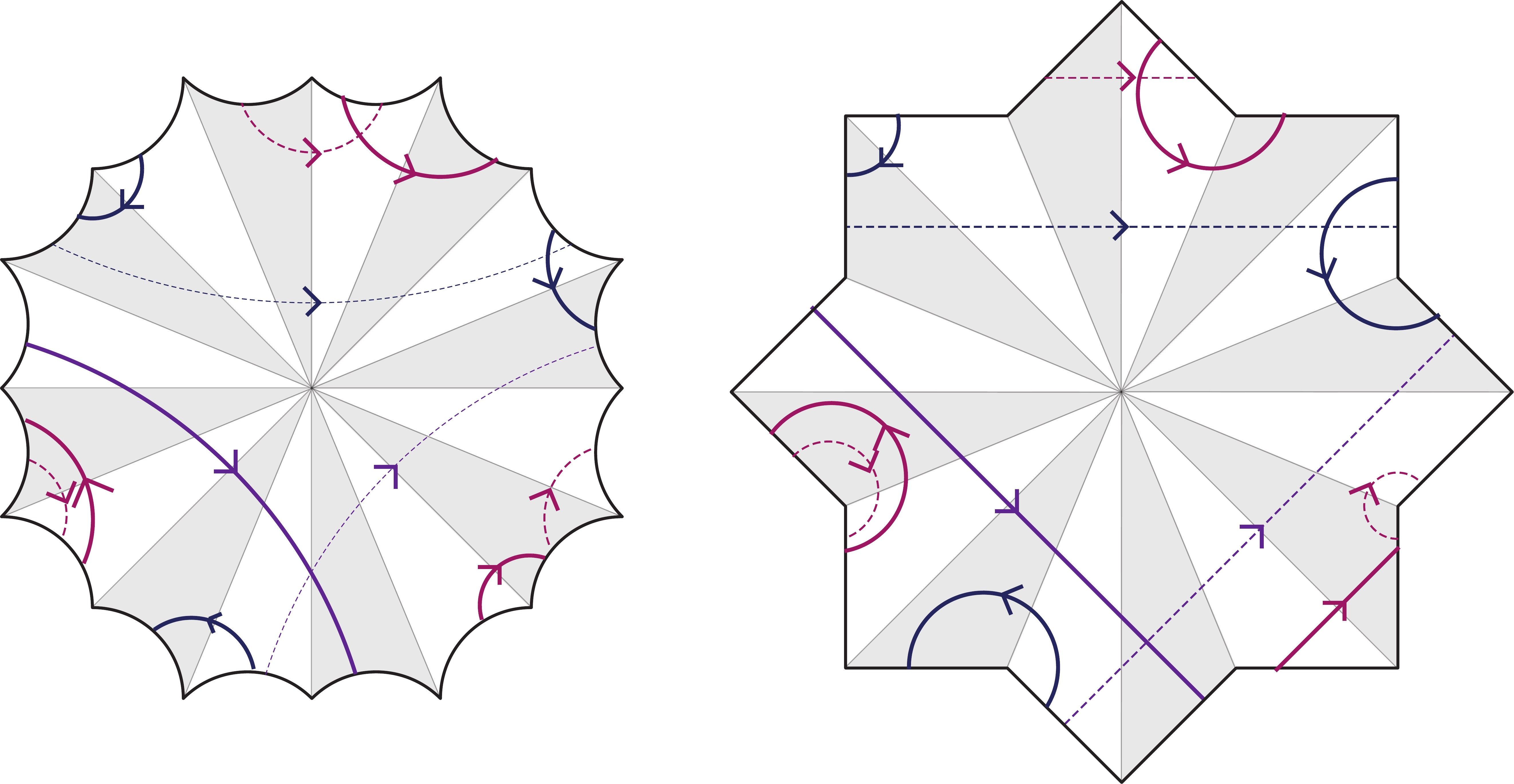} 
   \caption{Canonical homology basis on the Octa-4}
   \label{fig: octa4 basis}
\end{figure}

The column vectors of $P$ form a lattice in $\R^3.$ Moreover, due to our choice of cycles, only $\beta_k$ have nonzero periods in $\R^3,$ hence $$(P) = \begin{pmatrix}0 & 0 & 0 & 2 & 0 & 2\\
0 & 0 & 0 & 0 & -2 & -2\\
0 & 0 & 0 & 2 & 2 & 0\end{pmatrix}.$$ 

Since we already have an explicit basis of holomorphic 1-forms given by the admissible cone metrics (Example \ref{exmp: 125_}), we compute the period matrix $(\Pi)$ by $(\int_{\alpha_i} \omega_j, \int_{\beta_i} \omega_j)_{j = 1}^3,$ then $$(\Pi) = \begin{pmatrix}1 - i & \frac{-1 - i}{1 + \sqrt{2}} & \frac{1 + i}{1 + \sqrt{2}} & 1 + i & \sqrt{2} & 2 - \sqrt{2}\\
-2 i & 2 i & 2 i & 2 i & -2 & -2\\
-1 - i & (1 + \sqrt{2}) (1 - i) & (1 + \sqrt{2}) (-1 + i) & 1 - i & \sqrt{2} i & -(2 + \sqrt{2}) i \end{pmatrix}.$$ We write $(\Pi) = (A | B),$ then the Jacobian is defined as $$A^{-1} B = \begin{pmatrix}i & \frac{1 + i}{2} & \frac{1 + i}{2}\\
\frac{1 + i}{2} & i & \frac{1 + i}{2}\\
\frac{1 + i}{2} & \frac{1 + i}{2} & i\end{pmatrix}.$$ Then the coefficients $$(a_{i j}) = \begin{pmatrix} \frac{2 + \sqrt{2} - (4 + 3 \sqrt{2}) i}{4 + 2 \sqrt{2}} & 0 & \frac{1 - \sqrt{2} + i}{2}\\
0 & 1 & 0\\
\frac{1 + \sqrt{2} - i}{2} & 0 & \frac{1 + (1 - \sqrt{2}) i}{2}\end{pmatrix}$$ satisfy (\ref{eqn: lattice}), and $z \mapsto \textrm{Re} \int^z \left(\varphi_1, \varphi_2, \varphi_3\right)^t$ results in a parametrization of a harmonic surface where $\varphi_i = \sum_j a_{ij} \omega_j.$ The following figure shows the parametrization of the surface in $\R^3.$

\begin{figure}[htbp] 
	\centering
	\includegraphics[width=4in]{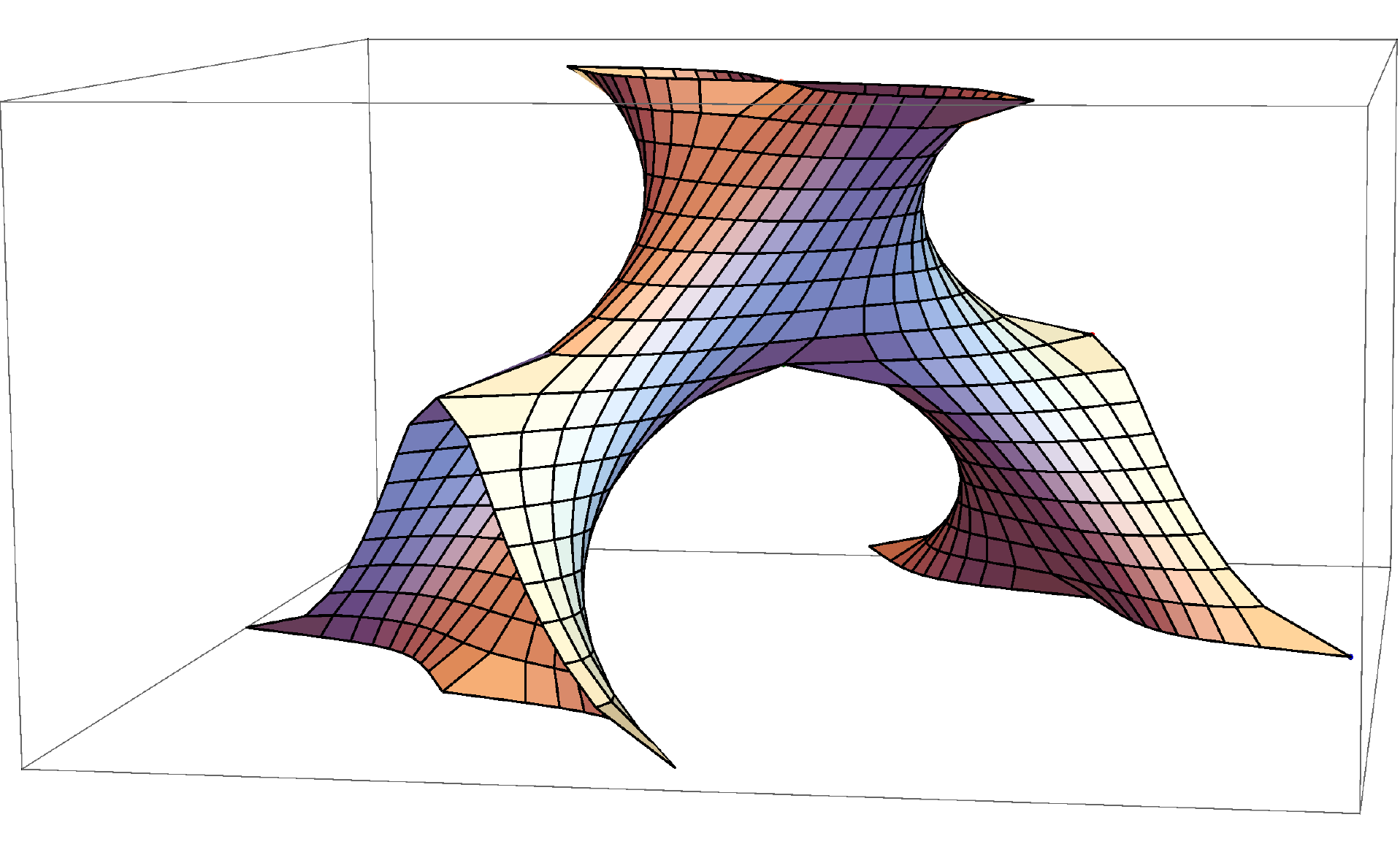}
	\caption{Fundamental piece of the harmonic parametrization of the Octa-4 surface}
\end{figure}
\end{proof}

\section{Figure Credits}
\begin{itemize}
\item Figure~\ref{octa4-5}\\*
``Construction of $\Pi,$'' \cite{lee2017triply}

\item Figure~\ref{hyperbolic_tiling}\\*
``Hyperbolic realization of fundamental piece,'' \cite{lee2017triply}

\item Figure~\ref{flat}\\*
``Flat structure of the fundamental piece,'' \cite{lee2017triply}
\end{itemize}
\newpage

\appendix
\section{All cyclically branched covers up to genus five}
\label{genera}

In this section, we list all cyclically branched coverings over punctured spheres of lower genera (Proposition~\ref{prop: g(X)} and Theorem~\ref{thm: genera})

\begin{mmaCell}[moredefined={Genus, Normalize1, Normalize2, \
dicovers},morepattern={d_, di_, d, di, \#, g_, \
g},morefunctionlocal={i, k}]{Input}
Genus[d_, di_] := 1-d+1/2 Sum[d-GCD[d, di[[i]]], \{i,1,Length[di]\}];
Normalize1[d_, di_] := 
	First[
	 Sort[Flatten[Map[Table[RotateLeft[Mod[#*di,d],k],\{k,0,Length[\ di]-1\}]&,
			Select[Range[1,d-1],GCD[#,d]==1 &] ],1]]];
Normalize2[\{d_, di_,g_\}] :=
	\{d,First[Sort[\{Normalize1[d,di],Normalize1[d,Reverse[di]]\}]],g\};
dicovers[di_]:=Module[\{\mmaLoc{d}\},
	\mmaLoc{d}=Plus@@di;
	Map[\{#,di,Genus[#,di]\}&,Select[Divisors[\mmaLoc{d}],Max[di]<#&]]
	]
\end{mmaCell}

For example, the following codes show the cyclically branched coverings over thrice punctured spheres of lower genera.

\vspace{0.4cm}

\begin{mmaCell}[moredefined={dicovers, \ Normalize2},morepattern={\#}]{Input}
Flatten[dicovers/@Union[Sort/@Tuples[Range[336],3]],1];
Select[
Select[
Union[Map[Normalize2,
\end{mmaCell}

\begin{mmaCell}{Output}
\{7,\{1,1,5\},3\},\{7,\{1,2,4\},3\},\{8,\{1,1,6\},3\},\{8,\{1,2,5\},3\},\{9,\{1,1,7\},4\},
\{9,\{1,2,6\},3\},\{10,\{1,1,8\},4\},\{10,\{1,2,7\},4\},\{11,\{1,1,9\},5\},\{11,\{1,2,8\},5\},
\{12,\{1,1,10\},5\},\{12,\{1,2,9\},4\},\{12,\{1,3,8\},3\},\{12,\{1,4,7\},4\},\{12,\{1,5,6\},3\},
\{14,\{1,6,7\},3\},\{15,\{1,4,10\},5\},\{15,\{1,5,9\},4\},\{16,\{1,7,8\},4\},\{18,\{1,8,9\},4\},
\{20,\{1,9,10\},5\},\{22,\{1,10,11\},5\}\}
\end{mmaCell}

\vspace{0.4cm}
We sort the coverings by their genus.

\begin{table}[htbp]
\centering
\begin{tabular}[h]{|c|c||c|c|}
\hline
$d$ & $(d_i)$ & $d$ & $(d_i)$ \\ \hline
7 & (1, 1, 5) & 4 & (1, 1, 1, 1) \\ \hline
7 & (1, 2, 4) & 4 & (1, 1, 3, 3) \\ \hline
8 & (1, 1, 6) & 6 & (1, 3, 3, 5) \\ \hline
8 & (1, 2, 5) & 6 & (1, 3, 4, 4) \\ \hline 
9 & (1, 2, 6) & 3 & (1, 1, 1, 1, 2) \\ \hline 
12 & (1, 3, 8) & 4 & (1, 1, 2, 2, 2) \\ \hline
12 & (1, 5, 6) & 2 & (1, 1, 1, 1, 1, 1, 1, 1) \\ \hline
14 & (1, 6, 7) && \\ \hline
\end{tabular}
	\label{tab: genus three cyc br covers}
	\caption{Genus three cyclically branched coverings over punctured spheres}
\end{table}

\begin{table}[htbp]
\centering
\begin{tabular}[h]{|c|c||c|c|}
\hline
$d$ & $(d_i)$ & $d$ & $(d_i)$ \\ \hline
9 & (1, 1, 7) & 5 & (1, 1, 1, 2) \\ \hline
10 & (1, 1, 8) & 5 & (1, 1, 4, 4) \\ \hline
10 & (1, 2, 7) & 5 & (1, 2, 3, 4) \\ \hline 
12 & (1, 2, 9) & 6 & (1, 1, 2, 2) \\ \hline
12 & (1, 4, 7) & 6 & (1, 2, 4, 5) \\ \hline
15 & (1, 5, 9) & 8 & (1, 4, 4, 7) \\ \hline
16 & (1, 7, 8) & 10 & (2, 5, 5, 8) \\ \hline
18 & (1, 8, 9) & 3 & (1, 1, 1, 1, 1, 1) \\ \hline
4 & (1, 1, 1, 2, 3) & 3 & (1, 1, 1, 2, 2, 2) \\ \hline
6 & (1, 2, 3, 3, 3) & 4 & (1, 2, 2, 2, 2, 3) \\ \hline
6 & (2, 2, 2, 3, 3) & 2 & (1, 1, 1, 1, 1, 1, 1, 1, 1, 1) \\ \hline
\end{tabular}
	\label{tab: genus four cyc br covers}
	\caption{Genus four cyclically branched coverings over punctured spheres}
\end{table}

\begin{table}[htbp]
\centering
\begin{tabular}[h]{|c|c||c|c|}
\hline
$d$ & $(d_i)$ & $d$ & $(d_i)$ \\ \hline11 & (1, 1, 9) & 6 & (1, 1, 3, 3, 4) \\ \hline
11 & (1, 2, 8) & 6 & (1, 2, 2, 3, 4) \\ \hline
12 & (1, 1, 10) & 4 & (1, 1, 1, 1, 2, 2) \\ \hline
15 & (1, 4, 10) & 4 & (1, 1, 2, 2, 3, 3) \\ \hline
20 & (1, 9, 10) & 6 & (2, 3, 3, 3, 3, 4) \\ \hline
22 & (1, 10, 11) & 3 & (1, 1, 1, 1, 1, 2, 2) \\ \hline
6 & (1, 1, 5, 5) & 4 & (1, 1, 2, 2, 2, 2, 2) \\ \hline
8 & (1, 1, 2, 4) & 2 & (1, 1, 1, 1, 1, 1, 1, 1, 1, 1, 1, 1) \\ \hline
10 & (1, 5, 5, 9) && \\ \hline
\end{tabular}
	\label{tab: genus five cyc br covers}
	\caption{Genus five cyclically branched coverings over punctured spheres}
\end{table}

\newpage

\section{Wronski computation for the Octa-4 surface}
\label{wronski}
In this section, we show Mathematica codes to compute the Wronksi metric of the underlying surface of the Octa-4 surface.

\begin{mmaCell}[moredefined={admiss, mplus, copies, wronski, \
pi},morepattern={d_, di_, di, d, \#, n_, l_, l, obj_, obj, pi_, aik_, \
aik},morefunctionlocal={a, i, k},morelocal={n, g, gk, gklogdiff, gdiff, raw, bi, rawd}]{Input}
  admiss[d_,di_]:=Select[Table[Mod[a di,d],\{a,1,d-1\}],
  Plus@@#==d && Times@@# !=0&]
  mplus[n_,l_]:=copies[l,Length[l]]+\mmaPat{n} IdentityMatrix[Length[l]]
  copies[obj_,n_]:=Module[\{\mmaLoc{i}\},Table[obj,\{i,1,\mmaPat{n}\}]]
  wronski[d_,di_,pi_,aik_]:=
  Module[\{n,g,gk,gklogdiff,gdiff,raw,bi, rawd\},
	n=Length[di];
	g=d(n-2)/2+1-Sum[GCD[di[[i]],d],\{i,1,n\}]/2;
  	gk=Table[Product[(x-pi[[i]])^aik[[k,i]],\{i,1,n\}],
		\{k,1,g\}];
  	gklogdiff=Map[Cancel[D[#,x]/#]&,gk];
  	gdiff=Table[NestList[Cancel[1/d gklogdiff[[i]]#+D[#,x]]&,
		1,g-1],\{i,1,g\}];
  	raw=PowerExpand[(Times@@gk)^(1/d)]Factor[Det[gdiff]];
  	rawd=D[raw,x]/raw;
  	bi=Map[Cancel[(x-#)rawd]/.x->#&,pi];
  	\{Cancel[raw /Product[(x-pi[[i]])^bi[[i]],\{i,1,n\}]],
  		Table[d/GCD[d,di[[i]]](g (g-1)/2+bi[[i]])-g (g+1)/2,
  		\{i,1,n\}]\}]
\end{mmaCell}

\begin{mmaCell}[moredefined={pi}]{Input}
  pi=\{0,1,-1\};
\end{mmaCell}

\begin{mmaCell}[moredefined={admiss}]{Input}
  admiss[8,\{1,2,5\}]
\end{mmaCell}

\begin{mmaCell}{Output}
  \{\{1,2,5\},\{2,4,2\},\{5,2,1\}\}
\end{mmaCell}

\begin{mmaCell}[moredefined={wronski, pi}]{Input}
  wronski[8,\{1,2,5\},pi,
\end{mmaCell}

\begin{mmaCell}{Output}
  \{\mmaFrac{3}{128} \mmaSup{(1+3 x)}{2},\{2,2,2\}\}
\end{mmaCell}

\end{document}